\documentclass[a4paper,11pt,oneside,reqno]{amsart}
\allowdisplaybreaks
\usepackage[utf8]{inputenc}
\usepackage{tikz}
\usetikzlibrary{arrows.meta,calc}
\usepackage{amssymb}
\usepackage{esint}
\usepackage{amsmath}
\usepackage{amsfonts}
\usepackage{amsthm}
\usepackage{bm}
\usepackage{mathrsfs}
\usepackage{mathtools}
\usepackage{comment}
\mathtoolsset{showonlyrefs=true}
\usepackage{appendix}
\usepackage{enumerate}
\usepackage[a4paper, total={6.6in, 9.7in}]{geometry}
\usepackage[stretch=30,shrink=30]{microtype}
\usepackage{xcolor}
\definecolor{antiquefuchsia}{rgb}{0.57, 0.36, 0.51}
\definecolor{azure}{rgb}{0.0, 0.5, 1.0}

\usepackage[backref=page, colorlinks = true, linkcolor = black, citecolor = antiquefuchsia]{hyperref}

\renewcommand*{\backref}[1]{}
\renewcommand*{\backrefalt}[4]{%
    \ifcase #1 (Not cited.)%
    \or        (Cited on page~#2.)%
    \else      (Cited on pages~#2.)%
    \fi}

\makeatletter
\def\th@plain{%
	\thm@notefont{}
	\itshape 
}
\def\th@definition{%
	\thm@notefont{}
	\normalfont 
}
\makeatother

\numberwithin{equation}{section}

\newtheorem{theorem}{Theorem}[section]
\newtheorem{lemma}[theorem]{Lemma}
\newtheorem{proposition}[theorem]{Proposition}
\newtheorem{corollary}[theorem]{Corollary}

\newtheorem*{question}{Question}

\theoremstyle{definition}

\newtheorem{remark}[theorem]{Remark}

\renewcommand{\b}{\mathbb{B}}

\renewcommand{\r}{\mathbb{R}}
\let\textcedilla\c
\renewcommand{\c}{\mathbb{C}}

\DeclareMathOperator{\dvol}{dvol}

\DeclareMathOperator{\vol}{vol}

\renewcommand\labelitemi{$\vcenter{\hbox{\tiny$\bullet$}}$}

\title[Negatively Curved Minimal Surfaces in the Round Four-Sphere]{Negatively Curved Minimal Surfaces in the Round Four-Sphere}

\author[Riccardo Caniato]{Riccardo Caniato}
\address{Warwick Mathematics Institute, Zeeman Building, University of Warwick, Coventry, CV4 7AL, United Kingdom}
\email{riccardo.caniato@warwick.ac.uk}
\date{\today} 

\begin{document}
\begin{abstract}
    We construct smooth closed embedded minimal surfaces with everywhere negative Gaussian curvature in the unit round four-sphere. More precisely, for every sufficiently large integer $m$, our construction gives a minimal surface $\Sigma_m$ 
    of genus $m^2+1$ such that $\operatorname{Area}(\Sigma_m)\to 4\pi^2$ as $m\to+\infty$. This result yields a complete solution to Problem 101 in Yau's 1982 list of open problems.
\end{abstract}
\maketitle
\tableofcontents

\section{Introduction}
\label{sec:s4-introduction}
\subsection{Background and main results}
Understanding how the curvature of an ambient manifold constrains the intrinsic geometry of its minimal surfaces is a classical and fascinating problem in Riemannian geometry. By the Gauss equation, the Gaussian curvature of a minimal surface equals the ambient sectional curvature along its tangent plane minus one half of the squared norm of its second fundamental form. Nonpositive ambient curvature therefore forces nonpositive Gaussian curvature. On the other hand, positive ambient curvature presents a more subtle question, since the sign of the Gaussian curvature of a minimal surface depends on the balance between the ambient curvature and the second fundamental form. Round spheres, as the model spaces of constant positive sectional curvature, provide a natural setting in which to study this interaction.

Let $\Sigma$ be a smooth Riemannian surface and let $\iota:\Sigma\to\mathbb{S}^n$ be a smooth minimal immersion, where $g_{\mathbb{S}^n}$ denotes the round metric on $\mathbb{S}^n$ of sectional curvature one. Write $g_\iota:=\iota^*g_{\mathbb{S}^n}$ for its induced metric, $K_\iota$ for the Gaussian curvature of $g_\iota$, and $A_\iota$ for its normal-bundle-valued second fundamental form in $\mathbb{S}^n$. The Gauss equation for the immersion $\iota$ in the ambient $\mathbb{S}^n$ reads
\begin{align*}
    K_\iota=1-\frac{|A_\iota|_{g_\iota}^2}{2},
\end{align*}
where the norm uses $g_\iota$ on the tangent bundle and the restriction of $g_{\mathbb{S}^n}$ on the normal bundle.
In his 1982 problem list, Yau asked whether the contribution of the second fundamental form can dominate the positive ambient curvature at every point of a closed minimal surface, thus giving $K_\iota<0$ everywhere \cite[Problem~101]{YauProblemSection}. Despite the extensive existence theory for minimal surfaces in spheres, this question remained open for more than four decades.

Two classical nonexistence results reveal substantial restrictions on the ambient dimension and the induced metric. Lawson \cite[Proposition~1.5]{LawsonCompleteMinimalSurfaces} proved that no closed minimal surface in the round $\mathbb S^3$ can have negative Gaussian curvature everywhere. Hence, negatively curved examples must be sought in higher dimensional round spheres. However, Bryant \cite[Theorem~2.3]{BryantConstantCurvature} proved that a surface of constant negative Gaussian curvature admits no minimal isometric immersion into a round sphere of any finite dimension. Thus, any solution to Yau's problem must have nonconstant Gaussian curvature.

Twistor geometry in even-dimensional spheres provides another natural class in which to seek negatively curved minimal surfaces. Bryant \cite{BryantMinimalS4} proved that every compact Riemann surface admits a linearly full conformal superminimal immersion into $\mathbb{S}^4$, and Hano \cite{HanoConformalImmersions} extended this existence result to every even-dimensional sphere $\mathbb S^{2n}$, $n\geq2$.\footnote{A branched conformal minimal immersion $\iota:\Sigma\to\mathbb S^{2n}\subset\mathbb R^{2n+1}$
is called \emph{superminimal} if, in any local conformal coordinate
$z$ on its regular locus, $(\partial_z^j\iota,\partial_z^k\iota)=0$
for all integers $j,k\geq1$, where the derivatives are
taken in $\mathbb R^{2n+1}$ and $(\,\cdot\,,\,\cdot\,)$
denotes the complex-bilinear extension of the Euclidean
inner product to $\mathbb C^{2n+1}$.} 
Here and throughout, linearly full means that the image of the immersion is not contained in any proper totally geodesic subsphere.
Subsequently, Chi--Mo \cite{ChiMoSuperminimalModuli} obtained dimension bounds for the moduli spaces of branched superminimal immersions into $\mathbb S^4$ with fixed degree and conformal structure, while Fern\'andez \cite{FernandezHarmonicSpheres} determined the exact complex dimensions of the spaces of linearly full superminimal two-spheres of fixed degree in $\mathbb S^{2n}$.\footnote{\,Note that specifying ``superminimal'' is redundant here, as every minimal two-sphere is automatically superminimal; see \cite{CalabiMinimalImmersionsSpheres}.}  
More recently, in \cite{CaniatoHighDimensionalFamilies}, we constructed families of arbitrarily large dimension consisting of linearly full branched superminimal immersions, with both the conformal structure of the domain and the target sphere $\mathbb S^{2n}$, $n\geq2$, fixed. However, in $\mathbb{S}^4$ superminimality imposes an obstruction to negative curvature: a nonconstant branched superminimal immersion of a closed Riemann surface into $\mathbb{S}^4$ cannot have nonpositive Gaussian curvature throughout its regular locus \cite[Remark~1.4]{CaniatoHighDimensionalFamilies}. Thus, any solution to Yau's problem in $\mathbb S^4$ must lie outside the superminimal class.

These nonexistence results might suggest that positive ambient sectional curvature obstructs the existence of closed minimal surfaces with everywhere negative Gaussian curvature. Holomorphic embeddings into complex projective spaces show, however, that no such obstruction holds in general. Tian's approximation theorem \cite{TianPolarizedKahlerMetrics} gives holomorphic embeddings of closed hyperbolic surfaces into projective spaces of increasing dimension, whose induced Fubini--Study metrics converge in $C^2$ to the hyperbolic metric after rescaling. Their Gaussian curvature is therefore everywhere negative for sufficiently large dimension. More recently, Mohsen \cite{MohsenNegativelyCurvedCompleteIntersections} proved an existence theorem with the ambient manifold fixed: every smooth complex projective manifold of complex dimension at least three, equipped with a prescribed Hermitian metric, contains a smooth closed complex curve with everywhere negative induced Gaussian curvature. For K\"ahler ambient metrics, these holomorphic curves are minimal; in particular, Mohsen's theorem gives embedded examples already in $\mathbb{CP}^3$ with its Fubini--Study metric. These curves arise as complete intersections defined by sections of high tensor powers of an ample line bundle, using the asymptotic methods of Donaldson and Auroux; see \cite{Donaldson1996,Auroux1997,Auroux2001Estimated}.

Turning back to the setting of round spheres, recent progress has come from the connection between harmonic maps, unitary representations and random matrix theory developed by Song \cite{SongRandomHarmonicMaps}. In particular, \cite[Theorem~0.3]{SongRandomHarmonicMaps} constructs closed branched minimal immersions $\iota_j:\Sigma_j\to\mathbb{S}^{n_j}$, with $n_j\to+\infty$, such that
\begin{align*}
    \lim_{j\to\infty}
    \frac{1}{\operatorname{Area}(\iota_j)}
    \int_{\Sigma_j}
    |K_{\iota_j}+8|\,
    d\operatorname{vol}_{g_{\iota_j}}=0.
\end{align*}
This averaged convergence does not establish negativity of the Gaussian curvature at every point. Song's argument gives smooth convergence on compact subsets of the underlying thrice-punctured sphere, but does not provide uniform curvature control near its cusps.

Building on Song's method, Ancona--Labourie--Roig-Sanchis--Toulisse \cite{AnconaLabourieRoigSanchisToulisse} obtained closed minimal immersions into all sufficiently high-dimensional round spheres with everywhere negative Gaussian curvature. More precisely, their curvatures converge uniformly to $-8$ as the ambient dimension tends to infinity. They replace the thrice-punctured sphere in Song's argument with compact hyperbolic orbifolds with trivial Teichm\"uller space, obtaining uniform geometric control over the entire surface. However, their argument does not provide an explicit dimension threshold or establish embeddedness.

In this paper, we prove the existence of closed embedded negatively curved minimal surfaces in the round four-sphere, the smallest possible ambient dimension. For a surface specified as a submanifold rather than by a parametrisation, we write $K_\Sigma$ for its Gaussian curvature.
\begin{theorem}
    \label{thm:s4-main}
    There exists $m_0\in\mathbb{N}$ such that, for every integer $m\geq m_0$, the unit round sphere $\mathbb{S}^4$ contains a smooth closed connected embedded minimal surface $\Sigma_m$ of genus $m^2+1$ satisfying
    \begin{align*}
        &K_{\Sigma_m}<0 \quad\text{ everywhere on }\Sigma_m \qquad\mbox{ and }\qquad \lim_{m\to\infty}\operatorname{Area}(\Sigma_m)=4\pi^2.
    \end{align*}
\end{theorem}
These surfaces are obtained by deforming two nearby toroidal sheets joined by $m^2$ small necks modelled on catenoids, via a procedure known as \emph{doubling}. Their areas tend to twice the area of the flat Clifford torus, while their genera tend to infinity. In particular, the Gauss--Bonnet theorem gives
\begin{align*}
    \int_{\Sigma_m}K_{\Sigma_m}\,d\operatorname{vol}_{\Sigma_m}=-4\pi m^2.
\end{align*}
The construction and the estimates that preserve negative curvature under the deformations are described below.
\begin{corollary}
    \label{cor:s4-higher-spheres}
    Every unit round sphere $\mathbb{S}^n$, $n\geq4$, contains infinitely many smooth closed connected embedded minimal surfaces with everywhere negative Gaussian curvature and distinct topological type.
\end{corollary}
\begin{proof}
    Compose any embedding supplied by Theorem~\ref{thm:s4-main} with a totally geodesic inclusion
    $\mathbb{S}^4\hookrightarrow\mathbb{S}^n$. This preserves the induced metric, minimality and
    embeddedness.
\end{proof}
\subsection{Outline of the proof}
We construct our minimal surfaces by joining two nearby copies of the Clifford torus with small catenoidal necks inside a suitable minimal hypersurface $\mathcal N\subset\mathbb S^4$, and then perturbing the resulting initial surfaces to correct their mean curvature. Our analytic approach relies on the doubling construction of Kapouleas--Yang \cite{KapouleasYangClifford} and the linearised doubling theory of Kapouleas--McGrath \cite{KapouleasMcGrathLD}. To the best of our knowledge, this is the first extension of doubling techniques to higher codimension. The choice of $\mathcal N$ is essential: its symmetries and its geometry as a hypersurface of $\mathbb S^4$ provide the cancellations and estimates needed both to solve the mean-curvature equations and to establish strictly negative Gaussian curvature on the final surfaces.

For each sufficiently large integer $m$, we choose a square lattice $L_m\subset\mathbb T$ with $m^2$ points. For $\tau>0$ to be determined later, we then let $\psi_{\tau}=\tau\Phi_m$ be the principal term of the height function defining the upper sheet, where
$\Phi_m\in L^1(\mathbb T)\cap C^\infty(\mathbb T\smallsetminus L_m)$ is the unique function invariant under translations by elements of $L_m$ that satisfies
\begin{align*}
    (\Delta_{\mathbb T}+2)\Phi_m=2\pi\sum_{p\in L_m}\delta_p \qquad\mbox{ distributionally on }\mathbb T.
\end{align*}
Here, $\delta_p$ denotes the Dirac measure of mass one at $p$. Thus, $\psi_{\tau}$ satisfies $(\Delta_{\mathbb T}+2)\psi_{\tau}=0$ away from $L_m$. This cancels the linear term in the scalar mean curvature within $\mathcal N$ of the graph with height $\psi_{\tau}$. The nonlinear terms will be estimated and corrected later. Near a lattice point $p\in L_m$, $\psi_{\tau}$ has the expansion
\begin{align*}
    \psi_{\tau}=\tau\log r+\tau A_m+o_r(1), \qquad r:=d_{\mathbb T}(\,\cdot\,,p)\to0,
\end{align*}
where $A_m\in\mathbb R$ is independent of $p$ by translation invariance, and $d_{\mathbb T}$ denotes the geodesic distance induced by the flat metric on $\mathbb T$. The Euclidean catenoid used to define the model catenoidal neck $F_{\tau,p}$ at $p$ in \eqref{eq:s4-symmetry-neck-map} has upper height $C_\tau(r)=\tau\operatorname{arcosh}(r/\tau)$ for $r\geq\tau$. Since
\begin{align*}
    C_\tau(r)=\tau\log r+\tau\log(2/\tau)+O(\tau^3r^{-2})\qquad\forall\,r\ge\tau,
\end{align*}
its logarithmic term agrees with that of $\psi_\tau$. To make the additive constants agree as well, we choose the reference radius $\tau_0$ so that
\begin{align*}
    A_m=\log(2/\tau_0),
    \qquad\text{ or equivalently }\qquad
    \tau_0=2e^{-A_m}.
\end{align*}
The Green-function estimates give $A_m=m^2/(2\pi)+\log m+O(1)$ as $m\to+\infty$, and hence they force
\begin{align*}
    \tau_0\asymp m^{-1}\exp\left(-\frac{m^2}{2\pi}\right),
\end{align*}
with comparison constants independent of $m$. We allow the common minimum radius of the catenoidal necks to vary by setting $\tau_\zeta:=e^\zeta\tau_0$, with $\zeta\in[-1,1]$. Since by construction $\pm\psi_{\tau_{\zeta}}$ are
the principal terms of the two sheet height functions, $\zeta$ varies the neck waist radius and the separation of the sheets at once. This parameter will be chosen when solving the remaining scalar mean-curvature equation, to cancel the coefficient which is not determined by the scalar correction alone.

In Section~\ref{sec:s4-scalar-family}, we join the two sheets to the catenoidal models $F_{\tau_{\zeta},p}$ at distances comparable to $\tau_\zeta^\alpha$, for a fixed sufficiently small $\alpha>0$. These distances are much larger than $\tau_\zeta$ and much smaller than the lattice spacing, which is of order $m^{-1}$. The resulting surface $M_\zeta\subset\mathcal N$ is smooth, closed, connected and embedded, of genus $m^2+1$.

To specify a normal deformation of $M_\zeta$ within $\mathbb{S}^4$ that will make it minimal, we introduce two smooth displacement functions $b,\phi\in C^{\infty}(M_{\zeta})$. Writing $\nu_\zeta$ for the outward unit normal of $M_\zeta$ within $\mathcal N$, we first move each point within $\mathcal N$ along $\phi\nu_{\zeta}$, and then move the resulting point in the direction of the unit normal $n$ to $\mathcal N$ in $\mathbb S^4$ by a displacement $b$:
\begin{align*}
    G_{\zeta,\phi}(x)&:=\exp_x^{\mathcal N}\bigl(\phi(x)\nu_\zeta(x)\bigr),\\
    \widehat X_{\zeta,\phi,b}(x)&:=\exp_{G_{\zeta,\phi}(x)}^{\mathbb S^4}\bigl(b(x)n(G_{\zeta,\phi}(x))\bigr)\qquad\forall\,x\in M_\zeta.
\end{align*}
The final minimal surfaces will be the images of $\widehat X_{\zeta,\phi,b}$ for a suitable choice of $\zeta,\phi,b$.

Throughout our construction, we impose the symmetries defined in Section~\ref{sec:s4-geometry} on the initial surfaces and all their deformations to make the linear problem solvable with suitable estimates.
The hypersurface $\mathcal N$ is preserved by the lattice translations, reflections in the torus coordinates, interchange of the two circle factors, and reflection across the central torus. Interchanging the circle factors reverses the chosen normal orientation of $\mathcal N$, while preserving the outward normal orientation of $M_\zeta$ within $\mathcal N$. Consequently, $b$ must change sign under this interchange, whereas $\phi$ is invariant.
On the central torus, lattice invariance excludes the kernel of $\Delta_{\mathbb T}+2$, making the Green function $\Phi_m$ uniquely determined. On each model neck, the symmetry conditions on $b$ force zero average around every circular cross-section. The resulting angular Poincar\'e inequality is essential to obtaining inverse estimates independent of the waist radius. The conditions on $\phi$ allow a constant angular component, which needs to be treated via a separate linear argument. Since these symmetries are induced by ambient isometries, the nonlinear mean-curvature equations preserve the corresponding conditions on the two components.

To test minimality after deformation, we parallel transport the two normals of $G_{\zeta,\phi}(M_\zeta)$ along the spherical displacement geodesics and project onto the new normal planes. For small deformations, these projections remain independent. Minimality is therefore equivalent to the vanishing of the mean-curvature vector's inner products with both normal fields.

For every $\zeta\in[-1,1]$ and every $\phi$ satisfying the bounds in Section~\ref{sec:s4-transverse}, we first solve for $b$ so that the mean-curvature component associated with $n$ vanishes. We linearise the mean-curvature equation associated with $n$ with respect to $b$ at $b=0$, keeping $\zeta$ and $\phi$ fixed. The uniform bound on the second fundamental form of $\mathcal N$ controls the zeroth-order coefficient of the resulting operator, even as the neck curvatures diverge. The angular Poincar\'e inequality, metric comparisons and elliptic estimates then give a uniformly bounded inverse in the weighted spaces used in that section. The quadratic gradient bound \eqref{eq:s4-height-graph-identities} on the sheets and the neck estimates bound the mean-curvature component associated with $n$ at $b=0$ by $C\tau_\zeta^2$ in the weighted H\"older norm used in Section~\ref{sec:s4-transverse}. The nonlinear estimates show that the fixed-point map preserves a closed ball of radius $C\tau_\zeta^2$ in the weighted displacement norm and has Lipschitz constant $O(\tau_\zeta)$. For sufficiently large $m$, the contraction theorem therefore gives a solution satisfying
\begin{align*}
    \|b\|_{L^\infty(M_\zeta)}\leq C\tau_\zeta^2.
\end{align*}
The weighted bound also controls the first two derivatives of $b$, which will appear in the subsequent Gaussian curvature estimates.

The displacement $b$ also changes the remaining mean-curvature component. By \eqref{eq:s4-feedback-estimates}, this change has norm at most $C\tau_\zeta^{2+\alpha/2}$ in $Y_\zeta$, the weighted $C^{0,\beta}$ space defined in \eqref{eq:s4-XY}. Adding this contribution to the scalar nonlinear remainder preserves the $O(\tau_\zeta^2)$ bound needed to solve the remaining
equation for $\phi$ and $\zeta$. The scalar linear construction solves the linearised equation up to a multiple of one prescribed function, and varying $\zeta$ allows us to cancel the remaining coefficient. We formulate these choices as a fixed-point problem that accounts for the dependence of $b$ on $(\phi,\zeta)$. Since positive powers of $\tau_\zeta$ absorb the polynomial losses in $m$, the estimates show that the fixed-point map preserves a nonempty convex set of corrections and parameters. In a weaker H\"older topology, this set is compact and the map is continuous. Schauder's fixed-point theorem therefore gives a choice for which both normal components vanish. The resulting surface $\Sigma_m:=\widehat X_{\zeta,\phi,b}(M_\zeta)$ is minimal in $\mathbb S^4$.

It remains to prove that $K_{\Sigma_m}<0$ everywhere. Write $\tau_m:=\tau_\zeta$ for the final neck parameter. For a smooth height function $w$ on an open subset of $\mathbb T$, the quadratic term in the Gaussian curvature of its graph within $\mathcal N$ is
\begin{align*}
    \mathcal K_2(w):=\det D^2w+w\Delta_{\mathbb T}w+|Dw|^2,
\end{align*}
where derivatives, norms and determinants are computed using the flat metric on $\mathbb T$. Since $\Delta_{\mathbb T}\psi_{\tau_m}=-2\psi_{\tau_m}$ away from $L_m$, the Green-function estimates give
\begin{align*}
    \mathcal K_2(\psi_{\tau_m})=|D\psi_{\tau_m}|^2-\frac12|D^2\psi_{\tau_m}|^2
    \leq-c\tau_m^2\bigl(m^4+\rho^{-4}\bigr)
    \qquad\mbox{ on }\mathbb T\smallsetminus L_m,
\end{align*}
where $\rho:=d_{\mathbb T}(\,\cdot\,,L_m)$ and $c>0$ is independent of $m$. At distances comparable to the lattice spacing, the magnitude of this negative bound is $m^4\tau_m^2\to0$ as $m\to+\infty$. To preserve negative curvature,
we must therefore control the total curvature error relative to $\tau_m^2(m^4+\rho^{-4})$.

Section~\ref{sec:s4-negative-curvature} proves this comparison on the sheet portions outside a fixed multiple $R$ of the neck radius. Choosing $R$ sufficiently large and then $m$ sufficiently large depending on $R$ makes the combined error from interpolation, higher-order curvature terms and both deformations smaller than half the magnitude of the negative bound.

On the remaining neck portions, we rescale lengths by $\tau_m^{-1}$. For fixed $R$, the initial necks converge smoothly to a compact portion of the Euclidean catenoid of waist radius one. Both deformations tend to zero in $C^2$ in the rescaled coordinates, so the final necks have the same limit. Since this catenoidal portion has Gaussian curvature bounded above by a strictly negative
constant, curvature continuity and scaling back give $K_{\Sigma_m}\leq-c_R\tau_m^{-2}$ on the neck portions, with $c_R>0$ independent of $m$. Together with the sheet estimates, this proves $K_{\Sigma_m}<0$ everywhere.

Finally, the neck and interpolation regions of $M_\zeta$ have total area $O(m^2\tau_m^{2\alpha})\to0$ as $m\to+\infty$, while each remaining sheet has area tending to the Clifford torus area $2\pi^2$. The two deformations change the area element by a factor that converges uniformly to one on $M_\zeta$ as $m\to+\infty$. Consequently, $\operatorname{Area}(\Sigma_m)\to4\pi^2$.

\begin{remark}[Comparison with doubling in $\mathbb S^3$]
    The choice of the minimal hypersurface $\mathcal{N}$ in which the Clifford torus $\mathbb{T}$ is doubled is essential to the curvature argument. In flat coordinates $(s,t)$ along the two circle factors, choose the unit normal to $\mathbb T$ within the equatorial $\mathbb S^3$ so that the principal curvatures in the $\partial_s$ and $\partial_t$ directions are $-1$ and $1$, respectively. For a smooth height function $w$ on an open subset of $\mathbb T$, with sufficiently small values, the normal graph within $\mathbb S^3$ has induced metric
    \begin{align*}
        (1+\sin 2w)\,ds^2+(1-\sin 2w)\,dt^2+dw\otimes dw.
    \end{align*}
    The linearisation of its Gaussian curvature at $w=0$ is $\partial_{ss}w-\partial_{tt}w$. This term changes sign for opposite sheet heights and therefore cannot supply a strictly negative bound on both sheets.
    
    In $\mathcal N$, the corresponding induced metric is 
    \begin{align*}
        \cos^2w\,g_{\mathbb T}+dw\otimes dw.
    \end{align*}
    This metric, and hence its Gaussian curvature, is unchanged under $w\mapsto-w$. Consequently, the curvature expansion at $w=0$ has no linear term. Its leading term is the quadratic expression $\mathcal K_2(w)$ described above, for which the Green heights give the same strictly negative bound on both sheets.
\end{remark}

\subsection{Open problems}
\label{subsec:s4-open-problems}
The results presented in this work leave several open questions concerning general ambient metrics, the possible genera of negatively curved minimal surfaces in round spheres, and the existence of negatively curved examples within more specific subclasses.

In particular, Song has proposed the following extension of Yau's question to general closed Riemannian manifolds
\cite[p.~4]{SongQuestions}.
\begin{question}[Song]
    \label{que:s4-general-ambient}
    Let $(M,g)$ be a closed Riemannian manifold with
    $\dim M\geq4$. Does there exist a closed connected
    surface $\Sigma$ and a smooth minimal immersion
    $\iota:\Sigma\to(M,g)$ such that the Gaussian
    curvature of $g_\iota:=\iota^*g$ is strictly
    negative at every point?
\end{question}
Corollary~\ref{cor:s4-higher-spheres} and Mohsen's work
give a positive answer to this question for round spheres
and for complex projective manifolds of complex dimension
at least three equipped with K\"ahler metrics, respectively.
The question remains open for general closed Riemannian
manifolds.

A more focused problem is to establish the same existence result for every smooth metric in some $C^\infty$-neighbourhood of the round metric on $\mathbb{S}^4$. Strict negativity of Gaussian curvature persists under sufficiently small $C^2$ perturbations of a fixed immersion and the ambient metric. However, one must also construct a minimal immersion for the perturbed metric, as curvature continuity alone does not provide this.

Even for the round metric, the present construction does not determine all genera for which negatively curved minimal surfaces exist. Theorem~\ref{thm:s4-main} gives only the genera $m^2+1$ for all sufficiently large integers $m$.
\begin{question}
    \label{que:s4-prescribed-genus}
    Which integers $\gamma\geq2$ occur as the genus of a smooth closed connected orientable embedded minimal surface in $\mathbb{S}^4$ with everywhere negative Gaussian curvature? In particular, does every sufficiently large integer occur?\,\footnote{\,The lower bound $\gamma\geq2$ follows from Gauss--Bonnet.}
\end{question}
Finally, the nonexistence of negatively curved closed superminimal surfaces in $\mathbb S^4$ motivates the corresponding existence problem in $\mathbb S^6$. The following question also asks for embeddedness, which is not required in
\cite[Conjecture~1.3]{CaniatoHighDimensionalFamilies}.
\begin{question}
    \label{que:s4-superminimal-six-sphere}
    Does $\mathbb{S}^6$ contain a smooth closed connected embedded linearly full superminimal surface with everywhere negative Gaussian curvature?
\end{question}

\renewcommand\labelitemi{$\vcenter{\hbox{\tiny$\bullet$}}$}
\begingroup
\renewcommand{\addcontentsline}[3]{}
\subsection*{On the use of AI}
The author used ChatGPT-6 Astra solely to improve the clarity and style of the prose, to assist in identifying potential errors or gaps in the mathematical arguments, and to create the figures appearing in the manuscript. The research ideas and all the proofs were developed without the use of AI tools. All mathematical arguments were independently checked by the author, who takes full responsibility for the content of this paper.
\par
\endgroup
\renewcommand\labelitemi{$\vcenter{\hbox{\tiny$\bullet$}}$}
\begingroup
\renewcommand{\addcontentsline}[3]{}
\subsection*{Acknowledgements}
I am very grateful to Antoine Song, for introducing me to this problem during our first meeting at Caltech more than three years ago, and for many enlightening discussions we had since then on this topic.

I would also like to thank Xingzhe Li, Ben Lowe, and Daniel Stern, for the interesting exchanges we had over time on (equivariant) minimal surfaces in round spheres.

Finally, I would like to thank Otis Chodosh, Christos Mantoulidis, and Felix Schulze, for their interest in this work and for the very insightful comments and suggestions which helped to improve the quality of the present manuscript.
\par
\endgroup
\section{The ambient hypersurface and its symmetries}
\label{sec:s4-geometry}

We first embed the Clifford torus inside an explicit minimal
hypersurface $\mathcal{N}$ of the round four-sphere and compute the geometric quantities needed in the construction. For a surface in $\mathcal{N}$, we then express the two components of its mean-curvature vector within $\mathbb{S}^4$ and their normal linearisation, and provide suitable tubular coordinates for deformations out of the hypersurface. Finally, we specify the finite group of symmetries used in the construction and derive the resulting restrictions on the normal variations and their angular Fourier terms near each neck.

\subsection{A minimal hypersurface containing the Clifford torus}

The hypersurface $\mathcal{N}$ constructed below will contain the initial surfaces used in the perturbation argument. Crucially for our purposes, we will show that $\mathcal N$ is minimal in $\mathbb S^4$ and contains a totally geodesic Clifford torus whose scalar Jacobi operator within $\mathcal N$ admits a particularly simple expression.

We regard the unit round four-sphere as
\begin{align*}
    \mathbb{S}^4=\{(w_1,w_2,x_5)\in\mathbb{C}^2\oplus\mathbb{R}\mbox{ : }|w_1|^2+|w_2|^2+x_5^2=1\}.
\end{align*}
Set $\ell:=\sqrt2\pi$ and let $\mathbb{T}:=(\mathbb{R}/\ell\mathbb{Z})^2$, equipped with the flat metric $g_{\mathbb{T}}:=ds^2+dt^2$. Define $f:\mathbb{T}\to\mathbb{S}^4$ by
\begin{align}
    \label{eq:s4-clifford-frame}
    f(s,t):=\frac{1}{\sqrt{2}}
    \bigl(e^{i\sqrt{2}s},e^{i\sqrt{2}t},0\bigr) \qquad\forall\,(s,t)\in\mathbb{T}.
\end{align}
Note that $f$ is an isometric embedding of $\mathbb{T}$ into $\mathbb{S}^4$ and $\operatorname{Area}(\mathbb{T})=2\pi^2$. Fix $z_0\in(0,\pi/4)$ and define the smooth embedding $\mathcal F:\mathbb{T}\times(-z_0,z_0)\to\mathbb{S}^4$ by
\begin{align}
    \label{eq:s4-suspension}
    \mathcal{F}(s,t,z):=\cos z\,f(s,t)+\sin z\,e_5 \qquad\forall\,(s,t,z)\in\mathbb{T}\times(-z_0,z_0),
\end{align}
where $e_5=(0,0,1)\in\mathbb{C}^2\oplus\mathbb{R}$. We let 
\begin{align*}
    \mathcal{N}:=\mathcal{F}\bigl(\mathbb{T}\times(-z_0,z_0)\bigr)
\end{align*}
and we define the projection $\pi:\mathcal{N}\to\mathbb{T}$ as
\begin{align*}
    \pi(\mathcal{F}(s,t,z)):=(s,t) \qquad\forall(s,t,z)\in\mathbb{T}\times(-z_0,z_0).
\end{align*}
Let $g_{\mathbb{S}^4}$ denote the unit round metric on $\mathbb{S}^4$,
let $\iota_{\mathcal{N}}:\mathcal{N}\hookrightarrow\mathbb{S}^4$
denote the inclusion, and equip $\mathcal{N}$ with the induced
metric $g_{\mathcal{N}}:=\iota_{\mathcal{N}}^*g_{\mathbb{S}^4}$.
Note that 
\begin{align}
    \label{eq:s4-suspension-metric}
    g_{\mathcal{N}}=(\mathcal{F}^{-1})^*\bigl(dz^2+\cos^2z\,(ds^2+dt^2)\bigr).
\end{align}
Define $N:\mathbb{T}\to\mathbb{C}^2\oplus\mathbb{R}$ by
\begin{align*}
    N(s,t):=\frac{1}{\sqrt{2}}\bigl(e^{i\sqrt{2}s},-e^{i\sqrt{2}t},0\bigr)\qquad\forall\,(s,t)\in\mathbb{T}.
\end{align*}
The vector $N(s,t)$ has unit Euclidean norm and is orthogonal to $\mathcal{F}(s,t,z)$ and to each of $\partial_s\mathcal{F}$, $\partial_t\mathcal{F}$ and $\partial_z\mathcal{F}$ at $(s,t,z)$. Consequently, the field $n$ defined along $\mathcal{N}$ by
\begin{align*}
    n\bigl(\mathcal{F}(s,t,z)\bigr):=N(s,t) \qquad\forall\,(s,t,z)\in\mathbb{T}\times(-z_0,z_0)
\end{align*}
is a smooth global unit normal to $\mathcal{N}$ within $\mathbb{S}^4$.

We call $\mathcal{N}$ the \textit{suspension band} of the Clifford torus and refer to $f(\mathbb{T})=\mathcal{F}(\mathbb{T}\times\{0\})$ as its \textit{central torus}. 

\begin{remark}
 \label{rem:s4-signed-distance}
    Note that the function on $\mathcal{N}$ defined by $\mathcal{F}(s,t,z)\mapsto z$ is the signed distance to the central torus in $(\mathcal{N},g_{\mathcal{N}})$, i.e.
    \begin{align*}
        \operatorname{dist}_{g_{\mathcal{N}}}
        \bigl(\mathcal{F}(s,t,z),f(\mathbb{T})\bigr)=|z|
        \qquad\forall\,(s,t,z)\in
        \mathbb{T}\times(-z_0,z_0).
    \end{align*}
    Indeed, the coordinate segment joining $\mathcal{F}(s,t,z)$ to $\mathcal{F}(s,t,0)$ has length $|z|$, while every curve joining $\mathcal{F}(s,t,z)$ to the central torus has length at least $|z|$, by \eqref{eq:s4-suspension-metric}.
\end{remark}
%
%
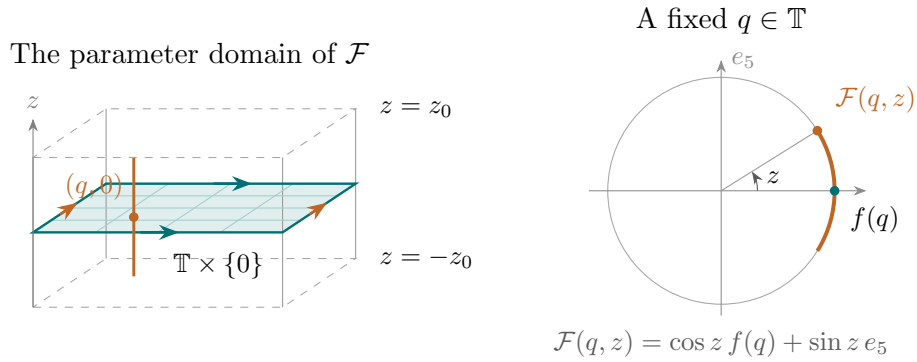
\begin{figure}[htbp]
    \centering
    \begingroup
    \definecolor{bandteal}{RGB}{0,112,116}
    \definecolor{bandorange}{RGB}{190,102,36}
    \begin{tikzpicture}[
        >=Stealth,
        every node/.style={font=\small}
    ]
        \begin{scope}[
            x={(1cm,0cm)},y={(0.42cm,0.28cm)},z={(0cm,0.9cm)}
        ]
            \node[font=\normalsize] at (1.65,1,2.28)
                {The parameter domain of $\mathcal F$};
            \fill[bandteal!12]
                (0,0,0)--(3.3,0,0)--(3.3,2.3,0)--(0,2.3,0)--cycle;
            \foreach \height in {-1.1,1.1} {
                \draw[black!35,dashed]
                    (0,0,\height)--(3.3,0,\height)--(3.3,2.3,\height)
                    --(0,2.3,\height)--cycle;
            }
            \foreach \x in {0,3.3} {
                \foreach \y in {0,2.3} {
                    \draw[black!40] (\x,\y,-1.1)--(\x,\y,1.1);
                }
            }
            \foreach \x in {0.825,1.65,2.475} {
                \draw[bandteal!35] (\x,0,0)--(\x,2.3,0);
            }
            \foreach \y in {0.575,1.15,1.725} {
                \draw[bandteal!35] (0,\y,0)--(3.3,\y,0);
            }
            \draw[bandteal,line width=0.9pt]
                (0,0,0)--(3.3,0,0)--(3.3,2.3,0)--(0,2.3,0)--cycle;
            \foreach \y in {0,2.3} {
                \draw[bandteal,line width=1pt,->]
                    (1.42,\y,0)--(1.93,\y,0);
            }
            \foreach \x in {0,3.3} {
                \draw[bandorange,line width=1pt,->]
                    (\x,0.85,0)--(\x,1.4,0);
            }
            \draw[bandorange,line width=1.1pt]
                (1.03,0.72,-0.87)--(1.03,0.72,0.87);
            \fill[bandorange] (1.03,0.72,0) circle[radius=1.8pt];
            \node[anchor=south east,text=bandorange] at (1.03,0.72,0.1)
                {$(q,0)$};
            \node[anchor=west] at (3.48,2.3,1.1) {$z=z_0$};
            \node[anchor=west] at (3.48,2.3,-1.1) {$z=-z_0$};
            \node[anchor=north] at (2.45,0,-0.14)
                {$\mathbb T\times\{0\}$};
            \draw[black!50,->] (0,0,-1.1)--(0,0,1.67)
                node[above] {$z$};
        \end{scope}
        \begin{scope}[shift={(9.1,0.55)}]
            \node[font=\normalsize] at (0,2.25)
                {A fixed $q\in\mathbb T$};
            \draw[black!35] (0,0) circle[radius=1.5];
            \draw[black!45,->] (-1.75,0)--(1.92,0);
            \draw[black!45,->] (0,-1.65)--(0,1.73)
                node[right] {$e_5$};
            \coordinate (base) at (1.5,0);
            \coordinate (point) at ({1.5*cos(32)},{1.5*sin(32)});
            \draw[black!40] (0,0)--(point);
            \draw[bandorange,line width=1.5pt]
                ({1.5*cos(-32)},{1.5*sin(-32)})
                arc[start angle=-32,end angle=32,radius=1.5];
            \fill[bandteal] (base) circle[radius=1.8pt];
            \fill[bandorange] (point) circle[radius=1.8pt];
            \node[below right=2pt] at (base) {$f(q)$};
            \node[above right=3pt,text=bandorange] at (point)
                {$\mathcal F(q,z)$};
            \draw[black!65,->]
                (0.48,0) arc[start angle=0,end angle=32,radius=0.48];
            \node at (0.69,0.18) {$z$};
            \node[align=center,text=black!65] at (0,-2.02)
                {$\mathcal F(q,z)=\cos z\,f(q)+\sin z\,e_5$};
        \end{scope}
    \end{tikzpicture}
    \endgroup
    \caption{The parameter band and the geodesics defining $\mathcal N$.
    Opposite vertical faces are identified; the central square represents
    $\mathbb T\times\{0\}$. For fixed $q$, the orange segment maps
    to the circle arc $\cos z\,f(q)+\sin z\,e_5$ shown on the right.}
    \label{fig:s4-band-coordinates}
\end{figure}

We fix the following conventions for immersions, which will
also be used in the subsequent sections.
Let $\iota:M\to(P,\overline g)$ be a smooth immersion and
equip $M$ with the induced metric $g:=\iota^*\overline g$.
We denote the Levi-Civita connection of $g$ by $\nabla^g$.
For a smooth function $u$ on $M$, the notation $\nabla^g u$
denotes its gradient with respect to $g$. We use the convention
\begin{align*}
    \Delta_g u:=\operatorname{div}_g(\nabla^g u)
    \qquad\forall\,u\in C^\infty(M)
\end{align*}
for the Laplace--Beltrami operator.
We omit the metric subscript when it is clear from the context.

Denote by $\overline\nabla$ the connection on $\iota^*TP$
induced by the Levi-Civita connection of $(P,\overline g)$.
Let
\begin{align*}
    T^\perp_\iota M
    :=\bigl(d\iota(TM)\bigr)^\perp\subset\iota^*TP
\end{align*}
be the normal bundle of $\iota$, where ${\cdot}^\perp$ indicates the $\overline g$-orthogonal complement of $d\iota(TM)$ in $\iota^*TP$. For a vector in $\iota^*TP$, the superscript ${}^\perp$
denotes its $\overline g$-orthogonal projection onto
$T^\perp_\iota M$.
The \textit{second fundamental form} of $\iota$ is the tensor $A_\iota\in\Gamma(\operatorname{Sym}^2T^*M\otimes
T^\perp_\iota M)$ defined by
\begin{align*}
    A_\iota(X,Y)
    :=\bigl(\overline\nabla_X(d\iota(Y))\bigr)^\perp
    \qquad\forall\,X,Y\in\Gamma(TM).
\end{align*}
For a smooth unit normal field $\eta\in\Gamma(T^\perp_\iota M)$, its scalar
component in the direction $\eta$ is the bilinear form
$h_\iota^\eta\in\Gamma(\operatorname{Sym}^2T^*M)$, defined by
\begin{align*}
    h_\iota^\eta(X,Y)
    :=\overline g\bigl(A_\iota(X,Y),\eta\bigr)
    \qquad\forall\,X,Y\in\Gamma(TM).
\end{align*}
The \textit{shape operator} of $\iota$ in the direction $\eta$ is the $g$-self-adjoint endomorphism
$S_\iota^\eta\in\Gamma(T^*M\otimes TM)$ characterised by
\begin{align*}
    g(S_\iota^\eta(X),Y)=h_\iota^\eta(X,Y)
    \qquad\forall\,X,Y\in\Gamma(TM).
\end{align*}
In particular, $|S_\iota^\eta|_g=|h_\iota^\eta|_g$.
The \textit{mean-curvature vector} of $\iota$ and its scalar component in the direction $\eta$ are defined using the trace convention:
\begin{align*}
    \mathbf{H}_\iota
    &:=\operatorname{tr}_g A_\iota
    =\sum_{a=1}^{\dim M}A_\iota(e_a,e_a),\\
    H_{\iota,\eta}
    &:=\operatorname{tr}_g h_\iota^\eta
    =\sum_{a=1}^{\dim M}h_\iota^\eta(e_a,e_a)
    =\overline g(\mathbf{H}_\iota,\eta),
\end{align*}
where $(e_a)_{a=1}^{\dim M}$ is any local
$g$-orthonormal frame of $TM$.

The \textit{normal connection} of $\iota$ is given by
\begin{align*}
    \nabla^{\perp,\iota}_X\xi
    :=(\overline\nabla_X\xi)^\perp
    \qquad\forall\,X\in\Gamma(TM),\quad
    \xi\in\Gamma(T^\perp_\iota M).
\end{align*}
With these conventions, the Weingarten identity reads
\begin{align}
    \label{eq:s4-weingarten-convention}
    \overline\nabla_X\eta
    =-d\iota(S_\iota^\eta(X))
    +\nabla^{\perp,\iota}_X\eta
    \qquad\forall\,X\in\Gamma(TM).
\end{align}
Indeed, differentiating $\overline g(\eta,d\iota(Y))=0$
gives
\begin{align*}
    \overline g(\overline\nabla_X\eta,d\iota(Y))
    =-\overline g(\eta,\overline\nabla_X(d\iota(Y)))=-h_\iota^\eta(X,Y)
    =-g(S_\iota^\eta(X),Y)
    \qquad\forall\,X,Y\in\Gamma(TM),
\end{align*}
which determines the tangential component of
$\overline\nabla_X\eta$; its normal component is
$\nabla^{\perp,\iota}_X\eta$ by definition.
In codimension one, this normal component vanishes because
$\eta$ has unit length and spans the normal bundle.

For the fixed metrics on $\mathbb{S}^4$ and $\mathcal{N}$,
we use the abbreviations
\begin{align*}
    \nabla^{\mathbb{S}^4}:=\nabla^{g_{\mathbb{S}^4}},
    \qquad
    \nabla^{\mathcal{N}}:=\nabla^{g_{\mathcal{N}}},
\end{align*}
for their respective Levi-Civita connections.

\begin{lemma}
    \label{lem:s4-suspension-geometry}
    Write $S:=S_{\iota_{\mathcal{N}}}^n\in\Gamma(T^*\mathcal{N}\otimes T\mathcal{N})$ for the shape
    operator of the inclusion $\iota_{\mathcal{N}}:\mathcal{N}\hookrightarrow\mathbb{S}^4$ with respect to the unit normal $n$, using the preceding conventions.
    Define $E_1,E_2,E_3\in\Gamma(T\mathcal{N})$ by
    \begin{align*}
        E_1(\mathcal{F}(s,t,z))&:=\frac{1}{\cos z}\partial_s\mathcal F(s,t,z),\\
        E_2(\mathcal{F}(s,t,z))&:=\frac{1}{\cos z}\,
        \partial_t\mathcal F(s,t,z),\\
        E_3(\mathcal{F}(s,t,z))&:=\partial_z\mathcal F(s,t,z),
    \end{align*}
    for every $(s,t,z)\in\mathbb{T}\times(-z_0,z_0)$.
    Then, $(E_1,E_2,E_3)$ is a global
    $g_{\mathcal{N}}$-orthonormal frame of $T\mathcal{N}$. For every $(s,t,z)\in\mathbb{T}\times(-z_0,z_0)$, writing $q:=\mathcal{F}(s,t,z)$, one has
    \begin{align}
        \label{eq:s4-S}
        \begin{split}
            S_q(E_1(q))&=-\sec z\,E_1(q),\\
            S_q(E_2(q))&=\sec z\,E_2(q),\\
            S_q(E_3(q))&=0.
        \end{split}
    \end{align}
    In particular, $\mathcal{N}$ is a minimal hypersurface of $\mathbb{S}^4$.

    The central torus $f(\mathbb{T})\subset\mathcal N$ is totally geodesic as a hypersurface of $(\mathcal{N},g_{\mathcal{N}})$. The scalar Jacobi operator
    $L_{\mathbb{T}}:C^\infty(\mathbb{T})\to C^\infty(\mathbb{T})$
    of the isometric embedding
    $f:(\mathbb{T},g_{\mathbb{T}})\to(\mathcal{N},g_{\mathcal{N}})$,
    with respect to the unit normal field $E_3\circ f$,
    is given by
    \begin{align}
        \label{eq:s4-base-Jacobi}
        L_{\mathbb{T}}u=\Delta_{\mathbb{T}}u+2u
        \qquad\forall\,u\in C^\infty(\mathbb{T}),
    \end{align}
    where $\Delta_{\mathbb{T}}:=\Delta_{g_{\mathbb{T}}}$.

    Finally, there exists a constant $C=C(z_0)>0$ such that, for every $(s,t,z)\in\mathbb{T}\times(-z_0,z_0)$ and $q:=\mathcal{F}(s,t,z)$,
    \begin{align}
        \label{eq:s4-S-derivative}
        |S_q|_{g_{\mathcal{N}}}\leq C, 
        \qquad
        |(\nabla^{\mathcal{N}}S)_q|_{g_{\mathcal{N}}}
        \leq C|z|.
    \end{align}
    Here $\nabla^{\mathcal{N}}S$ is the covariant derivative of
    the tensor field $S$, and both norms are induced by
    $g_{\mathcal{N}}$.
\end{lemma}

The proof of Lemma \ref{lem:s4-suspension-geometry} relies only on standard computations in Riemannian geometry, and is therefore deferred to Appendix~\ref{app:s4-geometric-identities}.

\subsection{Normal variations and tubular coordinates}

For a surface $M$ contained in $\mathcal{N}$, minimality in $\mathbb{S}^4$ requires both its scalar mean curvature within $\mathcal{N}$ and the component of its mean-curvature vector within $\mathbb{S}^4$ in the direction $n$ to vanish. We compute the latter component and the normal connection in terms of the shape operator $S$ of $\mathcal{N}\subset\mathbb{S}^4$ and the chosen unit normal of $M$ within $\mathcal{N}$, and specialise
these formulas to height graphs over the central torus. We then represent the normal linearisation of the mean curvature map as a $2\times2$ matrix of scalar differential operators, acting on the coefficients of a normal variation in these two directions. Its off-diagonal entries describe how a variation in either direction affects the mean-curvature component in the other direction. Finally, we construct convenient tubular coordinates around $\mathcal{N}$ in $\mathbb{S}^4$.

Let $M\subset\mathcal{N}$ be a smooth embedded two-sided surface in $\mathcal{N}$, and let
\begin{align*}
    \jmath:M\hookrightarrow\mathcal{N},
    \qquad
    \iota:M\hookrightarrow\mathbb{S}^4
\end{align*}
denote its inclusions in $\mathcal N$ and $\mathbb S^4$, respectively. Equip $M$ with the induced metric
$g_M:=\jmath^*g_{\mathcal{N}}=\iota^*g_{\mathbb{S}^4}$, and fix a smooth global unit normal field $\nu$ to $M$
within $\mathcal{N}$. No compactness or minimality assumption is imposed on $M$ in this subsection. We write $\nabla^M:=\nabla^{g_M}$ and $\Delta_M:=\Delta_{g_M}$, and denote the restriction of $n$ to $M$ again by $n$. Then $(\nu,n)$ is a global orthonormal frame of the normal bundle $T^\perp_\iota M$.

Using the conventions above, the mean-curvature vectors of the two inclusions decompose as
\begin{align*}
    \mathbf{H}_\jmath&=H_{\jmath,\nu}\nu,\\
    \mathbf{H}_\iota&=H_{\iota,\nu}\nu+H_{\iota,n}n.
\end{align*}
The equality $h_\iota^\nu=h_\jmath^\nu$, verified in the proof of Lemma \ref{lem:s4-normal-identities} below, implies $H_{\iota,\nu}=H_{\jmath,\nu}$. Thus minimality in $\mathcal{N}$ requires $H_{\iota,\nu}=0$, while minimality in $\mathbb{S}^4$ requires both $H_{\iota,\nu}=0$ and $H_{\iota,n}=0$.

Inner products of vectors are taken with respect to $g_{\mathbb{S}^4}$; tensor norms and contractions use the corresponding induced metrics.

For the normal connection of $\iota$, we abbreviate $\nabla^\perp:=\nabla^{\perp,\iota}$. In the frame $(\nu,n)$, the connection $\nabla^{\perp}$ is determined by the one-form $\omega\in\Gamma(T^*M)$ defined by
\begin{align}
    \label{eq:s4-omega-definition}
    \omega(X):=\langle\nabla^\perp_X\nu,n\rangle
    =\langle\nabla^{\mathbb{S}^4}_X\nu,n\rangle
    \qquad\forall\,X\in\Gamma(TM).
\end{align}
Indeed, differentiating
$\langle\nu,\nu\rangle=\langle n,n\rangle=1$ and
$\langle\nu,n\rangle=0$ gives
\begin{align}
    \label{eq:s4-normal-connection-frame}
    \nabla^\perp_X\nu=\omega(X)n,
    \qquad
    \nabla^\perp_Xn=-\omega(X)\nu
    \qquad\forall\,X\in\Gamma(TM).
\end{align}
Thus $\omega$ records the coupling of the two normal directions.
The following lemma expresses this one-form and the second
fundamental form in the direction $n$ in terms of $S$. To state
the estimates, write
\begin{align*}
    \nu_h(x):=\sum_{a=1}^2
    \langle\nu(x),E_a(x)\rangle E_a(x)
    \qquad\forall\,x\in M
\end{align*}
for the component of $\nu$ in the plane spanned by $E_1,E_2$.

\begin{lemma}
    \label{lem:s4-normal-identities}
    For every $X,Y\in\Gamma(TM)$, one has
    \begin{align}
        \label{eq:s4-normal-identities}
        \begin{split}
            h_\iota^n(X,Y)&=\langle S(X),Y\rangle,\\
            H_{\iota,n}&=-\langle S(\nu),\nu\rangle,\\
            \omega(X)&=\langle S(X),\nu\rangle.
        \end{split}
    \end{align}
    Moreover,
    \begin{align}
        \label{eq:s4-omega-derivative}
        \begin{split}
            (\nabla^M_X\omega)(Y)&=\langle(\nabla^{\mathcal{N}}_XS)(Y),\nu\rangle+h_\iota^\nu(X,Y)\langle S(\nu),\nu\rangle-\langle S(Y),S_\iota^\nu (X)\rangle
            \qquad\forall\,X,Y\in\Gamma(TM).
        \end{split}
    \end{align}
    There exists $C=C(z_0)>0$ such that, for every
    $(s,t,z)\in\mathbb{T}\times(-z_0,z_0)$ with
    $x:=\mathcal{F}(s,t,z)\in M$,
    \begin{align}
        \label{eq:s4-normal-bounds}
        \begin{aligned}
            |(h_\iota^n)_x|&\leq C,
            &\qquad|H_{\iota,n}(x)|&\leq C|\nu_h(x)|^2,\\
            |\omega_x|&\leq C|\nu_h(x)|,
            &\qquad|(\nabla^M\omega)_x|
            &\leq C\bigl(|z|+|(S_\iota^\nu)_x|\bigr).
        \end{aligned}
    \end{align}
    The norms of $h_\iota^n$, $\omega$, $\nabla^M\omega$ and $S_\iota^\nu$ are induced by $g_M$, while $|\nu_h|$ is computed using $g_{\mathcal{N}}$.
\end{lemma}

Since Lemma \ref{lem:s4-normal-identities} follows from routine computations, we postpone its proof to Appendix~\ref{app:s4-geometric-identities}.

\begin{remark}
    For a surface given by a height graph over the central torus, the identities \eqref{eq:s4-normal-identities} give explicit expressions for $H_{\iota,n}$ and $\omega$ in terms of the defining height function and its first derivatives. Indeed, let $U\subset\mathbb{T}$ be open and let $w:U\to(-z_0,z_0)$ be smooth. Define
    $F_w:U\to\mathcal{N}$ and its image $M_w$ by
    \begin{align*}
        F_w(s,t):=\mathcal{F}(s,t,w(s,t))
        \qquad\forall\,(s,t)\in U,
        \qquad M_w:=F_w(U).
    \end{align*}
    On $U$, set $w_s:=\partial_sw$, $w_t:=\partial_tw$,
    $Dw:=(w_s,w_t)$, $c:=\cos w$ and $Q:=c^2+|Dw|^2$,
    where $|Dw|^2=w_s^2+w_t^2$. Apply the preceding notation
    to $M_w$, choosing its unit normal $\nu$ within $\mathcal{N}$
    so that $\langle\nu,E_3\rangle>0$. Then, the following identities hold on $U$:
    \begin{align}
        \label{eq:s4-height-graph-identities}
        \begin{split}
            \nu\circ F_w
            &=\frac{-w_s(E_1\circ F_w)-w_t(E_2\circ F_w)
            +c(E_3\circ F_w)}{\sqrt Q},\\
            H_{\iota,n}\circ F_w&=\frac{w_s^2-w_t^2}{cQ},\\
            F_w^*\omega&=\frac{w_s\,ds-w_t\,dt}{\sqrt Q}.
        \end{split}
    \end{align}

    Indeed, the graph tangent vectors are $cE_1+w_sE_3$
    and $cE_2+w_tE_3$, evaluated at $F_w(s,t)$.
    Their common unit normal with positive $E_3$ component
    is the first displayed vector. Since $S(E_1)=-c^{-1}E_1$,
    $S(E_2)=c^{-1}E_2$ and $S(E_3)=0$ there, substitution
    into \eqref{eq:s4-normal-identities} gives the remaining
    two formulas.
\end{remark}
We next consider normal variations of a general surface
$M\subset\mathcal{N}$ as above. The geometric identities
already established allow us to express the
linearisation of its mean-curvature vector in the frame
$(\nu,n)$.
For $V\in\Gamma(T^\perp_\iota M)$, fix an open set
$U\subset M$ whose closure in $M$ is compact.
There exists $r_U>0$ such that the maps
\begin{align*}
    G_r:U&\rightarrow\mathbb{S}^4,\\
    x&\mapsto\exp_{\iota(x)}^{\mathbb{S}^4}(rV(x))
\end{align*}
are smooth immersions for every $r\in(-r_U,r_U)$. Indeed, $G_0=\iota|_U$,
and the differential of $G_r$ depends smoothly on
$(x,r)$. Since $\iota$ is an immersion, compactness
of $\overline U$ ensures that this differential
remains injective for all $x\in U$ when $|r|$
is sufficiently small. Moreover,
\begin{align*}
    \left.\frac{\partial}{\partial r}G_r(x)\right|_{r=0}
    =V(x)
    \qquad\forall\,x\in U.
\end{align*}
Thus, these maps define a normal variation of $\iota$
near every point of $M$. If $M$ is compact, we may take $U=M$.
Define the \textit{normal linearisation of the mean curvature map} at $\iota$ to be the operator $J:\Gamma(T^\perp_\iota M)\to\Gamma(T^\perp_\iota M)$ given by
\begin{align*}
    (JV)(x):=\left(
    \left.\frac{D}{dr}\right|_{r=0}\mathbf{H}_{G_r}(x)
    \right)^\perp
    \qquad\forall\,x\in M,
\end{align*}
where $D/dr$ denotes covariant differentiation in
$\mathbb{S}^4$ along the geodesic $r\mapsto G_r(x)$ and the projection
is onto $(T^\perp_\iota M)_x$. The following normal variation formula \eqref{eq:s4-normal-variation-formula} expresses $J$ in
terms of the Laplacian associated with the normal
connection, given by
\begin{align}
    \label{eq:s4-delta-perp}
    \Delta^\perp V
    :=\sum_{i=1}^2\left(
    \nabla^\perp_{e_i}\nabla^\perp_{e_i}V
    -\nabla^\perp_{\nabla^M_{e_i}e_i}V\right)
    \qquad\forall\,V\in\Gamma(T^\perp_\iota M),
\end{align}
where $(e_1,e_2)$ is any local $g_M$-orthonormal frame
of $TM$. For a surface in the unit round sphere, this
formula reads
\begin{align}
    \label{eq:s4-normal-variation-formula}
    JV=\Delta^\perp V+2V+
    \sum_{i,j=1}^2
    \langle A_\iota(e_i,e_j),V\rangle A_\iota(e_i,e_j)
    \qquad\forall\,V\in\Gamma(T^\perp_\iota M);
\end{align}
see \cite[Lemma~3.1]{LagaceLynchCanonicalFoliation}.
This local identity applies without assuming that $M$
is minimal. The term $2V$ is the ambient curvature
contribution, with one copy of $V$ for each orthonormal
tangent direction.

In the next lemma, we compute the scalar operators representing $J$.
In particular, we compare its component in the direction
$\nu$ with the scalar mean-curvature linearisation for
$M\subset\mathcal{N}$.

\begin{lemma}
    \label{lem:s4-normal-linearisation}
    Define $L_{M/\mathcal{N}},B,B^*:C^\infty(M)\to
    C^\infty(M)$ and $q\in C^\infty(M)$ by
    \begin{align}
        L_{M/\mathcal{N}}u
        &:=\Delta_Mu+
        \bigl(|S_\jmath^\nu|^2+
        \operatorname{Ric}_{g_{\mathcal{N}}}(\nu,\nu)
        \bigr)u,
        \label{eq:s4-scalar-Jacobi}\\
        q&:=2+|S_\iota^n|^2-|\omega|^2,\nonumber\\
        Bb&:=-2\langle\omega,db\rangle
        -(\operatorname{div}_M\omega)b
        +\langle h_\iota^\nu,h_\iota^n\rangle b,\nonumber\\
        B^*u&:=2\langle\omega,du\rangle
        +(\operatorname{div}_M\omega)u
        +\langle h_\iota^\nu,h_\iota^n\rangle u\nonumber
    \end{align}
    for every $u,b\in C^\infty(M)$. Here
    \begin{align*}
        \operatorname{div}_M\omega
        &:=\sum_{i=1}^2(\nabla^M_{e_i}\omega)(e_i),\\
        \langle h_\iota^\nu,h_\iota^n\rangle
        &:=\sum_{i,j=1}^2h_\iota^\nu(e_i,e_j)h_\iota^n(e_i,e_j),
    \end{align*}
    where $(e_1,e_2)$ is any local $g_M$-orthonormal
    frame of $TM$.
    Then, for every $u,b\in C^\infty(M)$,
    \begin{align}
        \label{eq:s4-J-matrix}
        \begin{pmatrix}
            \langle J(u\nu+bn),\nu\rangle\\
            \langle J(u\nu+bn),n\rangle
        \end{pmatrix}
        =
        \begin{pmatrix}
            L_{M/\mathcal{N}}+H_{\iota,n}^2&B\\
            B^*&\Delta_M+q
        \end{pmatrix}
        \begin{pmatrix}u\\b\end{pmatrix},
    \end{align}
    where $H_{\iota,n}^2$ and $q$ act by multiplication.
    If $M$ is closed, then $B^*$ is the formal
    $L^2(M,d\operatorname{vol}_M)$ adjoint of $B$, where $d\operatorname{vol}_M$
    is the area measure induced by $g_M$.
\end{lemma}

For completeness, we provide the details underlying the proof of Lemma \ref{lem:s4-normal-linearisation} in Appendix~\ref{app:s4-geometric-identities}.

Finally, we construct an explicit parametrisation of a
neighbourhood of $\mathcal{N}$ in $\mathbb{S}^4$ by following the geodesics normal to $\mathcal{N}$.
\begin{lemma}
    \label{lem:s4-tubular}
    Choose $b_0\in(0,\pi/2)$ such that
    $\tan b_0<\cos z_0$. Define
    $\Psi:\mathbb{T}\times(-z_0,z_0)\times(-b_0,b_0)
    \to\mathbb{S}^4$ by
    \begin{align}
        \label{eq:s4-tubular-map}
        \Psi(s,t,z,b)
        :=\cos b\,\mathcal{F}(s,t,z)+\sin b\,N(s,t) \qquad\forall\,(s,t,z,b)\in\mathbb{T}\times(-z_0,z_0)\times(-b_0,b_0).
    \end{align}
    Then, $\Psi$ is a smooth diffeomorphism onto an open neighbourhood of $\mathcal{N}$ in $\mathbb{S}^4$.
    Moreover,
    \begin{align*}
        \Psi(s,t,z,b)
        =\exp^{\mathbb{S}^4}_{\mathcal{F}(s,t,z)}
        \bigl(b\,n(\mathcal{F}(s,t,z))\bigr)
    \end{align*}
    for every $(s,t,z,b)$ in its domain.
\end{lemma}
\begin{proof}
    Since $N(s,t)=n(\mathcal F(s,t,z))$ is a unit
    tangent vector to $\mathbb S^4$ at $\mathcal F(s,t,z)$,
    the formula for spherical geodesics gives
    \begin{align*}
        \Psi(s,t,z,b)
        =
        \exp_{\mathcal F(s,t,z)}^{\mathbb S^4}
        \bigl(b\,n(\mathcal F(s,t,z))\bigr).
    \end{align*}
    Thus, $\Psi$ takes values in $\mathbb S^4$
    and has the asserted normal exponential formula.

    To prove that $\Psi$ is a diffeomorphism onto its image, write
    \begin{align*}
        \Psi(s,t,z,b)=\bigl(r_1e^{i\sqrt2s},r_2e^{i\sqrt2t},
        \cos b\sin z\bigr)
    \end{align*}
    with
    \begin{align*}
        r_1:=\frac{\cos b\cos z+\sin b}{\sqrt2},
        \qquad
        r_2:=\frac{\cos b\cos z-\sin b}{\sqrt2}.
    \end{align*}
    The choice of $b_0$ ensures that
    $|\tan b|<\cos z$, so $r_1,r_2>0$.
    Consequently, a point $(w_1,w_2,x_5)$ in the
    image of $\Psi$ uniquely determines its preimage by
    \begin{align*}
        b&=\arcsin\left(\frac{|w_1|-|w_2|}{\sqrt2}\right),
        \qquad
        z=\arcsin\left(\frac{x_5}{\cos b}\right),\\
        s&=\frac{\arg w_1}{\sqrt2}\pmod{\ell},
        \qquad
        t=\frac{\arg w_2}{\sqrt2}\pmod{\ell}.
    \end{align*}
    Here, $\arcsin$ takes values in $(-\pi/2,\pi/2)$,
    and the arguments are taken modulo $2\pi$,
    so $s,t$ are well defined modulo $\ell$.
    These formulas remain smooth and take values
    in the parameter domain for points of $\mathbb S^4$
    sufficiently close to any image point. Direct substitution verifies that they define a local inverse of $\Psi$. Hence, $\Psi$ is injective and a local
    diffeomorphism, and therefore a diffeomorphism onto an open subset of $\mathbb S^4$. This subset contains $\mathcal N$, since $\Psi(s,t,z,0)=\mathcal F(s,t,z)$.
\end{proof}
\subsection{Symmetries of the construction}
We will construct the initial surfaces inside $\mathcal N$ by
joining two sheets on opposite sides of the central torus with
small necks. Each neck will be parametrised using a Euclidean catenoid.
The neck positions will be indexed by a square lattice
$L_m\subset\mathbb T$, defined below: for each $p\in L_m$,
we will insert one neck near the point $f(p)$ of the
central torus $f(\mathbb T)$. We will impose the same finite group of
symmetries on the initial surfaces and on the maps used to
deform them in $\mathbb S^4$.

The symmetry conditions serve two purposes in the
subsequent linear analysis. Invariance under lattice translations removes the
kernel of $L_{\mathbb T}=\Delta_{\mathbb T}+2$ on the functions
used to construct the two sheets. Near each neck, the other
symmetries impose different conditions on the two coefficients
of a normal variation. The coefficient in the direction $n$
has no constant angular term and satisfies an angular
Poincar\'e inequality with a constant independent of the neck
size. This inequality is used to estimate the correction in
the direction $n$, uniformly as the neck radii tend to zero.
The coefficient in the normal direction
within $\mathcal N$ can have a constant angular term; that
term must be treated in the analysis of deformations within
$\mathcal N$. We establish these facts below and verify that
the normal linearisation from the preceding subsection
preserves the selected symmetry conditions.

Fix an integer $m\geq2$, and define
\begin{align}
    \label{eq:s4-lattice}
    L_m:=
    \left\{\left(\frac{j\ell}{m},\frac{k\ell}{m}\right):
    j,k\in\{0,\ldots,m-1\}\right\}\subset\mathbb T.
\end{align}
The corresponding points of the central torus are $f(L_m)$.
Define the following isometries of $\mathbb S^4$:
\begin{align*}
    R_1(w_1,w_2,x_5)&:=(e^{2\pi i/m}w_1,w_2,x_5),\\
    R_2(w_1,w_2,x_5)&:=(w_1,e^{2\pi i/m}w_2,x_5),\\
    C_1(w_1,w_2,x_5)&:=(\overline{w_1},w_2,x_5),\\
    C_2(w_1,w_2,x_5)&:=(w_1,\overline{w_2},x_5),\\
    \sigma(w_1,w_2,x_5)&:=(w_2,w_1,x_5),\\
    Z(w_1,w_2,x_5)&:=(w_1,w_2,-x_5)
    \qquad\forall\,(w_1,w_2,x_5)\in\mathbb S^4.
\end{align*}
Set
\begin{align*}
    \Gamma_m^0:=\langle R_1,R_2,C_1,C_2,\sigma\rangle,
    \qquad
    \Gamma_m:=\langle\Gamma_m^0,Z\rangle.
\end{align*}
These groups are finite. Indeed, every element permutes the
two complex coordinates, possibly conjugates each of them,
and multiplies them by $m$th roots of unity; an element of
$\Gamma_m$ may also change the sign of the fifth coordinate.
There are only finitely many maps of this form.

The definition of $\mathcal F$ gives the following actions
on its parameters:
\begin{align*}
    \begin{array}{c|c}
        \text{isometry}&
        \text{parameters of its value at }\mathcal F(s,t,z)\\ \hline
        R_1&(s+\ell/m,t,z)\\
        R_2&(s,t+\ell/m,z)\\
        C_1&(-s,t,z)\\
        C_2&(s,-t,z)\\
        \sigma&(t,s,z)\\
        Z&(s,t,-z)
    \end{array}
    \qquad
    \forall\,(s,t,z)\in\mathbb T\times(-z_0,z_0).
\end{align*}
For example,
\begin{align*}
    R_1\bigl(\mathcal F(s,t,z)\bigr)
    &=
    \frac{\cos z}{\sqrt2}
    \bigl(e^{i\sqrt2(s+\ell/m)},e^{i\sqrt2t},0\bigr)
    +\sin z\,e_5=\mathcal F(s+\ell/m,t,z),
\end{align*}
because $\sqrt2\ell/m=2\pi/m$. The other entries follow
by conjugating or interchanging the displayed coordinates,
or by changing the sign of $\sin z$.
In particular, every $g\in\Gamma_m$ preserves $\mathcal N$
and the central torus. Write $\widehat g:\mathbb T\to\mathbb T$
for the map induced by $g\in\Gamma_m$ on the first two parameters, i.e.
\begin{align*}
    \hat g:=f^{-1}\circ (g|_{f(\mathbb{T})})\circ f. 
\end{align*}
Then
\begin{align*}
    \pi\circ g=\widehat g\circ\pi,
    \qquad
    \widehat g(L_m)=L_m
    \qquad\forall\,g\in\Gamma_m.
\end{align*}
The two translations induced by $R_1,R_2$ are denoted by
\begin{align*}
    T_s(s,t):=(s+\ell/m,t),\qquad
    T_t(s,t):=(s,t+\ell/m)
    \qquad\forall\,(s,t)\in\mathbb T.
\end{align*}
%

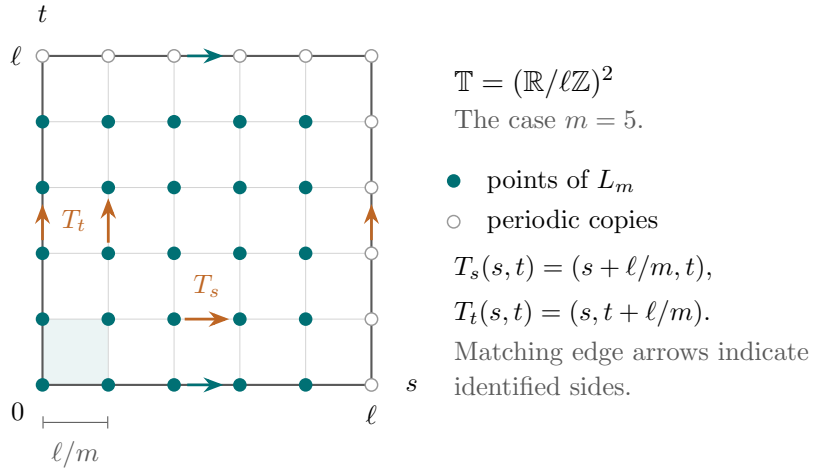
\begin{figure}[htbp]
    \centering
    \begingroup
    \definecolor{symteal}{RGB}{0,112,116}
    \definecolor{symorange}{RGB}{190,102,36}
    \begin{tikzpicture}[
        x=0.87cm,y=0.87cm,
        >=Stealth,
        every node/.style={font=\small},
        lattice point/.style={circle,fill=symteal,inner sep=1.8pt},
        repeated point/.style={circle,draw=black!40,fill=white,
                              line width=0.55pt,inner sep=1.65pt}
    ]
        \fill[symteal!8] (0,0) rectangle (1,1);
        \draw[step=1,black!17,line width=0.35pt] (0,0) grid (5,5);
        \draw[black!65,line width=0.8pt] (0,0) rectangle (5,5);
        \foreach \j in {0,...,4} {
            \foreach \k in {0,...,4} {
                \node[lattice point] at (\j,\k) {};
            }
        }
        \foreach \j in {0,...,4} {
            \node[repeated point] at (\j,5) {};
        }
        \foreach \k in {0,...,5} {
            \node[repeated point] at (5,\k) {};
        }
        \foreach \y in {0,5} {
            \draw[symteal,line width=1pt,->] (2.20,\y) -- (2.75,\y);
        }
        \foreach \x in {0,5} {
            \draw[symorange,line width=1pt,->] (\x,2.20) -- (\x,2.75);
        }
        \draw[symorange,line width=1.05pt,->,shorten <=4pt,shorten >=4pt]
            (2,1) -- (3,1) node[midway,above=4pt] {$T_s$};
        \draw[symorange,line width=1.05pt,->,shorten <=4pt,shorten >=4pt]
            (1,2) -- (1,3) node[midway,left=4pt] {$T_t$};
        \node[below left=3pt] at (0,0) {$0$};
        \node[below=4pt] at (5,0) {$\ell$};
        \node[left=4pt] at (0,5) {$\ell$};
        \node[right=9pt] at (5,0) {$s$};
        \node[above=9pt] at (0,5) {$t$};
        \draw[black!65,|-|] (0,-0.57) -- (1,-0.57)
            node[midway,below=3pt] {$\ell/m$};
        \node[anchor=west,font=\normalsize] at (6.1,4.60)
            {$\mathbb T=(\mathbb R/\ell\mathbb Z)^2$};
        \node[anchor=west,text=black!65] at (6.1,4.05)
            {The case $m=5$.};
        \node[lattice point] at (6.25,3.10) {};
        \node[anchor=west] at (6.60,3.10) {points of $L_m$};
        \node[repeated point] at (6.25,2.48) {};
        \node[anchor=west] at (6.60,2.48) {periodic copies};
        \node[anchor=west,align=left] at (6.1,1.43)
            {$T_s(s,t)=(s+\ell/m,t)$,\\[4pt]
             $T_t(s,t)=(s,t+\ell/m)$.};
        \node[anchor=west,align=left,text=black!65] at (6.1,0.25)
            {Matching edge arrows indicate\\identified sides.};
    \end{tikzpicture}
    \endgroup
    \caption{The lattice $L_m$ and its elementary translations.
    Opposite edges of the square are identified.}
    \label{fig:s4-square-lattice}
\end{figure}

We next show that the scalar Jacobi operator of the central
torus has trivial kernel on the space of smooth functions
invariant under the lattice translations.
For integers $a,b$, define the complex-valued Fourier
function
\begin{align*}
    e_{a,b}(s,t):=
    \exp\left(\frac{2\pi i}{\ell}(as+bt)\right)
    \qquad\forall\,(s,t)\in\mathbb T.
\end{align*}
Since $\ell^2=2\pi^2$, differentiation gives
\begin{align*}
    \Delta_{\mathbb T}e_{a,b}
    &=-2(a^2+b^2)e_{a,b},\\
    L_{\mathbb T}e_{a,b}
    &=2(1-a^2-b^2)e_{a,b}
    \qquad\forall\,(a,b)\in\mathbb Z^2.
\end{align*}
Moreover,
\begin{align*}
    e_{a,b}\circ T_s=e^{2\pi ia/m}e_{a,b},
    \qquad
    e_{a,b}\circ T_t=e^{2\pi ib/m}e_{a,b}.
\end{align*}
Uniqueness of Fourier coefficients therefore implies that
a smooth function invariant under both translations can
have a nonzero coefficient at $(a,b)$ only if $a,b\in m\mathbb Z$.
For $m\geq2$, no such pair satisfies $a^2+b^2=1$. Hence
\begin{align}
    \label{eq:s4-translation-kernel}
    \left.
    \begin{gathered}
        v\in C^\infty(\mathbb T),\qquad
        v\circ T_s=v,\qquad v\circ T_t=v,\\
        L_{\mathbb T}v=0
    \end{gathered}
    \right\}
    \quad\rightarrow\quad v=0.
\end{align}
The constant functions are permitted by these symmetries,
but they are not in the kernel, since
$L_{\mathbb T}1=2$. The later construction of the sheet
heights uses the operator $L_{\mathbb T}$ on this
translation-invariant space.

We next determine how the group $\Gamma_m$ acts on the two normal
directions of a surface in $\mathcal N$.
Every $g\in\Gamma_m$ either preserves the two complex
coordinate planes individually or interchanges them. Define
\begin{align*}
    \chi(g):=
    \begin{cases}
        +1,&\text{if $g$ preserves the two planes individually},\\
        -1,&\text{if $g$ interchanges the two planes}
    \end{cases}
    \qquad\forall\,g\in\Gamma_m.
\end{align*}
Interchanging the planes twice preserves them individually,
so $\chi(gh)=\chi(g)\chi(h)$ for every $g,h\in\Gamma_m$.
Thus $\chi:\Gamma_m\to\{1,-1\}$ is a group homomorphism,
also called a \textit{character}. It equals $-1$ on $\sigma$ and
$+1$ on each of the other five generators.

The value $\chi(g)$ determines whether the differential
of $g$ preserves or reverses the chosen unit normal field $n$:
it sends $n(x)$ to $n(g(x))$ when $\chi(g)=+1$, and to
$-n(g(x))$ when $\chi(g)=-1$, for every $x\in\mathcal N$.
For every isometry $g:\mathbb S^4\to\mathbb S^4$ and
$x\in\mathbb S^4$, let $dg_x:T_x\mathbb S^4\rightarrow T_{g(x)}\mathbb S^4$ denote its differential at $x$. We first compute how the differentials of the six generators transform the normal field $n$. Fix
$(s,t,z)\in\mathbb T\times(-z_0,z_0)$ and set
$x:=\mathcal F(s,t,z)$. Since $n(x)=N(s,t)$,
differentiating the coordinate formulas for the generators gives
\begin{align*}
    (dR_1)_x(n(x))
    &=N(s+\ell/m,t)=n(R_1(x)),\\
    (dR_2)_x(n(x))
    &=N(s,t+\ell/m)=n(R_2(x)),\\
    (dC_1)_x(n(x))
    &=N(-s,t)=n(C_1(x)),\\
    (dC_2)_x(n(x))
    &=N(s,-t)=n(C_2(x)),\\
    (d\sigma)_x(n(x))
    &=-N(t,s)=-n(\sigma(x))\\
    (dZ)_x(n(x))
    &=N(s,t)=n(Z(x)).\\
\end{align*}
Since $\chi(\sigma)=-1$ and $\chi$ equals $+1$ on the
other five generators, these computations prove
\begin{align*}
    dg_x(n(x))=\chi(g)n(g(x))
    \qquad
    \forall\,g\in\{R_1,R_2,C_1,C_2,\sigma,Z\},
    \quad x\in\mathcal N.
\end{align*}
Indeed, $\chi(\sigma)=-1$, whereas $\chi$ equals $+1$
on each of the other five generators. We now show that the same identity holds for every element of $\Gamma_m$. Suppose that it holds for two elements
$g,h\in\Gamma_m$. Applying first $h$ and then $g$, the
chain rule gives
\begin{align*}
    d(g\circ h)_x(n(x))
    &=dg_{h(x)}\bigl(dh_x(n(x))\bigr)\\
    &=\chi(h)\,dg_{h(x)}(n(h(x)))\\
    &=\chi(h)\chi(g)\,n(g(h(x)))\\
    &=\chi(g\circ h)\,n((g\circ h)(x))
    \qquad\forall\,x\in\mathcal N,
\end{align*}
where the last equality uses that $\chi$ is a group
homomorphism. Thus the identity is preserved under composition.
Moreover, $R_i^{-1}=R_i^{m-1}$ for $i\in\{1,2\}$, and
each of $C_1,C_2,\sigma,Z$ is its own inverse. Every element
of $\Gamma_m$ can therefore be written as a finite composition
of these six generators. We conclude that
\begin{align}
    \label{eq:s4-character}
    dg_x(n(x))=\chi(g)n(g(x))
    \qquad\forall\,g\in\Gamma_m,\quad x\in\mathcal N.
\end{align}

Let $\Omega\subset\mathcal N$ be an open set with smooth
boundary $M:=\partial_{\mathcal N}\Omega$, and assume
$g(\Omega)=\Omega$ for every $g\in\Gamma_m$. Use the
notation of the preceding subsection for the inclusions
$\jmath:M\hookrightarrow\mathcal N$ and
$\iota:M\hookrightarrow\mathbb S^4$, their induced metric
$g_M$, and their second fundamental forms. Choose $\nu$
to be the outward unit normal to $\Omega$ within
$\mathcal N$. Each $g$ preserves $M$, the metric, and
the side occupied by $\Omega$. Therefore,
\begin{align*}
    d g_x\bigl(\nu(x)\bigr)=\nu(g(x))
    \qquad\forall\,g\in\Gamma_m,\,\forall\,x\in M.
\end{align*}
For $k\in\mathbb Z_{\geq0}$ and $\beta\in(0,1)$, define
\begin{align}
    \label{eq:s4-symmetry-spaces}
    \begin{split}
        C^{k,\beta}_{\mathrm{inv}}(M)
        &:=
        \{u\in C^{k,\beta}(M):
        u\circ g=u\ \text{for every }g\in\Gamma_m\},\\
        C^{k,\beta}_{\chi}(M)
        &:=
        \{b\in C^{k,\beta}(M):
        b\circ g=\chi(g)b\ \text{for every }g\in\Gamma_m\}
    \end{split}
\end{align}
to be the spaces of the \textit{invariant} and $\chi$-\textit{equivariant} functions on $M$, respectively. In these compositions, $g$ denotes its restriction to $M$.
The same subscripts on $C^\infty(M)$ impose the same
conditions on smooth functions.
If $V=u\nu+bn$ is a normal vector field along $\iota$, then
\begin{align*}
    d g_x\bigl(V(x)\bigr)
    &=u(x)\nu(g(x))+\chi(g)b(x)n(g(x)),\\
    V(g(x))
    &=u(g(x))\nu(g(x))+b(g(x))n(g(x))
\end{align*}
for every $g\in\Gamma_m$ and $x\in M$. The two vectors
$\nu(g(x)),n(g(x))$ are linearly independent. Thus the
condition
\begin{align*}
    d g_x\bigl(V(x)\bigr)=V(g(x))
    \qquad\forall\,g\in\Gamma_m,\,\forall\,x\in M,
\end{align*}
which is the definition of \textit{equivariance} for a vector field $V\in\Gamma(T^\perp_\iota M)$, is
equivalent to $u\circ g=u$ and $b\circ g=\chi(g)b$.

An ambient isometry commutes with the exponential map.
Hence an equivariant normal field $V$ gives
\begin{align*}
    \exp_{\iota(g(x))}^{\mathbb S^4}(rV(g(x)))
       =g\bigl(\exp_{\iota(x)}^{\mathbb S^4}(rV(x))\bigr)
    \qquad\forall\,g\in\Gamma_m,\quad x\in M,\quad r\in\mathbb R.
\end{align*}
The maps are defined for all $r$ because the sphere is
complete; the immersion property holds on the small
intervals described above. Their images are invariant.
The same identity holds for $\exp^{\mathcal N}$ wherever
defined, because the group preserves $g_{\mathcal N}$.

The linearised operator respects the same conditions.
Indeed, invariance of the ambient connection and the
transformation rules for the two normals give
\begin{align*}
    g^*h_\iota^\nu=h_\iota^\nu,\qquad
    g^*h_\iota^n=\chi(g)h_\iota^n,\qquad
    g^*\omega=\chi(g)\omega
    \qquad\forall\,g\in\Gamma_m.
\end{align*}
Here pullback means evaluating a tensor at $g(x)$ on
the vectors $dg_x(X),dg_x(Y)$. Taking traces gives
$H_{\iota,\nu}\circ g=H_{\iota,\nu}$ and
$H_{\iota,n}\circ g=\chi(g)H_{\iota,n}$.
Isometries commute with covariant differentiation,
contraction and the Laplacian. Thus both diagonal
operators in \eqref{eq:s4-J-matrix} preserve the
corresponding function spaces, and
\begin{align*}
    (Bv)\circ g=\chi(g)B(v\circ g),\qquad
    (B^*v)\circ g=\chi(g)B^*(v\circ g)
    \qquad\forall\,v\in C^\infty(M),\quad g\in\Gamma_m.
\end{align*}
In particular, $B$ maps $C^\infty_\chi(M)$ to
$C^\infty_{\mathrm{inv}}(M)$, while $B^*$ maps in
the opposite direction. Consequently $J$ takes an
equivariant normal field to an equivariant normal
field, and its two equations restrict to these spaces.

We finish by describing the consequences near a neck.
Choose $\tau>0$ and $\xi_*>0$ such that
\begin{align}\label{Equation: conditions symmetries}
    \tau\cosh\xi_*<\frac{\ell}{2m},
    \qquad
    \tau\xi_*<z_0.
\end{align}
For each $p=(p_s,p_t)\in L_m$, define
\begin{align}
    \label{eq:s4-symmetry-neck-map}
    \begin{split}
        F_{\tau,p}:
        (-\xi_*,\xi_*)\times(\mathbb R/2\pi\mathbb Z)
        &\rightarrow\mathcal N,\\
        (\xi,\theta)
        &\mapsto
        \mathcal F\bigl(
        p+\tau\cosh\xi(\cos\theta,\sin\theta),\tau\xi
        \bigr).
    \end{split}
\end{align}
The first argument of $\mathcal F$ is computed in the
flat torus $\mathbb T$. The first inequality in \eqref{Equation: conditions symmetries} places it
in an embedded coordinate disk about $p$ containing no
other lattice point, and the second keeps the height
inside $(-z_0,z_0)$. In these coordinates, the map is the usual catenoid parametrisation, followed by $\mathcal F$. Its coordinate derivatives before applying $\mathcal F$
are
\begin{align*}
    \partial_\xi
    \bigl(p+\tau\cosh\xi(\cos\theta,\sin\theta),\tau\xi\bigr)
    &=
    \bigl(\tau\sinh\xi\cos\theta,
    \tau\sinh\xi\sin\theta,\tau\bigr),\\
    \partial_\theta
    \bigl(p+\tau\cosh\xi(\cos\theta,\sin\theta),\tau\xi\bigr)
    &=
    \bigl(-\tau\cosh\xi\sin\theta,
    \tau\cosh\xi\cos\theta,0\bigr).
\end{align*}
These two coordinate derivatives are linearly independent:
the last component of the first is $\tau>0$, whereas
the last component of the second is zero and the second
vector is nonzero. For a point $x\in\mathcal N$, write
\begin{align*}
    \mathcal F^{-1}(x)=(q,z),
    \qquad q=\pi(x)\in\mathbb T,\quad z\in(-z_0,z_0).
\end{align*}
and recall that $z$ is the signed distance coordinate
from the central torus within $(\mathcal N,g_{\mathcal N})$;
see Remark~\ref{rem:s4-signed-distance}. Therefore, we say that $z$ is the \textit{height} of $x$. For $x=F_{\tau,p}(\xi,\theta)$, the defining formula gives
\begin{align*}
    z=\tau\xi,
    \qquad
    q=p+\tau\cosh\xi(\cos\theta,\sin\theta).
\end{align*}
In particular, we recover the parameters $(\xi,\theta)$ from $x$ by
\begin{align*}
    \xi=\frac z\tau,
    \qquad
    (\cos\theta,\sin\theta)=\frac{q-p}{|q-p|},
\end{align*}
where $q-p$ is computed in the flat coordinate disk
about $p$. The denominator is nonzero because
$|q-p|=\tau\cosh\xi\geq\tau>0$.
These formulas give a smooth inverse on the image,
with $\theta$ understood modulo $2\pi$.
Together with the linear independence of the coordinate
derivatives and the fact that $\mathcal F$ is an embedding,
this proves that $F_{\tau,p}$ is an embedding. We call the image of $F_{\tau,p}$ the
\textit{catenoidal model neck} at $p$. Note that a catenoidal model need not be minimal in $(\mathcal N,g_{\mathcal N})$.

For each fixed $\xi\in(-\xi_*,\xi_*)$, varying $\theta$
gives a closed curve at height $z=\tau\xi$.
Its projection under $\pi$ is the circle in the flat
torus centred at $p$ with radius
\begin{align*}
    r(\xi):=\tau\cosh\xi.
\end{align*}
Since $\cosh\xi\geq1$, with equality only at $\xi=0$,
this projected radius is smallest when $\xi=0$.
We call the corresponding closed curve the
\textit{waist} of the neck. Explicitly,
\begin{align*}
    F_{\tau,p}(0,\theta)
    =f\bigl(p+\tau(\cos\theta,\sin\theta)\bigr)
    \qquad\forall\,\theta\in\mathbb R/2\pi\mathbb Z.
\end{align*}
Thus the waist lies in the central torus, and its
projection is the circle of radius $\tau$ about $p$.
The portions with $\xi>0$ and $\xi<0$ lie on opposite
sides of the central torus and meet along the waist.

Only group elements fixing $p$ on the base torus
preserve this individual neck. Denote this subgroup by
\begin{align*}
    \Gamma_{m,p}
    :=
    \{g\in\Gamma_m:\widehat g(p)=p\}.
\end{align*}
For $p=(j\ell/m,k\ell/m)\in L_m$, set
$T_p:=R_1^jR_2^k$. Its induced translation sends $0$
to $p$. Define
\begin{align*}
    C_{2,p}:=T_pC_2T_p^{-1},\qquad
    \sigma_p:=T_p\sigma T_p^{-1},\qquad
    Q_p:=\sigma_pC_{2,p}.
\end{align*}
These maps and $Z$ belong to $\Gamma_{m,p}$. At $p=0$,
$C_2$ changes $(r\cos\theta,r\sin\theta)$ to
$(r\cos\theta,-r\sin\theta)$, whereas $\sigma$
changes it to $(r\sin\theta,r\cos\theta)$.
Their composition $\sigma C_2$ changes it to
$(-r\sin\theta,r\cos\theta)$.
Conjugation by $T_p$ gives the same actions on the
displacement from any $p\in L_m$.
Finally, $\cosh(-\xi)=\cosh\xi$, so $Z$ changes
$\xi$ to $-\xi$ without changing $\theta$.
We have therefore proved the following table:
\begin{align*}
    \begin{array}{c|c|c}
        g&\text{map }\vartheta_g
        \text{ satisfying }g\circ F_{\tau,p}
        =F_{\tau,p}\circ\vartheta_g&\chi(g)\\ \hline
        C_{2,p}&(\xi,\theta)\mapsto(\xi,-\theta)&+1\\
        \sigma_p&(\xi,\theta)\mapsto(\xi,\pi/2-\theta)&-1\\
        Q_p&(\xi,\theta)\mapsto(\xi,\theta+\pi/2)&-1\\
        Z&(\xi,\theta)\mapsto(-\xi,\theta)&+1
    \end{array}
\end{align*}
All angular expressions are understood modulo $2\pi$.
The signs are unchanged by conjugation, because
$\chi(T_p)=1$.

Let $b$ be a smooth function on this model neck
satisfying $b\circ g=\chi(g)b$ for $g\in\Gamma_{m,p}$.
For the following calculations, write $b(\xi,\theta)$
for its pullback by $F_{\tau,p}$. The table above gives
\begin{align*}
    b(\xi,-\theta)&=b(\xi,\theta),\\
    b(\xi,\pi/2-\theta)&=-b(\xi,\theta),\\
    b(\xi,\theta+\pi/2)&=-b(\xi,\theta),\\
    b(-\xi,\theta)&=b(\xi,\theta)
\end{align*}
for every $\xi\in(-\xi_*,\xi_*)$ and
$\theta\in\mathbb R/2\pi\mathbb Z$.
To determine the angular Fourier terms, write
\begin{align*}
    b(\xi,\theta)=\sum_{k\in\mathbb Z}
    \widehat b_k(\xi)e^{ik\theta},
    \qquad
    \widehat b_k(\xi):=
    \frac1{2\pi}\int_0^{2\pi}
    b(\xi,\theta)e^{-ik\theta}\,d\theta.
\end{align*}
The quarter-turn identity implies
$(e^{ik\pi/2}+1)\widehat b_k(\xi)=0$.
Thus, a nonzero coefficient requires $k\equiv2\pmod4$.
The reflection $\theta\mapsto-\theta$ makes the
coefficients at $k$ and $-k$ equal, removing the sine
terms. The reflection in $\xi$ makes each remaining
coefficient even. Hence
\begin{align}
    \label{eq:s4-neck-modes}
    b(\xi,\theta)=
    \sum_{j\geq0}b_j(\xi)\cos((4j+2)\theta),
    \qquad
    b_j(-\xi)=b_j(\xi)
\end{align}
for every $(\xi,\theta)$ in the parameter domain.
Here
\begin{align*}
    b_j(\xi):=\frac1\pi\int_0^{2\pi}
    b(\xi,\theta)\cos((4j+2)\theta)\,d\theta
    \qquad\forall\,j\in\mathbb Z_{\geq0},\quad
    \xi\in(-\xi_*,\xi_*).
\end{align*}
The Fourier series and its derivatives converge on
compact subsets of this domain, since $b$ is smooth.

Parseval's identity and $(4j+2)^2\geq4$ give
\begin{align}
    \label{eq:s4-symmetry-angular-poincare}
    \int_0^{2\pi}|\partial_\theta b(\xi,\theta)|^2\,d\theta
    \geq
    4\int_0^{2\pi}|b(\xi,\theta)|^2\,d\theta
    \qquad\forall\,\xi\in(-\xi_*,\xi_*).
\end{align}
The constant $4$ depends only on the permitted angular
orders. In particular, it is independent of $\tau$
and $\xi_*$.

Evenness across the axis and waist, and oddness across
the diagonal, give the boundary conditions
\begin{align*}
    \partial_\theta b(\xi,0)=0, \quad b(\xi,\pi/4)=0 \qquad \forall\,\xi\in[0,\xi_*)
\end{align*}
and
\begin{align*}
    \partial_\xi b(0,\theta)=0 \qquad\forall\,\theta\in[0,\pi/4].
\end{align*}
The zero-value condition is a homogeneous Dirichlet
condition. The zero normal-derivative conditions are
homogeneous Neumann conditions. To relate the displayed
coordinate derivatives to normal derivatives, direct
differentiation of $F_{\tau,p}$ gives
\begin{align*}
    F_{\tau,p}^*g_{\mathcal N}
    =
    \tau^2\bigl(1+\cos^2(\tau\xi)\sinh^2\xi\bigr)\,d\xi^2
    +
    \tau^2\cos^2(\tau\xi)\cosh^2\xi\,d\theta^2.
\end{align*}
This metric is diagonal and positive definite.
Consequently, $\partial_\theta$ is perpendicular to
each angular boundary and $\partial_\xi$ is perpendicular
to the boundary $\xi=0$. These boundary conditions
follow from restrictions of smooth functions with the
stated symmetries. We impose no boundary condition at
the artificial truncation $\xi=\xi_*$.
Figure~\ref{fig:s4-neck-symmetries} illustrates the
three reflection conditions.

For comparison, let $u$ be a smooth function on the
model neck invariant under $\Gamma_{m,p}$, and again use its pullback by
$F_{\tau,p}$. Its quarter-turn identity is
$u(\xi,\theta+\pi/2)=u(\xi,\theta)$.
The Fourier calculation above now permits only
$k\equiv0\pmod4$. The two evenness conditions give
\begin{align*}
    u(\xi,\theta)
    =u_0(\xi)+\sum_{j\geq1}u_j(\xi)\cos(4j\theta),
    \qquad
    u_j(-\xi)=u_j(\xi)
    \quad\forall\,j\in\mathbb Z_{\geq0},
\end{align*}
where
\begin{align*}
    u_0(\xi)&:=\frac1{2\pi}\int_0^{2\pi}u(\xi,\theta)\,d\theta,\\
    u_j(\xi)&:=\frac1\pi\int_0^{2\pi}
    u(\xi,\theta)\cos(4j\theta)\,d\theta
    \qquad\forall\,j\geq1,\quad \xi\in(-\xi_*,\xi_*).
\end{align*}
The expansion holds for every $(\xi,\theta)$ in the
parameter domain.
Invariance under $\theta\mapsto\pi/2-\theta$ gives
zero angular derivative at $\theta=\pi/4$, by
differentiating that identity at its fixed point.
The other two reflections are the same as for $b$.
Thus
\begin{align*}
    \partial_\theta u(\xi,0)=0, \quad \partial_\theta u(\xi,\pi/4)&=0 \qquad\forall\,\xi\in[0,\xi_*)
\end{align*}
and
\begin{align*}
    \partial_\xi u(0,\theta)&=0 \qquad\forall\,\theta\in[0,\pi/4].
\end{align*}
Both $u$ and $b$ therefore have zero coefficients of
$\cos\theta$ and $\sin\theta$, which are the first
angular harmonics. However, $u_0$ need not vanish.
For example, a nonzero constant function satisfies
all the conditions on $u$ and has zero angular derivative,
so no inequality of the form
\eqref{eq:s4-symmetry-angular-poincare} can hold for all
invariant $u$.
This explains why the later estimates for the correction
in the direction $n$ use angular Poincar\'e inequalities,
whereas the analysis within $\mathcal N$ must also handle
the angular average $u_0$. The model calculation here
concerns a single catenoidal neck; the later analysis
also treats its attachment to the two sheets.

\begin{figure}[htbp]
    \centering
    \begingroup
    \definecolor{symteal}{RGB}{0,112,116}
    \definecolor{symorange}{RGB}{190,102,36}
    \begin{tikzpicture}[
        x=1cm,y=1cm,>=Stealth,
        every node/.style={font=\small},
        axis/.style={black!40,line width=0.55pt},
        symmetry axis/.style={dashed,line width=0.7pt}
    ]
        \begin{scope}[shift={(0,0)}]
            \node[font=\normalsize] at (0,2.75) {Angular action};
            \fill[symteal!11] (0,0) -- (1.7,0)
                arc[start angle=0,end angle=45,radius=1.7] -- cycle;
            \draw[black!45,line width=0.75pt] (0,0) circle[radius=1.7];
            \draw[axis,->] (0,-1.95) -- (0,2.05)
                node[above left=1pt] {$t-p_t$};
            \draw[symmetry axis,symteal] (-2.10,0) -- (2.10,0);
            \draw[symmetry axis,symorange] (-1.45,-1.45) -- (1.56,1.56);
            \node[below=4pt] at (2.02,0) {$s-p_s$};
            \draw[symteal,line width=1.2pt] (0,0) -- (1.7,0);
            \draw[symorange,line width=1.2pt] (0,0) -- (1.2021,1.2021);
            \fill[black!75] (0,0) circle[radius=1.6pt];
            \node[text=symteal,above=5pt] at (1.55,0)
                {$\mathrm N$};
            \node[text=symorange,right=4pt] at (1.22,1.22)
                {$\mathrm D$};
            \node[text=symteal] at (1.02,0.43) {$\mathcal S$};
            \draw[black!60,->] (0.58,0)
                arc[start angle=0,end angle=45,radius=0.58];
            \node[font=\scriptsize] at (0.77,0.18) {$\pi/4$};
            \draw[black!65,line width=0.85pt,->]
                ({2.14*cos(25)},{2.14*sin(25)})
                arc[start angle=25,end angle=115,radius=2.14];
            \node[fill=white,inner sep=2pt] at (0.85,2.22)
                {$Q_p:\ \chi=-1$};
            \node[anchor=west,text=symteal] at (-2.35,-2.26)
                {$C_{2,p}:\ \theta\mapsto-\theta,\quad\chi=+1$};
            \node[anchor=west,text=symorange] at (-2.35,-2.76)
                {$\sigma_p:\ \theta\mapsto\pi/2-\theta,\quad\chi=-1$};
            \node[anchor=west] at (-2.35,-3.26)
                {$Q_p:\ \theta\mapsto\theta+\pi/2$.};
        \end{scope}
        \begin{scope}[shift={(7.3,0)}]
            \node[font=\normalsize] at (0,2.75) {Axial action};
            \fill[symteal!8]
                (0.60,0) -- plot[domain=0:1.65,samples=45] ({0.60*cosh(\x)},\x)
                -- ({-0.60*cosh(1.65)},1.65)
                -- plot[domain=1.65:0,samples=45] ({-0.60*cosh(\x)},\x)
                -- cycle;
            \draw[axis,->] (0,-2.02) -- (0,2.16)
                node[above] {$\xi$};
            \draw[black!65,line width=0.9pt]
                plot[domain=-1.65:1.65,samples=70] ({0.60*cosh(\x)},\x);
            \draw[black!65,line width=0.9pt]
                plot[domain=-1.65:1.65,samples=70] ({-0.60*cosh(\x)},\x);
            \draw[black!40] (0,1.65)
                ellipse[x radius=1.6196,y radius=0.23];
            \draw[black!40] (0,-1.65)
                ellipse[x radius=1.6196,y radius=0.23];
            \draw[symteal,line width=0.85pt] (0,0)
                ellipse[x radius=0.60,y radius=0.11];
            \draw[symteal,symmetry axis] (-1.95,0) -- (1.95,0);
            \node[anchor=east,text=symteal,align=right]
                at (-0.82,0.30) {waist: $\xi=0$\\[-1pt] $\mathrm N$};
            \draw[black!65,line width=0.85pt,<->]
                (1.62,1.05) to[bend left=38] (1.62,-1.05);
            \node[anchor=west,align=left] at (2.10,0)
                {$Z$\\[2pt] $\chi=+1$};
            \node at (0,-2.26) {$Z:\ \xi\mapsto-\xi$};
            \node at (0,-2.76) {$b(-\xi,\theta)=b(\xi,\theta)$};
            \node[text=black!65] at (0,-3.26)
                {$r=\tau\cosh\xi,\qquad z=\tau\xi$.};
        \end{scope}
    \end{tikzpicture}
    \endgroup
    \caption{Symmetry conditions for $b$: Neumann data on the axis
    $\theta=0$ and waist $\xi=0$, and Dirichlet data on the
    diagonal $\theta=\pi/4$. Shading marks the chosen sector
    and upper half-neck.}
    \label{fig:s4-neck-symmetries}
\end{figure}
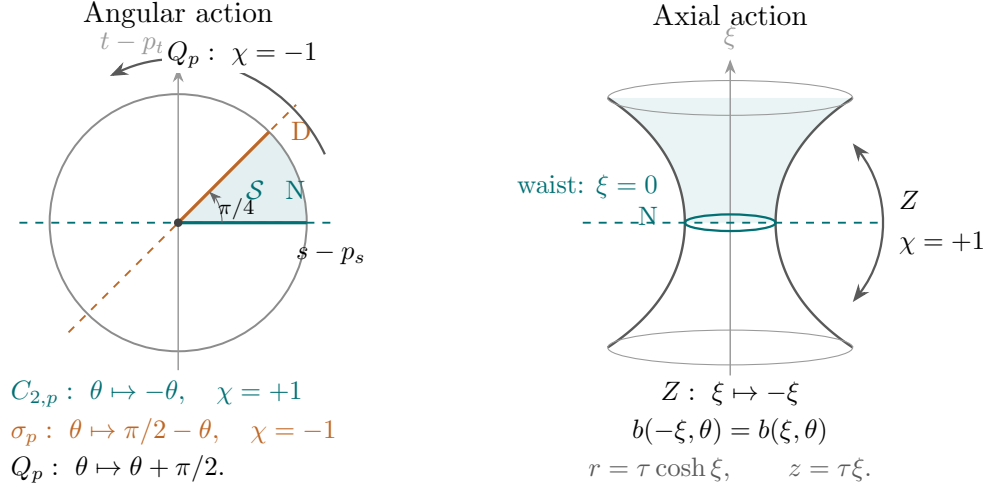
\section{Construction of the equivariant initial surfaces}
\label{sec:s4-scalar-family}

For each sufficiently large integer $m$ and each
$\zeta\in[-1,1]$, we construct a smooth closed embedded
surface $M_\zeta\subset\mathcal N$ which consists of two
sheets on opposite sides of the central torus, joined by
one catenoidal neck near each of the $m^2$ lattice points
in $L_m$. We then obtain the estimates that we will use to correct both components of the mean-curvature vector of $M_{\zeta}$ within $\mathbb S^4$, and to prove that the final surfaces have strictly negative Gaussian curvature. Throughout this section, scalar mean curvature and normal displacements of a surface refer to its inclusion into $(\mathcal N,g_{\mathcal N})$.

We first construct and estimate a Green function on a fixed
flat torus. After rescaling, it prescribes the height of the
sheets away from the lattice points. Its logarithmic terms
agree with those of the ends of a catenoid; comparing the
constant terms determines the waist radius $\tau_\zeta$.
We subtract the remaining constant difference and interpolate
between these graph heights and the catenoids. The resulting
surface is denoted by $M_\zeta$, and we prove its embeddedness,
topology, and geometric estimates directly.

In the last two subsections, we define the function spaces and
establish the scalar estimates used later. We identify all
$M_\zeta$ with a fixed reference surface $M_0$ so that the
unknown displacement functions have a fixed domain when $\zeta$ varies. Using the linear construction of Kapouleas--McGrath, we solve the scalar linearised equation up to a multiple of one specified function. We retain that coefficient and estimate the nonlinear error for every allowed scalar displacement. Section~\ref{sec:s4-transverse} will add a displacement normal to $\mathcal N$, and Section~\ref{sec:s4-coupled} will choose the scalar displacement and $\zeta$ so that both components of the mean-curvature vector of $M_{\zeta}$ within $\mathbb S^4$ vanish. No exact minimality is imposed in the present section.

All constants below are independent of $m$ and
$\zeta\in[-1,1]$ after the fixed cutoffs and exponents have
been chosen. Their dependence on a fixed derivative order
or H\"older exponent is allowed. Every assertion for
sufficiently large $m$ uses a threshold independent of
$\zeta$; we may increase that threshold finitely many times.

\subsection{A Green function on the fixed square torus}

We first construct a function with one logarithmic singularity on a fixed flat torus. Rescaling this function will give a solution of $L_{\mathbb T}\Phi_m=0$ away from $L_m$, with the same logarithmic singularity at every lattice point. We will need estimates both for this function and for its derivatives to construct the initial surfaces. We also prove a lower bound for its Hessian; this bound will later be used to establish negative Gaussian curvature between the necks.

Let $\mathbb T_0:=(\mathbb R/2\pi\mathbb Z)^2$, equipped with the flat metric $dX_1^2+dX_2^2$. Write $dX$ for its area measure, $D_0$ for its flat covariant derivative, and $\Delta_0$ for its Laplace--Beltrami operator. Thus, $\operatorname{Area}(\mathbb T_0)=4\pi^2$. Let $\delta_0$ be the Dirac measure of mass one at the origin of $\mathbb T_0$. For every $a\in[0,1/4]$, define
\begin{align}
    \label{eq:s4-cell-Green}
    G_a(X):=\sum_{n\in\mathbb Z^2\smallsetminus\{0\}}\frac{e^{in\cdot X}}{a-|n|^2}.
\end{align}
Here, $X=(X_1,X_2)$ and $n\cdot X:=n_1X_1+n_2X_2$ for every $n=(n_1,n_2)\in\mathbb Z^2$. The series on the right-hand side of \eqref{eq:s4-cell-Green} converges in $L^2(\mathbb T_0)$, because its coefficients are square summable. For $v\in L^1(\mathbb T_0)$, we define its Fourier coefficients by
\begin{align*}
    \widehat v(n):=\frac1{4\pi^2}\int_{\mathbb T_0}v(X)e^{-in\cdot X}\,dX \qquad\forall\,n\in\mathbb Z^2.
\end{align*}
With this convention, $4\pi^2\delta_0=\sum_{n\in\mathbb Z^2}e^{in\cdot X}$ in the sense of distributions on $\mathbb{T}_0$. Multiplication of each Fourier coefficient in \eqref{eq:s4-cell-Green} by $a-|n|^2$ therefore gives
\begin{align}
    \label{eq:s4-cell-equation}
    \begin{cases}
        (\Delta_0+a)G_a
        =4\pi^2\delta_0-1 & \text{ distributionally on }\mathbb{T}_0,\\[2mm]
        \displaystyle{\int_{\mathbb T_0}G_a\,dX=0.}
    \end{cases}
\end{align}
Note that the conditions in \eqref{eq:s4-cell-equation} uniquely determine $G_a$. As the Fourier coefficients of $G_a$ at $n$ and $-n$ are equal and real, $G_a$ is real valued. Moreover, since the Fourier coefficients of $G_a$  depend only on $|n|^2$, $G_a$ is invariant under both coordinate reflections and interchange of the coordinates. Finally, by standard elliptic regularity, we have $G_a\in C^\infty(\mathbb T_0\smallsetminus\{0\})$.

Here and throughout, we denote by $d_{\mathbb T_0}$ and $d_{\mathbb T}$ the geodesic distances induced by the flat metrics, and by $B_{\mathbb T_0}(p,r)$ and $B_{\mathbb T}(p,r)$ the corresponding open metric balls. All tensor norms on either torus are induced by the flat metric. To state estimates near the singularity $0\in\mathbb T_0$ of $G_a$, we introduce suitably rescaled H\"older norms on dyadic annuli. For $0<R<\pi/2$, set
\begin{align*}
    \mathcal A_R
    :=\{q\in\mathbb T_0\mbox{ : }
    R/2<d_{\mathbb T_0}(q,0)<2R\}.
\end{align*}
For a nonnegative integer $k$, $\beta\in(0,1)$, $R>0$,
an open subset $U\subset\mathbb T_0$ (or its closure), and
a scalar function or tensor field $v\in C^{k,\beta}(U)$, define
\begin{align*}
    \|v\|_{C^{k,\beta}_R(U)}
    &:=\sum_{j=0}^k R^j
    \|D_0^jv\|_{L^\infty(U)}+R^{k+\beta}[D_0^kv]_{\beta;U},
\end{align*}
where, for an open subset $U\subset\mathbb T_0$ and a tensor field $T$ on $U$, we set
\begin{align*}
    [T]_{\beta;U}:=\sup_{\substack{p,q\in U\\p\neq q}}\frac{|T(p)-P_{q,p}T(q)|}{d_{\mathbb T_0}(p,q)^\beta}.
\end{align*}
Here, $P_{q,p}$ denotes parallel transport of tensors from $q$ to $p$ along any smooth
path, acting as the identity on scalars.\footnote{Recall that, on a flat torus, parallel transport between two points is independent of the chosen path.}
On $\mathbb T$, we use the same definitions with $d_{\mathbb T}$ and the flat covariant derivative $D$ of $g_{\mathbb T}$ in place of $d_{\mathbb T_0}$ and $D_0$, respectively.
\begin{lemma}[Uniform estimates for $G_a$]
    \label{lem:s4-Green}
    Set $r_*:=1/4$. There exist constants
    $a_*\in(0,1/4]$ and $c,C>0$ with the following
    properties.

    For every $a\in[0,1/4]$, there are a number
    $c_a\in\mathbb R$ and a function
    $E_a\in C^\infty(B_{\mathbb T_0}(0,r_*)
    \smallsetminus\{0\})$ such that
    \begin{align}
        \label{eq:s4-Green-local}
        G_a(q)=2\pi\log r+c_a+E_a(q),
        \qquad
        |c_a|\leq C,
        \qquad
        |c_a-c_0|\leq Ca,
    \end{align}
    for every $q\in B_{\mathbb T_0}(0,r_*)
    \smallsetminus\{0\}$, with
    $r:=d_{\mathbb T_0}(q,0)$.
    For every nonnegative integer $k$ and every
    $\beta\in(0,1)$, there exists $C_{k,\beta}>0$
    such that
    \begin{align}
        \label{eq:s4-Green-remainder}
        \|E_a\|_{C^{k,\beta}_R(\mathcal A_R)}
        \leq C_{k,\beta}R^2(1+a|\log R|)
        \qquad
        \forall\,a\in[0,1/4],\quad 0<R<r_*/2.
    \end{align}

    On the whole punctured torus, there exist
    constants $C_j>0$, $j=1,2,\ldots$, such that
    \begin{align}
        \label{eq:s4-Green-global}
        |G_a(q)-2\pi\log r|
        &\leq C,\nonumber\\
        |D_0^jG_a(q)|
        &\leq C_j(1+r^{-j})
        \qquad\forall\,j=1,2,\ldots,
    \end{align}
    for every $a\in[0,1/4]$ and
    $q\in\mathbb T_0\smallsetminus\{0\}$, with
    $r:=d_{\mathbb T_0}(q,0)$.

    Finally, for every $a\in[0,a_*]$ and
    $q\in\mathbb T_0\smallsetminus\{0\}$,
    with $r:=d_{\mathbb T_0}(q,0)$, one has
    \begin{align}
        \label{eq:s4-Green-Hessian}
        |D_0^2G_a(q)|^2
        &\geq c(1+r^{-4}),\nonumber\\
        2a|D_0G_a(q)|^2-2|D_0^2G_a(q)|^2
        &\leq-c(1+r^{-4}).
    \end{align}
    All constants are independent of $a$
    in the stated ranges.
\end{lemma}
The proof of Lemma \ref{lem:s4-Green} is given in Appendix~\ref{app:s4-green}.
\subsection{The lattice solution and the choice of the waist radii}
We now construct the height function on $\mathbb T$ that will determine the shape of the two sheets away from the necks. We first rescale the Green function $G_{m^{-2}}$ to construct a function $\Phi_m$ on $\mathbb{T}$ whose singularities occur precisely at $L_m$ and which satisfies the equation $L_{\mathbb T}\Phi_m=0$ away from $L_m$. We then compare its expansion near a lattice point with the height of a catenoid. This comparison will determine the waist radius of the catenoidal necks up to a multiplicative constant.

For an integer $m\geq2$, define
\begin{align}
    \label{eq:s4-physical-Green}
    \Phi_m(s,t):=\frac{m^2}{2\pi}+\frac1{2\pi}G_{m^{-2}}\big(\sqrt2m\,s,\sqrt2m\,t\big) \qquad\forall\,(s,t)\in\mathbb T\smallsetminus L_m.
\end{align}
The arguments of $G_{m^{-2}}$ are taken modulo $2\pi$ in each coordinate. Since $\sqrt2m\ell=2\pi m$, the map $(s,t)\mapsto\big(\sqrt2m\,s,\sqrt2m\,t\big)$ is a well-defined covering of $\mathbb T_0$ by $\mathbb T$. The inverse image of the origin by such covering map is exactly $L_m$. The symmetries of $G_{m^{-2}}$ imply that $\Phi_m$ is invariant under the action of $\Gamma_m^0$ on $\mathbb T$.
The normalisation in \eqref{eq:s4-physical-Green} and Lemma \ref{lem:s4-Green} give
\begin{align}
    \label{eq:s4-physical-equation}
    \begin{cases}
        \displaystyle{L_{\mathbb T}\Phi_m=2\pi\sum_{p\in L_m}\delta_p} & \mbox{ distributionally on } \mathbb{T},\\[2mm]
        \displaystyle{\frac1{2\pi^2}\int_{\mathbb T}\Phi_m\,d\vol_{g_{\mathbb{T}}}=\frac{m^2}{2\pi}}
    \end{cases}
\end{align}
where $\delta_p$ is the Dirac measure of mass one at $p$.
For $q\in\mathbb T$, let $\rho(q):=d_{\mathbb T}(q,L_m)$. As every point in $\mathbb{T}$ lies within distance $\ell/(\sqrt2m)=\pi/m$ of a lattice point and
\begin{align*}
    d_{\mathbb T_0}\big((\sqrt2m\,s,\sqrt2m\,t),0\big)=\sqrt2m\rho(q)\qquad\forall\,q=(s,t)\in\mathbb T,
\end{align*}
Lemma~\ref{lem:s4-Green} implies that
\begin{align}
    \Phi_m(q)&=\frac{m^2}{2\pi}
        +\log\big(\sqrt2m\rho(q)\big)+O(1),
        \label{eq:s4-physical-values}\\
    |D^j\Phi_m(q)|&\leq C_j\rho(q)^{-j}
        \qquad\forall\,j\geq1,
        \label{eq:s4-physical-derivatives}\\
    |D^2\Phi_m(q)|^2&\geq c\bigl(m^4+\rho(q)^{-4}\bigr)
        \label{eq:s4-physical-Hessian}
\end{align}
for every $q\in\mathbb T\smallsetminus L_m$ and all
sufficiently large $m$, chosen so that $m^{-2}\leq a_*$. 
\begin{remark}[Local H\"older bounds]
    The above derivative bounds also give local H\"older estimates. More precisely, if $B_q:=B_{\mathbb{T}}(q,\rho(q)/8)$, then
    \begin{align*}
        \|D^j\Phi_m\|_{C^{k,\beta}_{\rho(q)}(B_q)}\leq C_{j,k,\beta}\rho(q)^{-j}
    \end{align*}
    for all integers $j\geq1$, $k\geq0$ and $\beta\in(0,1)$. Indeed, since $\rho$ is bounded below by $7\rho(q)/8$ on $B_q$, we can apply \eqref{eq:s4-physical-derivatives} there, and use one more derivative and the mean-value estimate to bound the H\"older seminorm.
\end{remark}
\begin{remark}[Bounds on the Gaussian curvature term]
    The quantity $|D\Phi_m|^2-\tfrac12|D^2\Phi_m|^2$, which will enter the Gaussian curvature estimates, has a negative upper bound independent of $m$ and $q$. Indeed, for $q=(s,t)\in\mathbb T\smallsetminus L_m$, set $X:=(\sqrt2m\,s,\sqrt2m\,t)\in\mathbb T_0$. By the chain rule and \eqref{eq:s4-Green-Hessian}, we have
    \begin{align}
    \begin{split}
        \label{eq:s4-height-quadratic-margin}
        |D\Phi_m(q)|^2-\tfrac12|D^2\Phi_m(q)|^2&=\frac{m^4}{4\pi^2}\bigl(2m^{-2}|D_0G_{m^{-2}}(X)|^2-2|D_0^2G_{m^{-2}}(X)|^2\bigr)\\
        &\leq-c\bigl(m^4+\rho(q)^{-4}\bigr)
    \end{split}
    \end{align}
    for every $q\in\mathbb T\smallsetminus L_m$ and $m^{-2}\leq a_*$, with $c>0$ independent of $m$ and $q$. Note that this inequality is strictly stronger than the assertion that the Hessian of $\Phi_m$ does not vanish.
\end{remark}
We next describe the expansion of $\Phi_m$ near each
lattice point. Set $\delta:=(100m)^{-1}$. This constant will simply represent a fixed small fraction of the lattice spacing.
Given any $p\in L_m$, define $r_p(q):=d_{\mathbb T}(q,p)$ for every $q\in\mathbb T$. Since $2\delta<\operatorname{inj}(\mathbb T)$, the exponential map $\exp_p^{\mathbb T}$ identifies the disk $\{v\in T_p\mathbb T \mbox{ : }|v|<2\delta\}$ in $T_p
\mathbb{T}$ with $B_{\mathbb T}(p,2\delta)$ in $\mathbb{T}$. For $q$ in this ball, we write $q-p:=(\exp_p^{\mathbb T})^{-1}(q)$; thus $r_p(q)=|q-p|$. Rescaling \eqref{eq:s4-Green-local} and \eqref{eq:s4-Green-remainder} gives
\begin{align}
    \label{eq:s4-physical-expansion}
    \begin{split}
        \Phi_m(q)=\log r_p(q)+A_m+\mathcal E_{m,p}(q), \qquad \|\mathcal E_{m,p}\|_{C^{k,\beta}_R(\mathcal A_R(p))}\leq C_{k,\beta}R^2(m^2+|\log R|),
    \end{split}
\end{align}
with
\begin{align*}
    A_m:=\frac{m^2}{2\pi}+\log(\sqrt2m)+\frac{c_{m^{-2}}}{2\pi},
\end{align*}
for every $q\in B_{\mathbb T}(p,2\delta)\smallsetminus\{p\}$,
every nonnegative integer $k$, every $\beta\in(0,1)$,
and every $R\in(0,\delta)$. Here, the norm of $\mathcal E_{m,p}$ is taken on the annulus $\mathcal A_R(p):=\{x\in\mathbb T:R/2<r_p(x)<2R\}$.

We compare this expansion with the height of a catenoidal
end. For $\tau>0$, the upper half of a Euclidean catenoid
of waist radius $\tau$ is the radial graph
\begin{align*}
    C_\tau(r):=\tau\operatorname{arcosh}(r/\tau)
    \qquad\forall\,r\geq\tau.
\end{align*}
Since $\operatorname{arcosh}x=\log(x+\sqrt{x^2-1})$,
\begin{align*}
    C_\tau(r)
    =\tau\log r+\tau\log(2/\tau)
        +O(\tau^3r^{-2})
    \qquad\text{as }r/\tau\to\infty.
\end{align*}
The logarithmic coefficients of $\tau\Phi_m$ and
$C_\tau$ agree for every $\tau>0$. Their constant terms
agree precisely when
\begin{align*}
    A_m=\log(2/\tau),
    \qquad\text{ or equivalently }\qquad
    \tau=2e^{-A_m}.
\end{align*}
We allow a multiplicative variation of this radius by setting
\begin{align}
    \label{eq:s4-scales}
    \tau_\zeta:=2e^{-A_m+\zeta}
    \qquad\forall\,\zeta\in[-1,1].
\end{align}
The reference waist radius $\tau_0=2e^{-A_m}$ therefore
corresponds to equality of the constant terms.

Define the proposed upper sheet height before interpolation by
\begin{align*}
    \psi_\zeta:=\tau_\zeta\Phi_m
    \quad\mbox{on }\mathbb T\smallsetminus L_m,
    \qquad\forall\,\zeta\in[-1,1].
\end{align*}
By \eqref{eq:s4-physical-values},
\begin{align*}
    \psi_\zeta(q)=\frac{m^2\tau_\zeta}{2\pi}
        +O(\tau_\zeta)
    \qquad\mbox{when }\rho(q)\geq\delta,
\end{align*}
where the remainder term is uniform in $q$, $m$, and $\zeta$.
Thus, $m^2\tau_\zeta/(2\pi)$ is the principal height scale of the sheets outside the radius-$\delta$ disks.
\begin{remark}[Asymptotics of the waist radii as $m\to+\infty$]
    \label{rem:s4-waist-asymptotics}
    Using the definition \eqref{eq:s4-scales} of $\tau_\zeta$ and the expression for $A_m$, we obtain
    \begin{align*}
        \tau_\zeta=\frac{\sqrt2}{m}\exp\left(-\frac{m^2}{2\pi}-\frac{c_{m^{-2}}}{2\pi}+\zeta\right)\qquad\forall\,\zeta\in[-1,1].
    \end{align*}
    Since $c_{m^{-2}}=c_0+O(m^{-2})$ by \eqref{eq:s4-Green-local}, there are constants $c,C>0$ such that
    $cm^2\leq|\log\tau_\zeta|\leq Cm^2$ for all sufficiently
    large $m$ and all $\zeta\in[-1,1]$. More generally,
    \begin{align*}
        \sup_{\zeta\in[-1,1]}m^p\tau_\zeta^q\rightarrow0
        \qquad\text{as }m\to\infty
    \end{align*}
    for every fixed pair of real numbers $p\geq0$ and $q>0$.
    We will use this fact whenever a fixed power of $m$
    multiplies a positive power of $\tau_\zeta$.
\end{remark}
The H\"older estimate in \eqref{eq:s4-physical-expansion}
with $k=1$ shows that the function $\psi_\zeta-\tau_\zeta\log(2r_p/\tau_\zeta)$ extends to a function $\widehat\psi_{\zeta,p}\in C^1(B_{\mathbb T}(p,2\delta))$. We let $a_p:=\widehat\psi_{\zeta,p}(p)$. For $p,p'\in L_m$, the translation $q\mapsto q+p'-p$ preserves $\psi_\zeta$ and satisfies $r_{p'}(q+p'-p)=r_p(q)$. Hence, $a_{p'}=a_p$. By \eqref{eq:s4-physical-expansion}, this common value is given by
\begin{align}
    \label{eq:s4-exact-mismatch}
    a_p
    =\tau_\zeta\bigl(A_m+\log(\tau_\zeta/2)\bigr)
    =\tau_\zeta\zeta
    \qquad\forall\,p\in L_m,\quad\zeta\in[-1,1],
\end{align}
where we used $\log(\tau_\zeta/2)=-A_m+\zeta$. Moreover, as $\widehat\psi_{\zeta,p}$ is invariant under the two coordinate reflections fixing $p$, we have $d\widehat\psi_{\zeta,p}(p)=0$. Consequently,
\begin{align*}
    \psi_\zeta(q)-\tau_\zeta\log(2r_p(q)/\tau_\zeta)=\tau_{\zeta}\zeta+o(\lvert q-p\rvert) \qquad\mbox{ as } q\to p.
\end{align*}
for every $p\in L_m$. The parameter $\zeta$ will be chosen in Section~\ref{sec:s4-coupled}, together with the normal
correction within $\mathcal N$, so that the resulting
surface is minimal in $\mathbb S^4$.
\begin{remark}[Mismatch jets]
    \label{rem:s4-mismatch-jets}
    We now relate the preceding calculation to the matching data used in the linearised doubling construction which we apply in Appendix~\ref{app:s4-scalar-input}.

    Since $L_{\mathbb T}\psi_\zeta=0$ on $\mathbb T\smallsetminus L_m$ and $\psi_\zeta$ has logarithmic coefficient $\tau_\zeta$ at every lattice point, the height function $\psi_\zeta$ is a \textit{linearised doubling solution} in the sense of \cite[Definition~3.6]{KapouleasMcGrathLD}. Its \textit{mismatch jet} at $p$, as defined in \cite[Definition~3.10]{KapouleasMcGrathLD}, is the affine function
    \begin{align*}
        v\mapsto
        \widehat\psi_{\zeta,p}(p)
        +d\widehat\psi_{\zeta,p}(p)[v]
        =\tau_\zeta\zeta
        \qquad\forall\,v\in T_p\mathbb T.
    \end{align*}
    Thus, the mismatch jets are constant and equal at all
    lattice points, and their common value divided by
    $\tau_\zeta$ is exactly $\zeta$.
\end{remark}
\subsection{Joining the two graphs by catenoidal necks}
We now join the two sheets by a catenoidal neck at each
lattice point. We first correct $\psi_\zeta$ to remove
the constant mismatch $\tau_\zeta\zeta$ computed in
\eqref{eq:s4-exact-mismatch}. The correction will satisfy
$L_{\mathbb T}v=0$ near each lattice point, so that its
image under $L_{\mathbb T}$ will be supported away from the
necks. We then interpolate between the corrected sheet heights
and the catenoidal heights, and complete the surface 
by inserting the central portions of the catenoids.

Choose a smooth nondecreasing function
$\vartheta:\mathbb R\to[0,1]$ such that
\begin{align*}
    \begin{cases}
        \vartheta(x)=0 & \forall\,x\le-1,\\
        \vartheta(x)=1 & \forall\,x\ge1,\\
        \vartheta(-x)=1-\vartheta(x) & \forall\,x\in\mathbb R.
    \end{cases}
\end{align*}
For $a_1,a_2\in\mathbb R$ with $a_1\ne a_2$, define
\begin{align}
    \label{eq:s4-cutoff}
    \eta_{a_1,a_2}(r)
    :=\vartheta\left(6\frac{r-a_1}{a_2-a_1}-3\right)
    \qquad\forall\,r\in\mathbb R.
\end{align}
Then $\eta_{a_1,a_2}+\eta_{a_2,a_1}=1$. If $a_1<a_2$,
the function $\eta_{a_1,a_2}$ is zero for
$r\leq(2a_1+a_2)/3$ and one for
$r\geq(a_1+2a_2)/3$. Thus, its variation from zero to one
occurs strictly between $a_1$ and $a_2$. Moreover,
\begin{align}
    \label{eq:der-bounds}
    \left|\frac{d^j}{dr^j}\eta_{a_1,a_2}(r)\right|
    \leq C_j|a_2-a_1|^{-j}
    \qquad\forall\,r\in\mathbb R,\quad j\geq1,
\end{align}
where $j$ is an integer and $C_j$ depends only on $j$
and $\vartheta$. This is the cutoff convention in
\cite[Definition~1.5]{KapouleasMcGrathLD}.

To remove the constant mismatch, we seek a smooth
function whose value is one and whose differential
vanishes at every lattice point. We also require this
function to solve $L_{\mathbb T}v=0$ near those points.
A radial solution with these properties is provided
by the Bessel function of the first kind of order zero,
\begin{align}
    \label{eq:j0-def}
    J_0(x):=\sum_{j=0}^{\infty}
        \frac{(-1)^j}{(j!)^2}\left(\frac{x^2}{4}\right)^j
        \qquad\forall\,x\in\mathbb R.
\end{align}
Recall that $\delta=(100m)^{-1}$, and define
$\widehat v_m,k_m:\mathbb T\to\mathbb R$ by
\begin{align}
    \label{eq:s4-substitute}
    \begin{split}
        \widehat v_m(q)
        &:=\sum_{p\in L_m}
            \eta_{2\delta,\delta}(r_p(q))
            J_0(\sqrt2r_p(q))
            \qquad\forall\,q\in\mathbb T,\\
        k_m&:=L_{\mathbb T}\widehat v_m.
    \end{split}
\end{align}
Each summand in the first equation of \eqref{eq:s4-substitute} is defined on $B_{\mathbb T}(p,2\delta)$ and extended by zero to its complement. These disks are pairwise disjoint. The power series shows that $J_0(\sqrt2r_p)$ is smooth at $p$, and termwise differentiation gives
\begin{align*}
    J_0(\sqrt2r_p)
        &=1-\frac{r_p^2}{2}+O(r_p^4)
            \qquad\text{as }q\to p,\\
    L_{\mathbb T}\bigl(J_0(\sqrt2r_p)\bigr)
        &=0
            \qquad\text{on }B_{\mathbb T}(p,2\delta).
\end{align*}
The cutoff equals one on $B_{\mathbb T}(p,\delta)$
and vanishes near the boundary of
$B_{\mathbb T}(p,2\delta)$. Consequently,
$\widehat v_m$ and $k_m$ are smooth on $\mathbb T$, and
\begin{align}
    \label{eq:s4-substitute-properties}
    \begin{split}
        \widehat v_m(p)&=1,\qquad
        d\widehat v_m(p)=0
            \qquad\forall\,p\in L_m,\\
        \operatorname{spt}k_m
        &\subset\bigcup_{p\in L_m}
            \{q\in\mathbb T:\delta\leq r_p(q)\leq2\delta\}.
    \end{split}
\end{align}
Indeed, applying $L_{\mathbb T}$ to a summand gives
zero wherever its cutoff is constant. Both functions
are $\Gamma_m^0$-invariant, since the symmetries preserve
distances and permute the lattice points.

We now define the corrected height function by
\begin{align}
    \label{eq:s4-corrected-exterior}
    \widetilde\psi_\zeta
        :=\psi_\zeta-\tau_\zeta\zeta\widehat v_m
        \qquad\text{on }\mathbb T\smallsetminus L_m.
\end{align}
For every $p\in L_m$, the function
\begin{align*}
    q\mapsto
    \widehat\psi_{\zeta,p}(q)
        -\tau_\zeta\zeta\widehat v_m(q)
    \qquad\forall\,q\in B_{\mathbb T}(p,2\delta)
\end{align*}
belongs to $C^1(B_{\mathbb T}(p,2\delta))$ and agrees
with $\widetilde\psi_\zeta-\tau_\zeta\log(2r_p/\tau_\zeta)$
on $B_{\mathbb T}(p,2\delta)\smallsetminus\{p\}$.
It therefore provides a $C^1$ extension of the difference $\widetilde\psi_\zeta-\tau_\zeta\log(2r_p/\tau_\zeta)$ across $p$. By \eqref{eq:s4-exact-mismatch} and \eqref{eq:s4-substitute-properties}, its value and differential at $p$ are
\begin{align*}
    \widehat\psi_{\zeta,p}(p)
        -\tau_\zeta\zeta\widehat v_m(p)
        &=\tau_\zeta\zeta-\tau_\zeta\zeta=0,\\
    d\widehat\psi_{\zeta,p}(p)
        -\tau_\zeta\zeta\,d\widehat v_m(p)
        &=0.
\end{align*}
Thus, this correction cancels the constant mismatch
and preserves the vanishing differential. Moreover, since $L_{\mathbb T}\psi_\zeta=0$ away from $L_m$, we have
\begin{align*}
    L_{\mathbb T}\widetilde\psi_\zeta
        =-\tau_\zeta\zeta k_m
        \qquad\text{ on }\,\mathbb T\smallsetminus L_m.
\end{align*}
By \eqref{eq:s4-substitute-properties}, the right-hand
side is supported in
\begin{align*}
    \bigcup_{p\in L_m}
        \{q\in\mathbb T:\delta\leq r_p(q)\leq2\delta\}.
\end{align*}
In particular, the corrected height still satisfies
$L_{\mathbb T}\widetilde\psi_\zeta=0$ on each punctured
disk $B_{\mathbb T}(p,\delta)\smallsetminus\{p\}$.

Fix the H\"older exponent $\beta\in(0,1)$ to be used
in the subsequent estimates, and choose $\alpha>0$
sufficiently small as in
\cite[Convention~2.10]{KapouleasMcGrathLD}, with the
additional restriction
\begin{align}
    \label{eq:s4-alpha}
    \alpha<\min\left\{\frac18,\frac15,\frac1{9(3+\beta)}\right\}.
\end{align}
A further restriction on $\alpha$, depending only on
a fixed cutoff, will be specified when we define the
diffeomorphisms between the initial surfaces. Every
choice of $\alpha$ is made before choosing $m$.
For the rest of this subsection, fix $\zeta\in[-1,1]$. We will interpolate between the corrected height and the catenoidal height
on the annuli $2r_c\leq r_p\leq3r_c$, where
\begin{align*}
    r_c:=\tau_{\zeta}^\alpha.
\end{align*}
By Remark~\ref{rem:s4-waist-asymptotics},
\begin{align*}
    \frac{\tau_{\zeta}}{r_c}=\tau_{\zeta}^{1-\alpha}\rightarrow0,
    \qquad
    \frac{r_c}{\delta}=100m\tau_{\zeta}^\alpha\rightarrow0
    \qquad\text{as }m\to\infty,
\end{align*}
uniformly in $\zeta$. Hence, for all sufficiently
large $m$,
\begin{align*}
    9\tau_{\zeta}<2r_c<3r_c<\delta.
\end{align*}
In particular, the interpolation annuli lie inside
the disks where $\widehat v_m=J_0(\sqrt2r_p)$
and are disjoint from the support of $k_m$.
On
\begin{align*}
    \Omega_\zeta
        :=\mathbb T\smallsetminus
            \bigcup_{p\in L_m}B_{\mathbb T}(p,9\tau_{\zeta}),
\end{align*}
define the interpolated height function
\begin{align}
    \label{eq:s4-glued-height}
    h_\zeta(q):=
    \begin{cases}
        \eta_{3r_c,2r_c}(r_p(q))C_{\tau_{\zeta}}(r_p(q))
        +\eta_{2r_c,3r_c}(r_p(q))\widetilde\psi_\zeta(q),
        &9\tau_{\zeta}\leq r_p(q)\leq3r_c,\,\,p\in L_m,\\
        \widetilde\psi_\zeta(q),
        &\rho(q)\geq3r_c,
    \end{cases}
\end{align}
for every $q\in\Omega_\zeta$. The disks
$B_{\mathbb T}(p,3r_c)$ are pairwise disjoint, so the
choice of $p$ in the first line is unambiguous.
By our cutoff choice, we have
\begin{align*}
    h_\zeta(q)&=C_{\tau_{\zeta}}(r_p(q))
        &&\text{ when }\,9\tau_{\zeta}\leq r_p(q)\leq2r_c,\,\,p\in L_m,\\
    h_\zeta(q)&=\widetilde\psi_\zeta(q)
        &&\text{ when }\,\rho(q)\geq3r_c.
\end{align*}
The two expressions defining $h_\zeta$ agree on
a neighbourhood of each circle $r_p=3r_c$.
Thus, $h_\zeta$ is smooth up to the boundary of
$\Omega_\zeta$.
Define
\begin{align*}
    \mathcal D_\zeta
        :=\{(\xi,\theta)\in\mathbb R\times
            (\mathbb R/2\pi\mathbb Z):\tau_{\zeta}\cosh\xi<2r_c\}.
\end{align*}
For each $p\in L_m$, define the catenoidal parametrisation
\begin{align}
    \label{eq:s4-catenoid-parametrisation}
    \mathcal C_{\zeta,p}(\xi,\theta)
        :=\mathcal F\bigl(
            p+\tau_{\zeta}\cosh\xi(\cos\theta,\sin\theta),\tau_{\zeta}\xi
            \bigr)
        \qquad\forall\,(\xi,\theta)\in\mathcal D_\zeta.
\end{align}
Here $p+v:=\exp_p^{\mathbb T}(v)$ for
$v\in T_p\mathbb T$ with $|v|<2\delta$.
This is the catenoidal model in
\eqref{eq:s4-symmetry-neck-map}, with
$\xi_*=\operatorname{arcosh}(2r_c/\tau_{\zeta})$.
Indeed, $2r_c<\delta<\ell/(2m)$ and
\begin{align*}
    \tau_{\zeta}\operatorname{arcosh}(2r_c/\tau_{\zeta})
        \leq C\tau_{\zeta}|\log\tau_{\zeta}|
        \leq Cm^2\tau_{\zeta}<z_0
\end{align*}
for all sufficiently large $m$. The conditions
\eqref{Equation: conditions symmetries} therefore hold,
and $\mathcal C_{\zeta,p}$ is an embedding.
Its waist corresponds to $\xi=0$.
As shown at the beginning of the proof of Lemma~\ref{lem:s4-initial-geometry} in Appendix~\ref{app:s4-initial-geometry}, the interpolated height $h_{\zeta}$ satisfies
\begin{align*}
    0<h_\zeta(q)
        \leq C\frac{m^2\tau_\zeta}{2\pi}<z_0
        \qquad\forall\,q\in\Omega_\zeta
\end{align*}
for all sufficiently large $m$. Thus, the graphs of $h_\zeta$ and $-h_\zeta$ are disjoint and lie in $\mathcal N$. Define the \textit{initial surface} of parameter $\zeta$ by
\begin{align}
    \label{eq:s4-initial-double}
    \begin{split}
        M_\zeta:=\{\mathcal F(q,h_\zeta(q)):q\in\Omega_\zeta\}
            \cup\{\mathcal F(q,-h_\zeta(q)):q\in\Omega_\zeta\}\cup\bigcup_{p\in L_m}
            \mathcal C_{\zeta,p}(\mathcal D_\zeta).
    \end{split}
\end{align}
On $9\tau_{\zeta}\leq r_p<2r_c$, we have
$h_\zeta=C_{\tau_{\zeta}}(r_p)$. Since
$C_{\tau_{\zeta}}(\tau_{\zeta}\cosh\xi)=\tau_{\zeta}|\xi|$, the upper and
lower graphs agree on this overlap with the
catenoidal portions where $\xi>0$ and $\xi<0$,
respectively.
%
%
\begin{figure}[htbp]
    \centering
    \begingroup
    \definecolor{doubleteal}{RGB}{0,112,116}
    \definecolor{doubleorange}{RGB}{190,102,36}
    \begin{tikzpicture}[
        x=1cm,y=1cm,>=Stealth,
        every node/.style={font=\small},
        sheet edge/.style={black!65,line width=0.8pt},
        neck edge/.style={black!75,line width=0.85pt},
        reference/.style={black!40,dashed,line width=0.55pt}
    ]
        \path[fill=doubleteal!8,even odd rule]
            (0,-1.15)--(7.3,-1.15)--(8.25,-0.40)--(0.95,-0.40)--cycle
            (1.6,-0.78) ellipse[x radius=0.66,y radius=0.19]
            (3.98,-0.78) ellipse[x radius=0.66,y radius=0.19]
            (6.36,-0.78) ellipse[x radius=0.66,y radius=0.19];
        \draw[sheet edge]
            (0,-1.15)--(7.3,-1.15)--(8.25,-0.40)--(0.95,-0.40)--cycle;
        \foreach \xc in {1.6,3.98,6.36} {
            \draw[black!35,dashed] (\xc,-0.78)
                ellipse[x radius=0.66,y radius=0.19];
        }
        \draw[reference]
            (0,0)--(7.3,0)--(8.25,0.75)--(0.95,0.75)--cycle;
        \draw[reference] (0.48,0.37)--(7.78,0.37);
        \path[fill=doubleteal!12,even odd rule]
            (0,1.15)--(7.3,1.15)--(8.25,1.90)--(0.95,1.90)--cycle
            (1.6,1.52) ellipse[x radius=0.66,y radius=0.19]
            (3.98,1.52) ellipse[x radius=0.66,y radius=0.19]
            (6.36,1.52) ellipse[x radius=0.66,y radius=0.19];
        \draw[sheet edge]
            (0,1.15)--(7.3,1.15)--(8.25,1.90)--(0.95,1.90)--cycle;
        \foreach \xc in {1.6,3.98,6.36} {
            \begin{scope}[shift={(\xc,0.37)}]
                \path[fill=doubleteal!7]
                    plot[domain=-2.23:2.23,samples=65]
                        ({0.14*cosh(\x)},{1.15*\x/2.23})
                    -- plot[domain=2.23:-2.23,samples=65]
                        ({-0.14*cosh(\x)},{1.15*\x/2.23}) -- cycle;
                \draw[neck edge]
                    plot[domain=-2.23:2.23,samples=65]
                        ({0.14*cosh(\x)},{1.15*\x/2.23});
                \draw[neck edge]
                    plot[domain=-2.23:2.23,samples=65]
                        ({-0.14*cosh(\x)},{1.15*\x/2.23});
                \draw[black!45] (0,1.15)
                    ellipse[x radius=0.66,y radius=0.19];
                \draw[black!35,dashed] (0,-1.15)
                    ellipse[x radius=0.66,y radius=0.19];
                \draw[doubleteal,line width=1pt] (0,0)
                    ellipse[x radius=0.14,y radius=0.04];
                \fill[doubleorange] (0,0) circle[radius=1.3pt];
            \end{scope}
        }
        \draw[black!50,->] (-0.30,-1.34)--(-0.30,2.13)
            node[above] {$z$};
        \node[anchor=west] at (8.65,1.53) {upper sheet: $z=h_\zeta(q)$};
        \draw[black!45] (8.56,1.53)--(7.91,1.65);
        \node[anchor=west] at (8.65,0.36) {central level: $z=0$};
        \draw[black!45,dashed] (8.56,0.36)--(7.94,0.51);
        \node[anchor=west] at (8.65,-0.81) {lower sheet: $z=-h_\zeta(q)$};
        \draw[black!45] (8.56,-0.81)--(7.91,-0.62);
        \node[anchor=north,align=center] at (1.6,-1.78)
            {waist\\$r_p=\tau_\zeta,\quad z=0$};
        \draw[black!45] (1.6,-1.68)--(1.02,-1.45)--(1.02,0.20)--(1.46,0.35);
        \node[anchor=north,align=center] at (5.65,-1.78)
            {centres $(p,0)$\\with $p\in L_m$};
        \draw[doubleorange,->] (5.50,-1.67)--(5.27,-1.41)--(5.27,0.15)--(6.27,0.35);
        \node[anchor=west,text=black!65] at (0,2.38)
            {A local part of $\mathcal F^{-1}(M_\zeta)$ in the $(s,t,z)$ coordinates};
    \end{tikzpicture}
    \endgroup
    \caption{Two sheets joined by catenoidal necks form $M_\zeta$.
    Three neighbouring necks in one lattice row are shown.}
    \label{fig:s4-initial-double}
\end{figure}
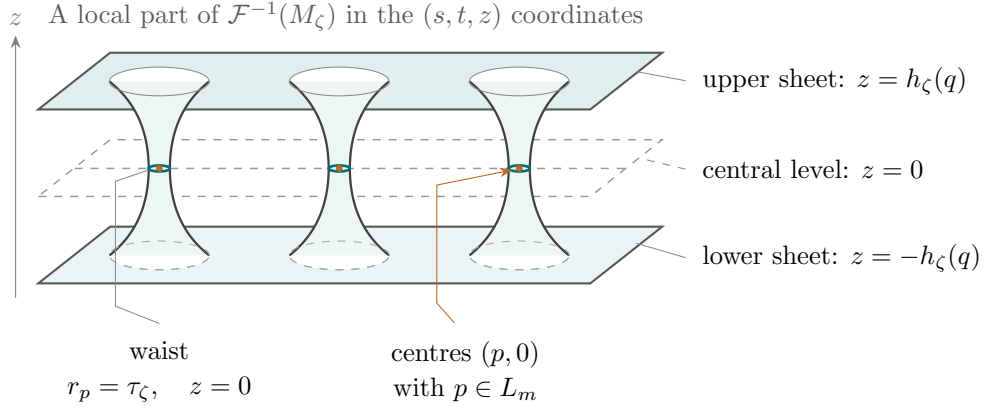

Recall that $z:\mathcal N\to(-z_0,z_0)$ is the signed
height coordinate, characterised by $z(\mathcal F(q,z'))=z'$ for every $(q,z')\in\mathbb T\times(-z_0,z_0)$. Define the positive Lipschitz function $\rho_\zeta:M_\zeta\to(0,\infty)$ by
\begin{align}
    \label{eq:s4-rho-on-double}
    \rho_\zeta(x):=\rho(\pi(x))=d_{\mathbb T}(\pi(x),L_m)\qquad\forall\,x\in M_\zeta.
\end{align}
\begin{lemma}
    \label{lem:s4-initial-geometry}
    For all sufficiently large $m$ and every
    $\zeta\in[-1,1]$, $M_\zeta$ is a smooth closed
    connected embedded $\Gamma_m$-invariant surface
    in $\mathcal N$, of genus $m^2+1$. Moreover, $M_{\zeta}$ is the
    boundary of a $\Gamma_m$-invariant open subset
    of $\mathcal N$ with compact closure.

    Let $\nu_\zeta$ be the outward unit normal of $M_{\zeta}$ within $\mathcal N$, and consider the two inclusions
    $\jmath_\zeta:M_\zeta\hookrightarrow\mathcal N$
    and $\iota_\zeta:M_\zeta\hookrightarrow\mathbb S^4$, which induce the same metric
    \begin{align*}
        g_\zeta:=\jmath_\zeta^*g_{\mathcal N}
            =\iota_\zeta^*g_{\mathbb S^4}
    \end{align*}
    on $M_{\zeta}$. For every $x\in M_\zeta$, we have
    $\tau_\zeta\leq\rho_\zeta(x)\leq\pi/m$ and
    \begin{align}
        \label{eq:s4-initial-geometry-bounds}
        |z(x)|\leq C\frac{m^2\tau_\zeta}{2\pi},
        \qquad
        |(\nu_\zeta)_h(x)|\leq
            C\frac{\tau_\zeta}{\rho_\zeta(x)},
        \qquad
        |S_{\jmath_\zeta}^{\nu_\zeta}(x)|_{g_\zeta}
            \leq C\frac{\tau_\zeta}{\rho_\zeta(x)^2}.
    \end{align}
    The notation for the shape operator and the
    horizontal component of the normal is as in
    Section~\ref{sec:s4-geometry}; the norm of
    $(\nu_\zeta)_h$ is taken with respect to
    $g_{\mathcal N}$.

    For each $x\in M_\zeta$, there is a smooth coordinate
    parametrisation from $B_{\mathbb R^2}(0,1)$,
    sending $0$ to $x$, whose image contains
    $B_{g_\zeta}(x,c\rho_\zeta(x))$ for a fixed $c>0$.
    In these coordinates, the metric
    $\rho_\zeta(x)^{-2}g_\zeta$ has uniformly elliptic
    coefficients with uniform $C^k$ bounds for every
    nonnegative integer $k$. Moreover,
    \begin{align}
        \label{eq:s4-scaled-injectivity}
        \operatorname{inj}_{g_\zeta}(x)
            \geq\rho_\zeta(x)/5
            \qquad\forall\,x\in M_\zeta.
    \end{align}
    All constants are independent of $m$, $\zeta$,
    and $x$.
\end{lemma}

The proof of Lemma \ref{lem:s4-initial-geometry} is given in Appendix~\ref{app:s4-initial-geometry}.
\subsection{Weighted norms and a fixed domain for the corrections}
Lemma~\ref{lem:s4-initial-geometry} gives uniform local geometry after rescaling distances near $x\in M_\zeta$ by $\rho_\zeta(x)^{-1}$. We use this scale to define weighted H\"older norms for scalar functions on $M_\zeta$, which will be used to
measure the scalar coefficients of normal displacements in the directions $\nu_\zeta$ and $n$, with different weights for the two components. We then identify the surfaces $M_\zeta$ with $M_0$, so that the fixed-point argument in Section~\ref{sec:s4-coupled} takes place on a fixed domain.

For a continuous positive function $w:M_\zeta\to(0,\infty)$, a nonnegative integer $k$, and $u\in C^{k,\beta}(M_\zeta)$, define
\begin{align}
    \label{eq:s4-weighted-norm}
    \|u\|_{C^{k,\beta}_\zeta(w)}:=\sup_{x\in M_\zeta}w(x)^{-1}\|u\|_{C^{k,\beta}(B_{g_\zeta}(x,\rho_\zeta(x)/100),\,\rho_\zeta(x)^{-2}g_\zeta)}.
\end{align}
In the definition of the local H\"older norm $\|u\|_{C^{k,\beta}(B_{g_\zeta}(x,\rho_\zeta(x)/100),\,\rho_\zeta(x)^{-2}g_\zeta)}$, covariant derivatives at two different points of $B_{g_\zeta}(x,\rho_\zeta(x)/100)$ are compared by parallel transport along the minimising geodesic joining them, whose existence and uniqueness follow from \eqref{eq:s4-scaled-injectivity}. Indeed, since $|d\pi|\leq\sec z_0$ makes $\rho_\zeta$ $\sec z_0$-Lipschitz, we have
\begin{align*}
    \operatorname{inj}_{g_\zeta}(y)\geq\frac{1-\sec z_0/100}{5}\rho_\zeta(x)>\frac{\rho_\zeta(x)}{50}\qquad\forall\,y\in B_{g_\zeta}(x,\rho_\zeta(x)/100),
\end{align*}
and, for any two points $y,z$ in this ball, the triangle
inequality gives
\begin{align*}
    d_{g_\zeta}(y,z)
        \leq d_{g_\zeta}(y,x)+d_{g_\zeta}(x,z)
        <\frac{\rho_\zeta(x)}{50}
        <\operatorname{inj}_{g_\zeta}(y).
\end{align*}
Thus, there is a unique minimising geodesic joining $y$ to $z$. We also extend the definition of these weighted norms to tensor fields and $\r^m$-valued functions on $M_{\zeta}$ in the obvious way.

To estimate scalar functions $u$ defining normal deformations of $M_{\zeta}$ within $\mathcal{N}$, we introduce the positive weight $f_\zeta:M_\zeta\to(0,\infty)$ given by
\begin{align}
    \label{eq:s4-weight}
    f_\zeta(x):=\max\{\rho_\zeta(x)^{3/2},
        \tau_\zeta^{(1-\alpha)/2}\rho_\zeta(x)^{1/2}\}
        \qquad\forall\,x\in M_\zeta.
\end{align}
%
Write $C^{k,\beta}_{\mathrm{inv}}(M_\zeta)$ for the functions $u\in C^{k,\beta}(M_{\zeta})$ which are invariant under every element of $\Gamma_m$, and define
\begin{align}
    \label{eq:s4-XY}
    \begin{aligned}
        X_\zeta&:=C^{2,\beta}_{\mathrm{inv}}(M_\zeta),
        &\qquad\|u\|_{X_\zeta}
            &:=\|u\|_{C^{2,\beta}_\zeta(f_\zeta)},\\
        Y_\zeta&:=C^{0,\beta}_{\mathrm{inv}}(M_\zeta),
        &\qquad\|E\|_{Y_\zeta}
            &:=\|E\|_{C^{0,\beta}_\zeta(f_\zeta/\rho_\zeta^2)}.
    \end{aligned}
\end{align}
For fixed $m$ and $\zeta$, these norms are equivalent to the usual H\"older norms on the compact surface $M_\zeta$. Moreover, invariance is preserved by uniform convergence. Hence, $X_\zeta$ and $Y_\zeta$ are Banach spaces.
\begin{remark}
    \label{rem:scales}
    Note that the norm $\|\,\cdot\,\|_{C^{k,\beta}_\zeta(w)}$ involves no derivatives of $w$ or $\rho_\zeta$, so the continuity of these two functions suffices for it to be well defined. If $\|u\|_{C^{k,\beta}_\zeta(w)}\leq K$, then
    \begin{align*}
        |(\nabla^{g_{\zeta}})^ju|_{g_\zeta}\leq Kw\rho_\zeta^{-j}\qquad\mbox{ on }\,M_\zeta,\quad\forall\,0\leq j\leq k,
    \end{align*}
    where $\nabla^{g_{\zeta}}$ is the Levi-Civita connection of $g_\zeta$. Since $Y_\zeta$ will be used to estimate terms involving second covariant derivatives of $u\in X_\zeta$, the preceding bound motivates the choice of the weight $f_\zeta/\rho_\zeta^2$ for its norm.
\end{remark}

\begin{remark}[Comparison with the linearised doubling norms]
    The norms in \cite[Definition~4.2]{KapouleasMcGrathLD} contain an additional term measuring the projection of the function being estimated onto the angular modes $\cos\theta$ and $\sin\theta$ on each neck, with the axial coordinate $\xi$ held fixed. This term vanishes for invariant functions. Indeed, the half-turn about $p\in L_m$ preserves $\xi$ and sends $\theta$ to $\theta+\pi$, so
    \begin{align*}
        \int_0^{2\pi}
            u(\mathcal C_{\zeta,p}(\xi,\theta))
            (\cos\theta,\sin\theta)\,d\theta=(0,0).
    \end{align*}
    Consequently, the norms on $X_\zeta$ and $Y_\zeta$ are uniformly equivalent to the norms defined in \cite[Definition~4.2]{KapouleasMcGrathLD} with exponent pairs $(3/2,1/2)$ and $(-1/2,-3/2)$, respectively, restricted to $\Gamma_m$-invariant functions. Changing the fixed radius in the local norms changes only the equivalence constants, by Lemma~\ref{lem:s4-initial-geometry} and local comparability of the weights.
\end{remark}

We now construct diffeomorphisms $F_\zeta:M_0\to M_\zeta$. Near each waist, we rescale the axial coordinate; away from the necks we preserve the base point on $\mathbb T$. Recall that $\tau_\zeta=e^\zeta\tau_0$, and set
\begin{align*}
    r_0:=\tfrac12\tau_0^\alpha,\qquad
    s_\zeta:=\operatorname{arcosh}(\tau_\zeta^{\alpha-1}/2),
    \qquad a_\zeta:=s_\zeta/s_0
    \qquad\forall\,\zeta\in[-1,1].
\end{align*}
For large $m$, we have $s_\zeta>0$ and $2r_0<\delta$.
Define
\begin{align}
    \label{eq:s4-core-identification}
    F_\zeta(\mathcal C_{0,p}(\xi,\theta))
        :=\mathcal C_{\zeta,p}(a_\zeta\xi,\theta)
        \qquad\forall\,p\in L_m,\quad |\xi|\leq s_0.
\end{align}
This sends the circles of radius $r_0$ to those of
radius $\tau_\zeta^\alpha/2$. To extend the map, use the
cutoff from \eqref{eq:s4-cutoff} and define, for $r\geq\tau_0$,
\begin{align}
    \label{eq:s4-radius-identification}
    \begin{split}
        A_\zeta(r)&:=\tau_\zeta
            \cosh\bigl(a_\zeta\operatorname{arcosh}(r/\tau_0)\bigr),\\
        R_\zeta(r)&:=r+
            \eta_{2r_0,r_0}(r)(A_\zeta(r)-r).
    \end{split}
\end{align}
Thus $A_\zeta$ is the radial change induced by
\eqref{eq:s4-core-identification}, while $R_\zeta=A_\zeta$
near $r=r_0$ and $R_\zeta(r)=r$ near and beyond $r=2r_0$.
For a point on either sheet whose base point is
$p+r(\cos\theta,\sin\theta)$ with $r_0\leq r\leq2r_0$,
let $F_\zeta$ be the point on the corresponding sheet
of $M_\zeta$ above $p+R_\zeta(r)(\cos\theta,\sin\theta)$.
For $\rho_0\geq2r_0$, preserve the base point and the
choice of sheet.

\begin{lemma}[Comparison of the surfaces and their function spaces]
    \label{lem:s4-identifications}
    For $\alpha$ sufficiently small in terms of the fixed
    cutoff and $m$ sufficiently large, the maps
    $F_\zeta:M_0\to M_\zeta$ are smooth
    $\Gamma_m$-equivariant diffeomorphisms. They depend
    smoothly on $\zeta$, and $F_0=\operatorname{id}$.
    Moreover, there is $C_F>0$, independent of $m$ and
    $\zeta$, such that
    \begin{align}
        \label{eq:s4-norm-identification}
        \begin{split}
            C_F^{-1}\|u\|_{X_\zeta}
                &\leq\|u\circ F_\zeta\|_{X_0}
                \leq C_F\|u\|_{X_\zeta},\\
            C_F^{-1}\|E\|_{Y_\zeta}
                &\leq\|E\circ F_\zeta\|_{Y_0}
                \leq C_F\|E\|_{Y_\zeta}
        \end{split}
    \end{align}
    for every $u\in X_\zeta$ and $E\in Y_\zeta$.
\end{lemma}

\begin{proof}
    We first prove that the radial interpolation is
    increasing. The expansion of $\operatorname{arcosh}$ gives
    \begin{align*}
        s_0&=(1-\alpha)|\log\tau_0|
            +O(\tau_0^{2(1-\alpha)}),\\
        s_\zeta-s_0&=-(1-\alpha)\zeta
            +O(\tau_0^{2(1-\alpha)}),\qquad
        a_\zeta-1=O(|\log\tau_0|^{-1}).
    \end{align*}
    For $r\in[r_0,2r_0]$, set
    $\xi:=\operatorname{arcosh}(r/\tau_0)$. Since
    $\xi=s_0+O(1)$, we obtain
    \begin{align*}
        \log\frac{A_\zeta(r)}r
            &=\alpha\zeta+O(|\log\tau_0|^{-1}),\\
        A_\zeta'(r)
            &=a_\zeta\frac{A_\zeta(r)}r
                \frac{\tanh(a_\zeta\xi)}{\tanh\xi}.
    \end{align*}
    Thus $A_\zeta/r-1$ and $A_\zeta'-1$ are
    $O(\alpha)+o_m(1)$ on this interval. The cutoff
    derivative is bounded by $C/r_0$, so
    $R_\zeta'=1+O(\alpha)+o_m(1)>1/2$ after choosing
    $\alpha$ small and then $m$ large.

    The transition radii and their images lie in the
    catenoidal portions of both surfaces: indeed,
    \eqref{eq:s4-alpha} gives
    $2r_0=\tau_0^\alpha<2\tau_\zeta^\alpha$, and
    $R_\zeta$ increases from $\tau_\zeta^\alpha/2$
    to $2r_0$. Hence the definitions agree smoothly
    at both ends of the transition. At a waist,
    \eqref{eq:s4-core-identification} is the smooth
    invertible map $(\xi,\theta)\mapsto(a_\zeta\xi,\theta)$;
    elsewhere, strict radial monotonicity gives a smooth
    inverse. This proves that $F_\zeta$ is a diffeomorphism.
    Its formulas also give smooth dependence on $\zeta$,
    equivariance, and $F_0=\operatorname{id}$.

    We next compare the scales. For $|\xi|\leq s_0$,
    \begin{align*}
        \frac{\rho_\zeta(F_\zeta(\mathcal C_{0,p}(\xi,\theta)))}
            {\rho_0(\mathcal C_{0,p}(\xi,\theta))}
            =e^\zeta\frac{\cosh(a_\zeta\xi)}{\cosh\xi}.
    \end{align*}
    This ratio is bounded above and below because
    $|(a_\zeta-1)\xi|\leq C$. The transition estimates
    give the same bounds there, and the radii agree
    outside. Since $e^{-1}\leq\tau_\zeta/\tau_0\leq e$,
    it follows that
    \begin{align*}
        C^{-1}\rho_0\leq\rho_\zeta\circ F_\zeta\leq C\rho_0,
        \qquad
        C^{-1}f_0\leq f_\zeta\circ F_\zeta\leq Cf_0.
    \end{align*}
    On $[r_0,2r_0]$, differentiation also gives
    $r^{j-1}|R_\zeta^{(j)}(r)|\leq C$ for $j=1,2,3$.
    Together with the linear axial map near the waists
    and the unchanged base coordinates outside, these
    bounds give uniform $C^3$ bounds for $F_\zeta$ and
    its inverse in the rescaled charts of
    Lemma~\ref{lem:s4-initial-geometry}.
    The chain rule and the weight comparisons now give
    the upper bounds in \eqref{eq:s4-norm-identification}.
    For the H\"older seminorms, use the local estimates
    for nearby pairs and the supremum bounds for pairs
    separated by a fixed fraction of the local scale.
    Applying the same argument to $F_\zeta^{-1}$ gives
    the lower bounds.
\end{proof}
\subsection{The scalar linear equation and the remaining corrections}
We now study the linearised mean-curvature equation for normal deformations of $M_\zeta$ within $\mathcal N$. We use the resulting solution operator to construct a preliminary correction, and establish uniform estimates for the further deformations considered in Sections~\ref{sec:s4-transverse} and~\ref{sec:s4-coupled}.

Write $H_\zeta:=H_{\jmath_\zeta,\nu_\zeta}$ for the scalar mean curvature of the inclusion of $M_{\zeta}$ into $\mathcal N$, measured with its outward unit normal. Its linearisation under a normal variation within $\mathcal N$ is the operator $\mathcal L_\zeta:C^{2,\beta}(M_\zeta)\to C^{0,\beta}(M_\zeta)$ given by
\begin{align}
    \label{eq:s4-L-zeta}
    \mathcal L_\zeta u:=\Delta_{g_\zeta}u+\bigl(|S_{\jmath_\zeta}^{\nu_\zeta}|_{g_\zeta}^2+\operatorname{Ric}_{g_{\mathcal N}}(\nu_\zeta,\nu_\zeta)\bigr)u \qquad\forall\,u\in C^{2,\beta}(M_\zeta).
\end{align}
No minimality assumption on $M_\zeta$ is needed for this linearisation formula. 

To formulate the linear problem, define the lift of $k_m\in C^{\infty}(\mathbb{T})$ to $M_\zeta$ by
\begin{align}
    \label{eq:s4-J-lift}
    \mathcal J_\zeta:=k_m\circ\pi\big|_{M_\zeta}.
\end{align}
This function is smooth, $\Gamma_m$-invariant, and vanishes
on the necks by \eqref{eq:s4-substitute-properties}. It equals $k_m$ at the projected point on both sheets. This agrees with the transformation rule for scalar mean curvature: the reflection $Z$ exchanges the sheets and carries their outward unit normals to one another. In the notation of \cite[Notation~3.25]{KapouleasMcGrathLD}, we have $\mathcal J_\zeta=J_{M_\zeta}(k_m,k_m)$.

We use the construction of \cite{KapouleasMcGrathLD} to solve $\mathcal L_\zeta u=E$ up to a multiple of $\mathcal J_\zeta$. The extra term allows the linear construction to adjust the common value of its auxiliary solution on the central torus at the lattice points while maintaining the weighted estimates. The assertion below is a right inverse for the augmented map $(u,c)\mapsto\mathcal L_\zeta u-c\mathcal J_\zeta$; it does not assert invertibility of $\mathcal L_\zeta$ alone. The total coefficient of $\mathcal J_\zeta$ in the remaining equation will be set equal to zero by choosing $\zeta$ in Section~\ref{sec:s4-coupled}.

The following proposition provides the linear solution operator and nonlinear estimates needed for the coupled fixed-point construction in Section~\ref{sec:s4-coupled}. It constructs a preliminary scalar correction and controls the mean-curvature error after adding a further correction $v\circ F_\zeta^{-1}$ on $M_\zeta$, where $(v,\zeta)$ belongs to the set
\begin{align}
    \label{eq:s4-scalar-ball}
    \mathcal B_m:=\{(v,\zeta)\in X_0\times[-1,1]\mbox{ : } \|v\|_{X_0}\leq\tau_0^{1+\alpha}\}.
\end{align}
\begin{proposition}[Uniform scalar estimates]
    \label{prop:s4-scalar-input}
    For all sufficiently large $m$, the following hold uniformly
    for $|\zeta|\leq1$.
    \begin{enumerate}
        \item There is a linear map
            $\mathscr{R}_\zeta:Y_\zeta\to X_\zeta\times\mathbb{R}$,
            $\mathscr{R}_\zeta E=(u_E,c_E)$, such that
            \begin{align}
                \label{eq:s4-scalar-inverse}
                \mathcal{L}_\zeta u_E=E+c_E\mathcal{J}_\zeta,
                \qquad
                \|u_E\|_{X_\zeta}+|c_E|\|k_m\|_{C^{0,\beta}(\mathbb{T})}
                \leq C\delta^{-4-2\beta}\|E\|_{Y_\zeta}.
            \end{align}
            In particular,
            \begin{align}
                \label{eq:s4-coefficient-bound}
                |c_E|\leq C\delta^{-2-\beta}\|E\|_{Y_\zeta}.
            \end{align}
            The solution operators can be chosen with the following continuity
            property. Fix $m$ and $0<\beta'<\beta$. Suppose
            $\zeta_j\to\zeta$, $E_j\in Y_{\zeta_j}$ have uniformly
            bounded $Y_{\zeta_j}$ norms, $E\in Y_\zeta$, and
            $E_j\circ F_{\zeta_j}\to E\circ F_\zeta$ in
            $C^{0,\beta'}(M_0)$. If
            $(u_j,c_j)=\mathscr{R}_{\zeta_j}E_j$ and
            $(u,c)=\mathscr{R}_\zeta E$, then
            $u_j\circ F_{\zeta_j}\to u\circ F_\zeta$ in
            $C^{2,\beta'}(M_0)$ and $c_j\to c$.
        \item Define the scalar mean-curvature error $D_\zeta$ and the preliminary
            correction $u_\zeta$ by
            \begin{align}
                \label{eq:s4-initial-correction}
                D_\zeta=H_\zeta+\tau_\zeta\zeta\mathcal{J}_\zeta,
                \qquad
                (u_\zeta,c_{H,\zeta})=-\mathscr{R}_\zeta D_\zeta.
            \end{align}
            Then,
            \begin{align}
                \|D_\zeta\|_{Y_\zeta}&\leq C\tau_\zeta^{1+\alpha/3},
                \label{eq:s4-H-error}\\
                \mathcal{L}_\zeta u_\zeta+H_\zeta
                &=(-\tau_\zeta\zeta+c_{H,\zeta})\mathcal{J}_\zeta.
                \label{eq:s4-linear-cancellation}
            \end{align}
        \item For every $(v,\zeta)\in\mathcal{B}_m$, define
            the total scalar displacement $\phi:M_\zeta\to\mathbb R$ by
            \begin{align}
                \label{eq:s4-full-scalar-correction}
                \phi=u_\zeta+v\circ F_\zeta^{-1}.
            \end{align}
            Then
            \begin{align}
                \label{eq:s4-full-ball-bound}
                \|\phi\|_{X_\zeta}
                +|c_{H,\zeta}|\|k_m\|_{C^{0,\beta}(\mathbb{T})}
                \leq\tau_\zeta^{1+\alpha/4}.
            \end{align}
            The map
            \begin{align*}
                G_{\zeta,\phi}:M_\zeta&\rightarrow\mathcal N,\\
                x&\mapsto\exp_{\jmath_\zeta(x)}^{\mathcal N}
                    (\phi(x)\nu_\zeta(x))
            \end{align*}
            is well defined and is an embedding of class
            $C^{2,\beta}$. The same conclusion holds with
            $\phi$ replaced by $t\phi$ for every $t\in[0,1]$.
            Write
            $M_{\zeta,\phi}:=G_{\zeta,\phi}(M_\zeta)$.
            Choose the unit normal $\nu_{\zeta,\phi}$ along
            $G_{\zeta,\phi}$ by continuity from $\nu_\zeta$
            along the family $G_{\zeta,t\phi}$, $0\leq t\leq1$,
            and set
            $H_{\zeta,\phi}:=H_{G_{\zeta,\phi},\nu_{\zeta,\phi}}$.
            This is a function on the fixed domain $M_\zeta$,
            computed using the induced metric
            $G_{\zeta,\phi}^*g_{\mathcal N}$.
            The nonlinear remainder term
            \begin{align}
                \label{eq:s4-Q-definition}
                Q_\zeta(\phi)=H_{\zeta,\phi}-H_\zeta-\mathcal{L}_\zeta\phi
            \end{align}
            satisfies
            \begin{align}
                \label{eq:s4-scalar-quadratic}
                \|Q_\zeta(\phi)\|_{Y_\zeta}
                \leq C\tau_\zeta^{-\alpha/2}\|\phi\|_{X_\zeta}^2
                \leq C\tau_\zeta^2.
            \end{align}
            The bound \eqref{eq:s4-full-ball-bound} also gives
            \begin{align}
                \label{eq:s4-scalar-two-jet}
                |(\nabla^{g_{\zeta}})^j\phi|_{g_\zeta}\leq C\tau_\zeta^{1+\alpha/4}f_\zeta\rho_\zeta^{-j},
                \qquad\mbox{ on }\,M_\zeta,\quad\forall\,j\in\{0,1,2\},
            \end{align}
            with the corresponding local scaled H\"older estimates.
    \end{enumerate}
\end{proposition}

The proof of Proposition \ref{prop:s4-scalar-input} is given in Appendix~\ref{app:s4-scalar-input}.
\section{Perturbation normal to the suspension band}
\label{sec:s4-transverse}
In this section, we construct a small displacement of
$M_{\zeta,\phi}$ normal to $\mathcal N$ that cancels the
mean-curvature component corresponding to this normal
direction, for all sufficiently large $m$ and every
$(v,\zeta)\in\mathcal B_m$. 
Fix $(v,\zeta)\in\mathcal B_m$, with $\mathcal B_m$ as in \eqref{eq:s4-scalar-ball}, and set
\begin{align}
    \label{eq:s4-transverse-abbreviations}
    \phi:=u_\zeta+v\circ F_\zeta^{-1}.
\end{align}
Let
\begin{align}
    \label{eq:s4-scalar-graph-map}
    G:=G_{\zeta,\phi}:M_\zeta\to\mathcal N,
    \qquad g_\phi:=G^*g_{\mathcal N},
\end{align}
where $G_{\zeta,\phi}$ is the $C^{2,\beta}$ embedding
given by Proposition~\ref{prop:s4-scalar-input}-(3). Write
$\iota:=\iota_{\mathcal N}\circ G:M_\zeta\to\mathbb S^4$,
where $\iota_{\mathcal N}$ is the inclusion of $\mathcal{N}$ in $\mathbb{S}^4$, and define
\begin{align*}
    \nu:=\nu_{\zeta,\phi},\qquad n(x):=N\bigl(\pi(G(x))\bigr)\qquad\forall\,x\in M_\zeta.
\end{align*}
Then, $(\nu,n)$ is an orthonormal frame of the normal
bundle of $\iota$, and Lemma~\ref{lem:s4-normal-identities}
gives
\begin{align*}
    H_{\iota,\nu}=H_{\zeta,\phi},\qquad
    \mathbf H_\iota=H_{\zeta,\phi}\nu+H_{\iota,n}n.
\end{align*}
We denote by $h_\iota^\nu$ and $h_\iota^n$ the scalar second fundamental forms of $\iota$ in the directions $\nu$ and $n$, respectively, and by $\omega$ the normal-connection one-form defined in \eqref{eq:s4-omega-definition} for the frame $(\nu,n)$. The field $\nu$ is $C^{1,\beta}$, whereas $n$ is $C^{2,\beta}$ because it is the restriction of a fixed smooth field on $\mathcal N$. This regularity suffices for the arguments below.

Unless otherwise stated, constants are independent of $m$ and $(v,\zeta)\in\mathcal B_m$, and may depend on the fixed band, cutoffs, $\alpha$ and $\beta$. The notation $o_m(1)$ denotes convergence to zero as $m\to\infty$, uniformly in $(v,\zeta)\in\mathcal B_m$.
\subsection{The normal deformation and its mean-curvature components}
\begin{lemma}
    \label{lem:s4-full-family-geometry}
    Let $q$ be the function defined in \eqref{eq:s4-scalar-Jacobi}
    for the embedding $\iota$, and write
    \begin{align*}
        \nu_h(x):=\sum_{a=1}^2
        \langle\nu(x),E_a(G(x))\rangle E_a(G(x))
        \qquad\forall\,x\in M_\zeta.
    \end{align*}
    For all sufficiently large $m$, one has
    $C^{-1}g_\zeta\leq g_\phi\leq Cg_\zeta$ and
    \begin{align}
        \label{eq:s4-full-family-geometry}
        \begin{aligned}
            \|z\circ G\|_{C^{0,\beta}_\zeta(1)}
                &\leq C\frac{m^2\tau_\zeta}{2\pi},
            &\qquad\|\nu_h\|_{C^{0,\beta}_\zeta(\rho_\zeta^{-1})}
                &\leq C\tau_\zeta,\\
            \|S_\iota^\nu\|_{C^{0,\beta}_\zeta(\rho_\zeta^{-2})}
                &\leq C\tau_\zeta,
            &\qquad\|h_\iota^n\|_{C^{0,\beta}_\zeta(\rho_\zeta^2)}
                &\leq C,\\
            \|H_{\iota,n}\|_{C^{0,\beta}_\zeta(\rho_\zeta^{-2})}
                &\leq C\tau_\zeta^2,
            &\qquad\|\omega\|_{C^{0,\beta}_\zeta(1)}
                &\leq C\tau_\zeta,\\
            \|\nabla^{g_\phi}\omega\|_{C^{0,\beta}_\zeta(1)}
                &\leq C\tau_\zeta,
            &\qquad\|q\|_{C^{0,\beta}_\zeta(1)}
                &\leq C.
        \end{aligned}
    \end{align}
    Moreover,
    \begin{align}
        \label{eq:s4-deformed-scale}
        d_{\mathbb T}\bigl(\pi(G(x)),L_m\bigr)
        =\bigl(1+o_m(1)\bigr)\rho_\zeta(x)
        \qquad\forall\,x\in M_\zeta.
    \end{align}
\end{lemma}
The proof of Lemma~\ref{lem:s4-full-family-geometry} is given in Appendix~\ref{app:s4-normal}. 

Observe that the displacement in the direction $n$ must be $\chi$-equivariant for the deformed surface to remain $\Gamma_m$-equivariant. For $b\in C^{2,\beta}_\chi(M_\zeta)$ (see \eqref{eq:s4-symmetry-spaces}), let
\begin{align}
    \label{eq:s4-transverse-graph}
    \begin{split}
        X_b(x)&:=\widehat X_{\zeta,\phi,b}(x)
            =\exp_{\iota(x)}^{\mathbb S^4}(b(x)n(x))=\cos(b(x))\iota(x)+\sin(b(x))n(x),\\[2mm]
        V_b(x)&:=-\sin(b(x))\iota(x)+\cos(b(x))n(x)
            \qquad\forall\,x\in M_\zeta.
    \end{split}
\end{align}
Note that $X_b$ and $V_b$ inherit their regularity from $b$, and hence they are of class $C^{2,\beta}$ on $M_{\zeta}$. In local coordinates on $M_\zeta$, write
$b_i:=\partial_i b$ and $\omega_i:=\omega(\partial_i)$
for $i\in\{1,2\}$. As the Weingarten equation gives
$\partial_i n=-d\iota(S_\iota^n(\partial_i))-\omega_i\nu$, setting $\mathsf P_b:=\cos b\,\operatorname{Id}_{TM_\zeta}
-\sin b\,S_\iota^n$ we obtain
\begin{align}
    \label{eq:s4-graph-tangent}
    \partial_iX_b
        =d\iota\bigl(\mathsf P_b(\partial_i)\bigr)
            -\sin b\,\omega_i\nu+b_iV_b.
\end{align}
Since the three summands on the right-hand side of \eqref{eq:s4-graph-tangent} are mutually orthogonal, the induced metric $g_b:=X_b^*g_{\mathbb S^4}$ satisfies
\begin{align}
    \label{eq:s4-graph-metric-exact}
    (g_b)_{ij}=g_\phi\bigl(\mathsf P_b(\partial_i),\mathsf P_b(\partial_j)\bigr)+\sin^2b\,\omega_i\omega_j+b_ib_j.
\end{align}
Fix a sufficiently small constant $c_0>0$ and suppose
$\|b\|_{C^{2,\beta}_\zeta(1)}\leq c_0\tau_\zeta$. Since
$\rho_\zeta\geq\tau_\zeta$, Lemma~\ref{lem:s4-full-family-geometry}
gives $|db|_{g_\phi}\leq Cc_0$ and
$|\mathsf P_b-\operatorname{Id}|_{g_\phi}\leq Cc_0\tau_\zeta$.
Thus, $X_b$ is a $C^{2,\beta}$ immersion and
$\frac14g_\phi\leq g_b\leq Cg_\phi$.

\smallskip
By \eqref{eq:s4-graph-tangent}, the orthogonal projections of $V_b$ and $\nu$ onto the normal bundle of the immersion $X_b$ are
\begin{align}
    \label{eq:s4-testing-projections}
    \begin{split}
        U_b&:=V_b-\sum_{i,j=1}^2
            (g_b)^{ij}b_i\partial_jX_b,\\
        W_b&:=\nu+\sin b\sum_{i,j=1}^2
            (g_b)^{ij}\omega_i\partial_jX_b,
    \end{split}
\end{align}
where $(g_b)^{ij}$ are the entries of the inverse
of the metric matrix $((g_b)_{ij})$. Note that
\begin{align}
    \label{eq:s4-testing-lengths}
    \begin{split}
        |U_b|^2&=1-|db|_{g_b}^2,\\
        |W_b|^2&=1-\sin^2b\,|\omega|_{g_b}^2,\\
        \langle U_b,W_b\rangle
            &=\sin b\,\langle db,\omega\rangle_{g_b}.
    \end{split}
\end{align}
Consequently,
$|U_b|^2|W_b|^2-\langle U_b,W_b\rangle^2
\geq1-Cc_0^2\geq1/2$, after possibly decreasing $c_0$. Therefore, for sufficiently small $c_0$ the unit normal fields
\begin{align}
    \label{eq:s4-testing-normals}
    n_b:=\frac{U_b}{|U_b|},\qquad
    \nu_b:=\frac{W_b}{|W_b|}
\end{align}
span the normal plane of $X_b$ at every point of $M_{\zeta}$. They need not be orthogonal, and for $b=0$ they equal $n$ and $\nu$.

To estimate the displacement $b$ and the two mean-curvature components of $X_b$ in $\mathbb{S}^4$, consider the following function spaces:
\begin{align}
    \label{eq:s4-transverse-spaces}
    \begin{aligned}
        \mathcal X_\zeta&:=C^{2,\beta}_\chi(M_\zeta),
        &\qquad\|b\|_{\mathcal X_\zeta}
            &:=\|b\|_{C^{2,\beta}_\zeta(1)},\\
        \mathcal Y_\zeta&:=C^{0,\beta}_\chi(M_\zeta),
        &\qquad\|E\|_{\mathcal Y_\zeta}
            &:=\|E\|_{C^{0,\beta}_\zeta(\rho_\zeta^{-2})},\\
        \mathcal Z_\zeta&:=C_{\mathrm{inv}}^{0,\beta}(M_\zeta),
        &\qquad\|E\|_{\mathcal Z_\zeta}
            &:=\|E\|_{C^{0,\beta}_\zeta(\rho_\zeta^{-2})}.
    \end{aligned}
\end{align}
Define the closed ball
\begin{align*}
    \mathcal U_{\zeta,c_0}:=\{b\in\mathcal X_\zeta\mbox{ : }\|b\|_{\mathcal X_\zeta}\leq c_0\tau_\zeta\}\subset\mathcal{X}_{\zeta}.
\end{align*}
Let $\mathscr P_{\zeta,\phi},\mathscr T_{\zeta,\phi}:\mathcal U_{\zeta,c_0}\to  C^{0,\beta}(M_\zeta)$ be given by
\begin{align}
    \label{eq:s4-component-maps}
    \mathscr P_{\zeta,\phi}(b)
        :=\langle\mathbf H_{X_b},n_b\rangle,
    \qquad
    \mathscr T_{\zeta,\phi}(b)
        :=\langle\mathbf H_{X_b},\nu_b\rangle \qquad\forall\,b\in\mathcal U_{\zeta,c_0},
\end{align}
where $\mathbf H_{X_b}$ is the mean-curvature vector of $X_b$ in $\mathbb S^4$. Note that
\begin{align*}
    \mathscr P_{\zeta,\phi}(0)=H_{\iota,n},
    \qquad\mathscr T_{\zeta,\phi}(0)=H_{\zeta,\phi},
\end{align*}
and $X_b$ is a minimal immersion if and only if
$\mathscr P_{\zeta,\phi}(b)=\mathscr T_{\zeta,\phi}(b)=0$.
The transformation rules for $\iota,n,\nu$ give
$X_b\circ a=a\circ X_b$ for every $a\in\Gamma_m$.
Since isometries commute with normal projection and
normalisation, $n_b$ transforms with character $\chi$
and $\nu_b$ is equivariant. Therefore
\begin{align}
    \label{eq:s4-component-parities}
    \begin{split}
        \mathscr P_{\zeta,\phi}(b)\circ a
            &=\chi(a)\mathscr P_{\zeta,\phi}(b),\\
        \mathscr T_{\zeta,\phi}(b)\circ a
            &=\mathscr T_{\zeta,\phi}(b)
            \qquad\forall\,a\in\Gamma_m.
    \end{split}
\end{align}
In particular, $\mathscr P_{\zeta,\phi}$ and $\mathscr T_{\zeta,\phi}$ map into $\mathcal{Y}_{\zeta}$ and $\mathcal{Z}_{\zeta}$ respectively.
\begin{figure}[htbp]
    \centering
    \begingroup
    \definecolor{normalteal}{RGB}{0,112,116}
    \definecolor{normalorange}{RGB}{190,102,36}
    \begin{tikzpicture}[
        x=1cm,y=1cm,>=Stealth,
        every node/.style={font=\small},
        reference/.style={black!40,line width=0.6pt},
        base edge/.style={normalteal,line width=0.9pt},
        new edge/.style={normalorange,line width=0.9pt}
    ]
        \path[fill=black!3]
            (0,0)--(7.0,0)--(8.25,1.55)--(1.25,1.55)--cycle;
        \draw[reference]
            (0,0)--(7.0,0)--(8.25,1.55)--(1.25,1.55)--cycle;
        \node[anchor=north west,text=black!65] at (0.12,-0.06)
            {fixed hypersurface $\mathcal N$};
        \path[fill=normalteal!12]
            (0.85,0.30)--(6.50,0.30)--(7.10,1.02)--(1.45,1.02)--cycle;
        \draw[base edge]
            (0.85,0.30)--(6.50,0.30)--(7.10,1.02)--(1.45,1.02)--cycle;
        \path[fill=normalorange!12]
            (1.03,2.05)
            .. controls (2.70,2.48) and (4.20,2.04) .. (6.75,2.35)
            --(7.34,3.07)
            .. controls (4.82,2.76) and (3.28,3.20) .. (1.62,2.77)
            --cycle;
        \draw[new edge]
            (1.03,2.05)
            .. controls (2.70,2.48) and (4.20,2.04) .. (6.75,2.35)
            --(7.34,3.07)
            .. controls (4.82,2.76) and (3.28,3.20) .. (1.62,2.77)
            --cycle;
        \draw[black!35,dashed]
            (1.45,1.02) .. controls (1.45,1.65) and (1.48,2.40) .. (1.62,2.77);
        \draw[black!35,dashed]
            (6.50,0.30) .. controls (6.50,1.04) and (6.55,1.98) .. (6.75,2.35);
        \coordinate (basepoint) at (3.65,0.63);
        \coordinate (endpoint) at (3.93,2.61);
        \draw[black!55,->,line width=0.65pt]
            (basepoint)--(3.65,1.29);
        \node[anchor=east,text=black!70] at (3.52,1.13) {$n(x)$};
        \draw[normalorange,->,line width=1.05pt]
            (basepoint) .. controls (3.65,1.35) and (3.69,2.09) .. (endpoint);
        \fill[normalteal] (basepoint) circle[radius=1.6pt];
        \fill[normalorange] (endpoint) circle[radius=1.6pt];
        \node[anchor=north] at (3.65,-0.38) {$G(x)$};
        \draw[black!40] (3.65,-0.28)--(3.65,0.51);
        \node[anchor=south] at (3.93,3.27) {$X_b(x)$};
        \draw[black!40] (3.93,3.17)--(3.93,2.72);
        \node[anchor=west] at (4.11,1.86)
            {geodesic of length $|b(x)|$};
        \node[anchor=west,align=left,text=normalteal] at (8.57,0.62)
            {$M_{\zeta,\phi}=G(M_\zeta)$\\surface in $\mathcal N$};
        \draw[normalteal!70] (8.43,0.66)--(7.02,0.89);
        \node[anchor=west,align=left,text=normalorange] at (8.57,2.64)
            {$X_b(M_\zeta)$\\displaced surface};
        \draw[normalorange!70] (8.43,2.68)--(7.31,2.88);
    \end{tikzpicture}
    \endgroup
    \caption{The second deformation moves $G(x)$ along the spherical
    geodesic with initial velocity $n(x)$, through signed
    distance $b(x)$.}
    \label{fig:s4-normal-displacement}
\end{figure}
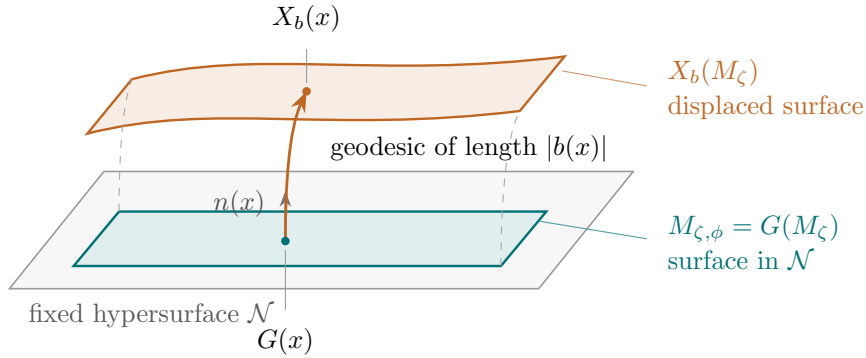

To solve the nonlinear equation $\mathscr P_{\zeta,\phi}(b)=0$ locally around $0$, we first identify its linearisation.
\begin{lemma}
    \label{lem:s4-component-linearisation}
    For every $h\in\mathcal X_\zeta$,
    \begin{align}
        \label{eq:s4-component-linearisation}
        D\mathscr P_{\zeta,\phi}(0)h=P_{\zeta,\phi}h:=(\Delta_{g_\phi}+q)h,
        \qquad D\mathscr T_{\zeta,\phi}(0)h=Bh,
    \end{align}
    where $B$ is the operator in \eqref{eq:s4-scalar-Jacobi} for $\iota$.
\end{lemma}
\begin{proof}
Along the variation $t\mapsto th$ of $0\in\mathcal X_{\zeta}$ with initial velocity $h$, by definition of $\mathscr P_{\zeta,\phi}$ we have
\begin{align*}
    D\mathscr P_{\zeta,\phi}(0)h&=\frac{d}{dt}\bigg\vert_{t=0}\mathscr P_{\zeta,\phi}(th)=\frac{d}{dt}\bigg\vert_{t=0}\langle\mathbf H_{X_{th}},n_{th}\rangle\\
    &=\bigg\langle\frac{D}{dt}\bigg\vert_{t=0}\mathbf  H_{X_{th}},n_{0}\bigg\rangle+\bigg\langle\mathbf H_{X_{0}},\frac{D}{dt}\bigg\vert_{t=0}n_{th}\bigg\rangle\\
    &=\bigg\langle\frac{D}{dt}\bigg\vert_{t=0}\mathbf  H_{X_{th}},n\bigg\rangle+\bigg\langle\mathbf H_{\iota},\frac{D}{dt}\bigg\vert_{t=0}n_{th}\bigg\rangle,
\end{align*}
and, analogously, 
\begin{align*}
    D\mathscr T_{\zeta,\phi}(0)h=\bigg\langle\frac{D}{dt}\bigg\vert_{t=0}\mathbf  H_{X_{th}},\nu\bigg\rangle+\bigg\langle\mathbf H_{\iota},\frac{D}{dt}\bigg\vert_{t=0}\nu_{th}\bigg\rangle.
\end{align*}
Differentiating the identities \eqref{eq:s4-testing-projections} at $t=0$ gives
\begin{align*}
    \frac{D}{dt}\bigg\vert_{t=0}U_{th}=-d\iota(\nabla^{g_\phi}h), \qquad \frac{D}{dt}\bigg\vert_{t=0}W_{th}=h\,d\iota(\omega^{\sharp_{g_\phi}}).
\end{align*}
Moreover, differentiating \eqref{eq:s4-testing-lengths} at $t=0$, we get
\begin{align*}
    \frac{d}{dt}\bigg\vert_{t=0}\lvert U_{th}\rvert=\frac{d}{dt}\bigg\vert_{t=0}\lvert W_{th}\rvert=0.
\end{align*}
Therefore, we obtain
\begin{align*}
    \frac{D}{dt}\bigg\vert_{t=0}n_{th}=-d\iota(\nabla^{g_\phi}h)\in d\iota(TM_{\zeta}), \qquad \frac{D}{dt}\bigg\vert_{t=0}\nu_{th}=h\,d\iota(\omega^{\sharp_{g_\phi}})\in d\iota(TM_{\zeta}).
\end{align*}
Since $\mathbf H_{\iota}$ is normal to the immersion $\iota$, applying Lemma~\ref{lem:s4-normal-linearisation} to the normal variation with initial velocity $hn$ therefore gives
\begin{align*}
    D\mathscr P_{\zeta,\phi}(0)h&=\bigg\langle\frac{D}{dt}\bigg\vert_{t=0}\mathbf  H_{X_{th}},n\bigg\rangle=(\Delta_{g_\phi}+q)h,\\
    D\mathscr T_{\zeta,\phi}(0)h&=\bigg\langle\frac{D}{dt}\bigg\vert_{t=0}\mathbf  H_{X_{th}},\nu\bigg\rangle=Bh.
\end{align*}
The statement follows.
\end{proof}
\subsection{An auxiliary metric and a uniform inverse for its Laplacian}
To invert $P_{\zeta,\phi}$ with a bound independent of $m$ and $(v,\zeta)\in\mathcal B_m$, we first construct an auxiliary metric $g_*$ on $M_\zeta$ and prove that $\Delta_{g_*}$ is invertible on $\chi$-equivariant functions, with a uniform bound in the weighted norms. Then, we obtain the inverse of $P_{\zeta,\phi}$ by showing that $P_{\zeta,\phi}$ is a small perturbation of $\Delta_{g_{\ast}}$ in the appropriate norms.

Recall that $r_c=\tau_\zeta^\alpha$, and define
\begin{align}
    \label{eq:s4-model-wedge}
    \mathcal W_{m,\zeta}
        :=\left\{(\xi,\theta)\mbox{ : }
            0<\theta<\pi/4,\quad
            0<\xi<\log\frac{\ell}{m\tau_\zeta\cos\theta}
            \right\}.
\end{align}
Let $h_\zeta$ be the height function
given in \eqref{eq:s4-glued-height}. As in
Appendix~\ref{app:s4-initial-geometry}, we extend $h_\zeta$
to the region $\{q\in\mathbb T\mbox{ : }
d_{\mathbb T}(q,L_m)>\tau_\zeta\}\subset\mathbb T$
by setting $h_\zeta(q):=C_{\tau_\zeta}(d_{\mathbb T}(q,p))$
whenever $\tau_\zeta<d_{\mathbb T}(q,p)\leq9\tau_\zeta$,
for every $p\in L_m$.
Choose a smooth cutoff function
$\eta_*:[0,\infty)\to[0,1]$ equal to one on $[0,1/4]$
and zero on $[1/2,\infty)$, and define
\begin{align}
    \label{eq:s4-model-radius}
    r_{m,\zeta}(\xi)
        :=\frac{\tau_\zeta}{2}\left(
            e^\xi+
            \eta_*\!\left(\frac{\tau_\zeta e^\xi}{2r_c}\right)
            e^{-\xi}\right)
        \qquad\forall\,\xi\geq0.
\end{align}
Fix any $p\in L_m$, and define
$\Theta_{m,\zeta,p}:\mathcal W_{m,\zeta}\to M_\zeta$
by
\begin{align}
    \label{eq:s4-model-identification}
    \Theta_{m,\zeta,p}(\xi,\theta)
        :=\mathcal F\bigl(
            p+r_{m,\zeta}(\xi)(\cos\theta,\sin\theta),
            h_\zeta(p+r_{m,\zeta}(\xi)(\cos\theta,\sin\theta))
            \bigr).
\end{align}
For all sufficiently large $m$, $\Theta_{m,\zeta,p}$ is
a diffeomorphism from $\mathcal W_{m,\zeta}$ onto its image $D_{m,\zeta,p}\subset M_{\zeta}$, which is the upper part of $M_\zeta$ above the region
\begin{align*}
    \left\{p+(s,t)\mbox{ : }
        0<t<s<\frac{\ell}{2m},\quad
        s^2+t^2>\tau_\zeta^2\right\}.
\end{align*}
Note that the parametrisation $\Theta_{m,\zeta,p}$ extends smoothly to the boundary of $\mathcal W_{m,\zeta}$. Moreover, since $p$ is a nearest lattice point to the projection of every point in $D_{m,\zeta,p}$, the definition of $\rho_\zeta$ gives
\begin{align}
    \label{eq:s4-model-scale}
    \rho_\zeta\circ\Theta_{m,\zeta,p}=r_{m,\zeta}.
\end{align}
%
%
\begin{figure}[htbp]
    \centering
    \begingroup
    \definecolor{modelteal}{RGB}{0,112,116}
    \definecolor{modelorange}{RGB}{190,102,36}
    \begin{tikzpicture}[
        x=5.1cm,y=1.22cm,>=Stealth,
        every node/.style={font=\small}
    ]
        \fill[modelteal!8]
            (0,0)--({pi/4},0)--({pi/4},{2.6-ln(cos(45))})
            --plot[domain=pi/4:0,samples=50]
                (\x,{2.6-ln(cos(\x*180/pi))})--cycle;
        \draw[modelteal,line width=1pt] (0,0)--({pi/4},0);
        \draw[modelteal,line width=1pt] (0,0)--(0,2.6);
        \draw[modelorange,line width=1.2pt]
            ({pi/4},0)--({pi/4},{2.6-ln(cos(45))});
        \draw[modelteal,line width=1pt]
            plot[domain=0:pi/4,samples=50]
                (\x,{2.6-ln(cos(\x*180/pi))});
        \draw[black!40,dashed] (0,2.6)--({pi/4},2.6);
        \node[left=4pt] at (0,2.6) {$\log\frac{\ell}{m\tau_\zeta}$};
        \draw[black!65,->] (0,-0.10)--(0,3.38)
            node[above] {$\xi$};
        \draw[black!65,->] (-0.025,0)--(0.94,0)
            node[right] {$\theta$};
        \node[below left=3pt] at (0,0) {$0$};
        \node[below=4pt] at ({pi/4},0) {$\pi/4$};
        \node[left=6pt,text=modelteal] at (0,1.35)
            {$\partial_\theta\tilde b=0$};
        \node[right=7pt,text=modelorange] at ({pi/4},1.35)
            {$\tilde b=0$};
        \node[below=15pt,text=modelteal] at ({pi/8},0)
            {$\partial_\xi\tilde b=0$};
        \node[align=left,anchor=west] at (1.02,2.98)
            {Cell wall: $\xi=\log\left(\frac{\ell}{m\tau_\zeta\cos\theta}\right)$\\[3pt]
             $\partial_{\eta_{\mathrm{out}}}\tilde b=0$};
        \draw[black!45] (0.61,2.83)--(0.96,2.83);
        \draw[black!55,dashed]
            (0,1.91)--({pi/4},1.91);
        \node[above=2pt,fill=modelteal!8,inner sep=1pt]
            at ({pi/8},1.91) {$0<\theta<\pi/4$};
        \draw[modelorange,line width=1.25pt]
            ({acos(exp(-0.16))*pi/180},2.76)--({pi/4},2.76);
        \fill[modelorange]
            ({acos(exp(-0.16))*pi/180},2.76)
            circle[radius=1.4pt];
        \node[anchor=west,align=left]
            at (1.02,1.94)
            {$\theta=\arccos(\frac{\ell e^{-\xi}}{m\tau_\zeta})$\\for $\xi>\log\frac{\ell}{m\tau_\zeta}$};
        \draw[black!45]
            ({acos(exp(-0.16))*pi/180},2.76)
            --(0.91,2.14)--(0.97,2.14);
    \end{tikzpicture}
    \endgroup
    \caption{The coordinate domain $\mathcal W_{m,\zeta}$. The bottom, left, right and curved sides correspond to the waist, a cell axis, a cell diagonal and a cell wall, respectively. The displayed boundary conditions follow from $\chi$-equivariance; $\eta_{\mathrm{out}}$ denotes the outward unit normal with respect to $g_*$. In the figure, $b\in\mathcal{X}_{\zeta}$ and $\tilde b:=b\circ\Theta_{m,\zeta,p}$.}
    \label{fig:s4-reflected-domain}
\end{figure}
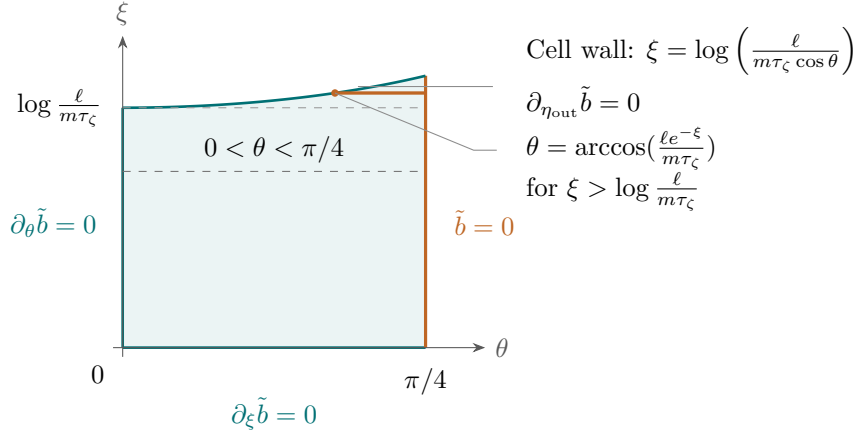

Define the metric $g_*$ on $D_{m,\zeta,p}\subset M_{\zeta}$ by
\begin{align}
    \label{eq:s4-model-metric}
        g_*:=\big(\Theta_{m,\zeta,p}^{-1}\big)^*\big(r_{m,\zeta}^2(d\xi^2+d\theta^2)\big).
\end{align}
Note that the $\Gamma_m$-translates of $\overline{D_{m,\zeta,p}}$ cover $M_\zeta$. We then extend $g_{\ast}$ to $M_{\zeta}$ by requiring every transformation $a\in\Gamma_m$ to be an isometry, namely
\begin{align*}
    (a\circ\Theta_{m,\zeta,p})^*g_*
        =\Theta_{m,\zeta,p}^*g_*
        \qquad\forall\,a\in\Gamma_m.
\end{align*}
Near the waist, $r_{m,\zeta}(\xi)=\tau_\zeta\cosh\xi$, thus
\begin{align*}
    \Theta_{m,\zeta,p}^*g_*
        =\tau_\zeta^2\cosh^2\xi(d\xi^2+d\theta^2).
\end{align*}
Since this expression is even in $\xi$, reflection
across the waist gives a smooth extension. Likewise, the coefficient $r_{m,\zeta}^2$ is independent of $\theta$, so reflections across the axes and diagonals give smoothness of our extension on each individual cell.
To check smoothness across neighbouring cells, observe that for $\tau_\zeta e^\xi\geq r_c$ we have
$r_{m,\zeta}(\xi)=\tau_\zeta e^\xi/2$, which implies $dr_{m,\zeta}=r_{m,\zeta}\,d\xi$. Hence,
\begin{align*}
    \Theta_{m,\zeta,p}^*g_*=(dr_{m,\zeta})^2+r_{m,\zeta}^2\,d\theta^2=(\pi\circ\Theta_{m,\zeta,p})^*g_{\mathbb T}.
\end{align*}
On the cell wall, $\xi=\log(\ell/(m\tau_\zeta\cos\theta))$,
and therefore
\begin{align*}
    \tau_\zeta e^\xi
        =\frac{\ell}{m\cos\theta}
        \geq\frac{\ell}{m}>r_c
\end{align*}
for all sufficiently large $m$, because $mr_c\to0$
as $m\to\infty$, uniformly for $|\zeta|\leq1$.
Thus, the metrics defined using neighbouring cells
both equal $\pi^*g_{\mathbb T}$ near their common side. Therefore, we have defined a
$\Gamma_m$-invariant metric $g_*$ on $M_\zeta$.

In the following, a subscript ``$*$'' on the weighted norms and the spaces
in \eqref{eq:s4-transverse-spaces} means that
$g_\zeta$ is replaced by $g_*$, even though we maintain the same scale $\rho_\zeta$ and weights.
\begin{lemma}[Uniform inverse for the auxiliary Laplacian]
    \label{lem:s4-model-inverse}
    The operator
    $\Delta_{g_*}:\mathcal X_{\zeta,*}\to\mathcal Y_{\zeta,*}$
    is an isomorphism, and
    \begin{align}
        \label{eq:s4-model-inverse-bound}
        \|b\|_{\mathcal X_{\zeta,*}}
            \leq C\|\Delta_{g_*}b\|_{\mathcal Y_{\zeta,*}}
            \qquad\forall\,b\in\mathcal X_{\zeta,*},
    \end{align}
    for some $C>0$ independent of $m$ and $\zeta$.
\end{lemma}
\begin{proof}
Since $g_*$ is $\Gamma_m$-invariant, $\Delta_{g_*}$
preserves $\chi$-equivariance. We first prove the
coercivity needed to solve the equation.

Assume that $b\in H^1(D_{m,\zeta,p},g_*)$ has zero trace on the cell diagonal $\theta=\pi/4$. For almost every fixed $\xi$, we have $\{\theta\in\mathbb{R} \mbox{ : } (\xi,\theta)\in\mathcal{W}_{m,\zeta}\}=(\pi/4-L_{\xi},\pi/4)$, for some $0<L_{\xi}\leq\pi/4$. Letting $\tilde b:=\Theta_{m,\zeta,p}^*b$, the even extension of $\tilde b(\xi,\,\cdot\,):(\pi/4-L_{\xi},\pi/4)\to\mathbb{R}$ across the left endpoint has zero trace at both ends of an interval of length $2L_{\xi}$. The Poincar\'e inequality therefore gives
\begin{align*}
    \int_{\frac{\pi}{4}-L_{\xi}}^{\frac{\pi}{4}}
        |\partial_\theta \tilde b(\xi,\theta)|^2\,d\theta\geq\frac{\pi^2}{4L_{\xi}^2}\int_{\frac{\pi}{4}-L_{\xi}}^{\frac{\pi}{4}}
            \tilde b(\xi,\theta)^2\,d\theta\geq4\int_{\frac{\pi}{4}-L_{\xi}}^{\frac{\pi}{4}}
            \tilde b(\xi,\theta)^2\,d\theta.
\end{align*}
Integrating in $\xi$ and using
$\Theta_{m,\zeta,p}^*g_*
=r_{m,\zeta}^2(d\xi^2+d\theta^2)$, $\rho_\zeta\circ\Theta_{m,\zeta,p}=r_{m,\zeta}$,
we obtain
\begin{align}
    \label{eq:s4-angular-coercivity}
    \begin{split}
        4\int_{D_{m,\zeta,p}}
            \rho_\zeta^{-2}b^2\,\dvol_{g_*}&=4\int_{\mathcal W_{m,\zeta}}
            \tilde b^2\,d\xi\,d\theta\\
            &\leq\int_{\mathcal W_{m,\zeta}}
            \bigl(|\partial_\xi\tilde b|^2
                 +|\partial_\theta \tilde b|^2\bigr)\,d\xi\,d\theta=\int_{D_{m,\zeta,p}}
            |\nabla^{g_*}b|_{g_*}^2\,\dvol_{g_*}.
    \end{split}
\end{align}
Consider the closed subspace
\begin{align*}
    H^1_\chi(M_\zeta)
        :=\{b\in H^1(M_\zeta,g_*)\mbox{ : }
            b\circ a=\chi(a)b
            \ \forall\,a\in\Gamma_m\}.
\end{align*}
Since the character $\chi$ is negative on the diagonal reflection, every function in this subspace has zero trace on the diagonal $\theta=\pi/4$. Summing \eqref{eq:s4-angular-coercivity} over the sectors covering $M_\zeta$, and using $\rho_\zeta\leq\pi/m$, gives
\begin{align*}
    \frac{4m^2}{\pi^2}\int_{M_\zeta}b^2\,\dvol_{g_*}
    \leq4\int_{M_\zeta}\rho_\zeta^{-2}b^2\,\dvol_{g_*}
    \leq\int_{M_\zeta}
        |\nabla^{g_*}b|_{g_*}^2\,\dvol_{g_*}
    \qquad\forall\,b\in H^1_\chi(M_\zeta).
\end{align*}
For every $E\in\mathcal Y_{\zeta,*}$, the Lax--Milgram theorem therefore gives a unique $b\in H^1_\chi(M_\zeta)$ satisfying
\begin{align*}
    \int_{M_\zeta}
        \langle\nabla^{g_*}b,\nabla^{g_*}w\rangle_{g_*}
        \,\dvol_{g_*}
    =-\int_{M_\zeta}Ew\,\dvol_{g_*}
    \qquad\forall\,w\in H^1_\chi(M_\zeta).
\end{align*}
Fix any arbitrary $w\in H^1(M_\zeta,g_*)$ and note that its average
\begin{align*}
    w_\chi:=\frac{1}{|\Gamma_m|}
        \sum_{a\in\Gamma_m}\chi(a)(w\circ a)
\end{align*}
belongs to $H^1_\chi(M_\zeta)$. Changing variables
by the elements of $\Gamma_m$, and using the
equivariance of $b$ and $E$, gives
\begin{align*}
    \int_{M_\zeta}
        \langle\nabla^{g_*}b,\nabla^{g_*}w\rangle_{g_*}
        \,\dvol_{g_*}
    &=\int_{M_\zeta}
        \langle\nabla^{g_*}b,\nabla^{g_*}w_\chi\rangle_{g_*}
        \,\dvol_{g_*}\\
    &=-\int_{M_\zeta}Ew_\chi\,\dvol_{g_*}
     =-\int_{M_\zeta}Ew\,\dvol_{g_*}.
\end{align*}
Hence, $\Delta_{g_*}b=E$ holds weakly on $M_\zeta$.
Standard elliptic regularity gives
$b\in C^{2,\beta}_\chi(M_\zeta)$, proving bijectivity of $\Delta_{g_{\ast}}$.

It remains to prove \eqref{eq:s4-model-inverse-bound}, with a constant $C>0$ independent of $m$ and $\zeta$. Let 
\begin{align*}
    A_0:=\sup_{M_\zeta}\rho_\zeta^2|E|
\end{align*}
and define $\Phi:D_{m,\zeta,p}\subset M_{\zeta}\to\mathbb R$ by
\begin{align*}
    \Phi(\Theta_{m,\zeta,p}(\xi,\theta))
        :=\frac{\pi^2}{16}-\theta^2 \qquad\forall\,(\xi,\theta)\in\mathcal{W}_{m,\zeta}
\end{align*}
The formula for the metric gives
\begin{align*}
    (-\Delta_{g_*}\Phi)\circ\Theta_{m,\zeta,p}=-r_{m,\zeta}^{-2}(\partial_\xi^2+\partial_\theta^2)\left(\frac{\pi^2}{16}-\theta^2\right)=2r_{m,\zeta}^{-2} \qquad\mbox{ on } \,\mathcal{W}_{m,\zeta}.
\end{align*}
Thus, we have $-\Delta_{g_*}\Phi=2\rho_\zeta^{-2}$.
Moreover, $\Phi$ vanishes on the diagonal and has
zero outward normal derivative on the axis and waist.

Let $\eta_{\mathrm{out}}$ denote the outward
$g_*$-unit normal to the boundary arcs of
$D_{m,\zeta,p}$. On the cell wall
$\xi=\log(\ell/(m\tau_\zeta\cos\theta))$, we have
\begin{align*}
    \eta_{\mathrm{out}}\circ\Theta_{m,\zeta,p}
        &=\frac{1}{r_{m,\zeta}}
            d\Theta_{m,\zeta,p}
            (\cos\theta\,\partial_\xi
             -\sin\theta\,\partial_\theta),\\
    (\partial_{\eta_{\mathrm{out}}}\Phi)
        \circ\Theta_{m,\zeta,p}
        &=\frac{2\theta\sin\theta}{r_{m,\zeta}}
        \geq0.
\end{align*}
Since the character $\chi$ is positive on the reflections across the axis, waist and cell wall, the solution $b$ has zero normal derivative on these boundary arcs; see Figure~\ref{fig:s4-reflected-domain}. Consequently, for either choice of sign, $w:=A_0\Phi\pm b$ satisfies
\begin{align*}
    -\Delta_{g_*}w
        =2A_0\rho_\zeta^{-2}\mp E
        \geq A_0\rho_\zeta^{-2}\geq0,
\end{align*}
with zero trace on the diagonal and nonnegative
outward normal derivative on the other boundary arcs.
Its negative part $w^-:=\max\{-w,0\}$ also has zero
diagonal trace. Integration by parts therefore gives
\begin{align*}
    -\int_{D_{m,\zeta,p}}
        |\nabla^{g_*}w^-|_{g_*}^2\,\dvol_{g_*}=\int_{D_{m,\zeta,p}}
        (-\Delta_{g_*}w)w^-\,\dvol_{g_*}+\int_{\partial D_{m,\zeta,p}}
        (\partial_{\eta_{\mathrm{out}}}w)w^-
        \,d\sigma_{g_*}
    \geq0.
\end{align*}
By \eqref{eq:s4-angular-coercivity}, $w^-=0$.
Applying this conclusion to both signs yields
$|b|\leq A_0\Phi$ on $D_{m,\zeta,p}$.
Since $|b|$ is $\Gamma_m$-invariant, we obtain
\begin{align}
    \label{eq:s4-model-barrier}
    \|b\|_{L^\infty(M_\zeta)}
        \leq\frac{\pi^2}{16}A_0
        \leq\frac{\pi^2}{16}
            \|E\|_{\mathcal Y_{\zeta,*}}.
\end{align}
We are just left to estimate the derivatives of $b$. We first verify that the interior Schauder estimates can be applied with constants independent of $m$
and $\zeta$. For $r_c/2\leq\tau_\zeta e^\xi\leq r_c$, the definition of $r_{m,\zeta}$ gives
\begin{align*}
    \left|\partial_\xi^j
        \left(\frac{2r_{m,\zeta}(\xi)}
                    {\tau_\zeta e^\xi}-1\right)\right|=\left|\partial_\xi^j
        \left(e^{-2\xi}
            \eta_*\left(\frac{\tau_\zeta e^\xi}{2r_c}\right)
        \right)\right|\leq C_j(\tau_\zeta/r_c)^2
        \qquad\forall\,j\in\{0,1,2,3\}.
\end{align*}
Together with the exact formulas
$r_{m,\zeta}=\tau_\zeta\cosh\xi$ for
$\tau_\zeta e^\xi\leq r_c/2$ and
$r_{m,\zeta}=\tau_\zeta e^\xi/2$ for
$\tau_\zeta e^\xi\geq r_c$, this implies
\begin{align*}
    |\nabla^{g_*}\rho_\zeta|_{g_*}
        \circ\Theta_{m,\zeta,p}
    =\frac{|r_{m,\zeta}'|}{r_{m,\zeta}}
    \leq2
\end{align*}
for sufficiently large $m$. At cell walls,
$\rho_\zeta$ is the minimum of the adjacent distance
functions and $g_*=\pi^*g_{\mathbb T}$.
Hence, $\rho_\zeta$ is $2$-Lipschitz on $M_\zeta$.
Fix a sufficiently large constant $\Lambda>4$.
If $\rho_\zeta(x)\leq\Lambda\tau_\zeta$, the
Lipschitz bound gives
\begin{align*}
    \rho_\zeta(y)\leq2\Lambda\tau_\zeta
    \qquad\forall\,y\in
        B_{g_*}(x,\rho_\zeta(x)/2).
\end{align*}
Since $\tau_\zeta/r_c\to0$ as $m\to\infty$,
uniformly for $|\zeta|\leq1$, this ball lies in the
exact catenoidal region for sufficiently large $m$.
The rescaled catenoid charts and injectivity estimate
of Appendix~\ref{app:s4-initial-geometry} therefore
apply.
If $\rho_\zeta(x)\geq\Lambda\tau_\zeta$, consider
the sheet containing $x$ over
$B_{\mathbb T}(\pi(x),\rho_\zeta(x)/2)$.
Every point of this disk has distance at least
$\rho_\zeta(x)/2>\tau_\zeta$ from $L_m$, so the
sheet is a graph throughout the disk.
Abbreviating $s=r_{m,\zeta}(\xi)$, we have
\begin{align*}
    \Theta_{m,\zeta,p}^*g_*=
        \left(\frac{r_{m,\zeta}(\xi)}
                   {r_{m,\zeta}'(\xi)}\right)^2
        ds^2+s^2d\theta^2,
\end{align*}
where
\begin{align*}
    \left(\frac{r_{m,\zeta}(\xi)}
               {r_{m,\zeta}'(\xi)}\right)^2
    =
    \begin{cases}
        (1-\tau_\zeta^2/s^2)^{-1},
            &\tau_\zeta e^\xi\leq r_c/2,\\
        1+O((\tau_\zeta/r_c)^2),
            &r_c/2\leq\tau_\zeta e^\xi\leq r_c,\\
        1,
            &\tau_\zeta e^\xi\geq r_c.
    \end{cases}
\end{align*}
The preceding derivative estimates also bound the
error in the middle line after one or two
applications of $s\partial_s$.
Since $s\geq\rho_\zeta(x)/2\geq\Lambda\tau_\zeta/2$
on the disk, rescaling coordinates and lengths by
$\rho_\zeta(x)^{-1}$ makes the metric
$O(\Lambda^{-2})+O((\tau_\zeta/r_c)^2)$ close to the
Euclidean metric in $C^2$. This comparison also
holds across cell walls, where
$g_*=\pi^*g_{\mathbb T}$.
For $\Lambda$ and then $m$ sufficiently large, the
exponential-map argument of
Appendix~\ref{app:s4-initial-geometry} therefore
gives
$\operatorname{inj}_{g_*}(x)\geq\rho_\zeta(x)/10$
and uniform bounds for the rescaled metric
coefficients. Thus, in both cases, interior Schauder estimates apply uniformly in $m$ and $\zeta$ at scale $\rho_\zeta(x)$.

Let $r:=\rho_\zeta(x)$. The Lipschitz bound on $\rho_{\zeta}$ gives $9r/10\leq\rho_\zeta\leq11r/10$ on $B_{g_*}(x,r/20)$. Hence, for $y,z\in B_{g_*}(x,r/20)$ satisfying $d_{g_*}(y,z)<r/200$ we have
$y,z\in B_{g_*}(y,\rho_\zeta(y)/100)$, and the definition of $\mathcal Y_{\zeta,*}$ gives
\begin{align*}
    |E(y)-E(z)|
        \leq Cr^{-2-\beta}
            \|E\|_{\mathcal Y_{\zeta,*}}d_{g_*}(y,z)^\beta.
\end{align*}
If instead $d_{g_*}(y,z)\geq r/200$, the same inequality
follows from the supremum bound
$|E|\leq Cr^{-2}\|E\|_{\mathcal Y_{\zeta,*}}$.
Consequently,
\begin{align*}
    r^2\|E\|_{C^{0,\beta}(B_{g_*}(x,r/20),r^{-2}g_*)}
        \leq C\|E\|_{\mathcal Y_{\zeta,*}}.
\end{align*}
Therefore, the local Schauder estimates yield
\begin{align*}
    \|b\|_{C^{2,\beta}
        (B_{g_*}(x,r/100),r^{-2}g_*)}
    &\leq C\left(
        \|b\|_{L^\infty(M_\zeta)}
        +r^2\|E\|_{C^{0,\beta}
            (B_{g_*}(x,r/20),r^{-2}g_*)}
        \right)\\
    &\leq C\left(
        \|b\|_{L^\infty(M_\zeta)}
        +\|E\|_{\mathcal Y_{\zeta,*}}
        \right).
\end{align*}
Taking the supremum over $x\in M_\zeta$ and applying
\eqref{eq:s4-model-barrier} proves
\eqref{eq:s4-model-inverse-bound}.
\end{proof}
\subsection{A uniform inverse for the linearised operator}
We now compare $g_*$ with $g_\phi$ and show that
$P_{\zeta,\phi}$ is a small perturbation of $\Delta_{g_{\ast}}$ in the appropriate norm.
\begin{lemma}[Comparison of the induced and auxiliary metrics]
    \label{lem:s4-metric-comparison}
    Let $\eta_m:[-1,1]\to(0,+\infty)$ be given by
    \begin{align}
        \label{eq:s4-comparison-smallness}
        \eta_m(\zeta):=\left(\frac{m^2\tau_\zeta}{2\pi}\right)^2+\tau_\zeta^{2-2\alpha}
            +\tau_\zeta^{1-\alpha/4}+\tau_\zeta^{1+\alpha/4}m^{-1/2}.
    \end{align}
    Then, $\eta_m(\zeta)\to0$ uniformly in $\zeta$ as $m\to\infty$, and
    \begin{align}
        \label{eq:s4-metric-comparison}
        \begin{split}
            \|g_\zeta-g_*\|_{C^{1,\beta}_{\zeta,*}(\rho_\zeta^2)}
                &\leq C\bigg(\left(\frac{m^2\tau_\zeta}{2\pi}\right)^2+\tau_\zeta^{2-2\alpha}\bigg),\\
            \|g_\phi-g_*\|_{C^{1,\beta}_{\zeta,*}(\rho_\zeta^2)}
                &\leq C\eta_m(\zeta).
        \end{split}
    \end{align}
    Consequently, there is $C>0$, independent of $m$ and $\zeta\in[-1,1]$, such that
    \begin{align}
    \label{eq:s4-norm-equivalence}
    \begin{split}
        C^{-1}\|b\|_{\mathcal X_\zeta}
            &\leq\|b\|_{\mathcal X_{\zeta,*}}
            \leq C\|b\|_{\mathcal X_\zeta},\\
        C^{-1}\|E_1\|_{\mathcal Y_\zeta}
            &\leq\|E_1\|_{\mathcal Y_{\zeta,*}}
            \leq C\|E_1\|_{\mathcal Y_\zeta},\\
        C^{-1}\|E_2\|_{\mathcal Z_\zeta}
            &\leq\|E_2\|_{\mathcal Z_{\zeta,*}}
            \leq C\|E_2\|_{\mathcal Z_\zeta}
    \end{split}
    \end{align}
    for every $b\in\mathcal X_\zeta$, $E_1\in\mathcal{Y}_{\zeta}$, and $E_2\in\mathcal Z_\zeta$. Moreover,
    \begin{align}
        \label{eq:s4-Laplacian-comparison}
        \|(\Delta_{g_\phi}-\Delta_{g_*})b\|_{\mathcal Y_\zeta}
            \leq C\eta_m(\zeta)\|b\|_{\mathcal X_\zeta}
            \qquad\forall\,b\in\mathcal X_\zeta.
    \end{align}
\end{lemma}
The proof of Lemma~\ref{lem:s4-metric-comparison} is given in Appendix~\ref{app:s4-normal}.
\begin{proposition}[Uniform inverse in the direction $n$]
    \label{prop:s4-transverse-inverse}
    For all sufficiently large $m$, the operator
    $P_{\zeta,\phi}:\mathcal X_\zeta\to\mathcal Y_\zeta$ is an isomorphism. Its inverse $\mathscr G_{\zeta,\phi}$ satisfies
    \begin{align}
        \label{eq:s4-actual-inverse}
        \|\mathscr G_{\zeta,\phi}E\|_{\mathcal X_\zeta}
            \leq C\|E\|_{\mathcal Y_\zeta}
            \qquad\forall\,E\in\mathcal Y_\zeta.
    \end{align}
    Furthermore,
    \begin{align}
        \label{eq:s4-forcing-and-mixed}
        \begin{split}
            \|H_{\iota,n}\|_{\mathcal Y_\zeta}&\leq C\tau_\zeta^2,\\
            \|Bb\|_{\mathcal Z_\zeta}
                &\leq C\tau_\zeta\|b\|_{\mathcal X_\zeta}
                \qquad\forall\,b\in\mathcal X_\zeta.
        \end{split}
    \end{align}
\end{proposition}

\begin{proof}
Since $g_\phi$ and $q$ are $\Gamma_m$-invariant,
$P_{\zeta,\phi}$ maps $\mathcal X_\zeta$ into
$\mathcal Y_\zeta$. We first estimate its difference from $\Delta_{g_*}$. By Lemma~\ref{lem:s4-full-family-geometry} and since $\rho_\zeta\leq\pi/m$, we have
\begin{align*}
    \|qb\|_{\mathcal Y_\zeta}
        &\leq C\bigg(\sup_{M_\zeta}\rho_\zeta^2\bigg)
            \|q\|_{C^{0,\beta}_\zeta(1)}
            \|b\|_{\mathcal X_\zeta}
        \leq Cm^{-2}\|b\|_{\mathcal X_\zeta}.
\end{align*}
Consequently, setting
$\mathcal D:=P_{\zeta,\phi}-\Delta_{g_*}$ and using
\eqref{eq:s4-Laplacian-comparison}, we obtain
\begin{align*}
    \|\mathcal D b\|_{\mathcal Y_\zeta}
        \leq C\bigl(\eta_m(\zeta)+m^{-2}\bigr)
            \|b\|_{\mathcal X_\zeta}.
\end{align*}
Let $K_*:=(\Delta_{g_*})^{-1}$ be the inverse of $\Delta_{g_*}$ given by Lemma~\ref{lem:s4-model-inverse}. By \eqref{eq:s4-model-inverse-bound} and \eqref{eq:s4-norm-equivalence}, $K_*$ is uniformly bounded from $\mathcal Y_\zeta$ to $\mathcal X_\zeta$. Thus, for all sufficiently large $m$,
\begin{align*}
    \|K_*\mathcal D\|_{\mathcal X_\zeta\to\mathcal X_\zeta}
        \leq\tfrac12,\qquad
    P_{\zeta,\phi}
        =\Delta_{g_*}(\operatorname{Id}_{\mathcal X_\zeta}+K_*\mathcal D).
\end{align*}
The Neumann series therefore gives the inverse
\begin{align*}
    \mathscr G_{\zeta,\phi}
        =\sum_{j=0}^\infty(-K_*\mathcal D)^jK_*,\qquad
    \|\mathscr G_{\zeta,\phi}\|_{\mathcal Y_\zeta\to\mathcal X_\zeta}
        \leq2\|K_*\|_{\mathcal Y_\zeta\to\mathcal X_\zeta}\leq C.
\end{align*}
This proves \eqref{eq:s4-actual-inverse}.

The estimate for $H_{\iota,n}$ is part of
Lemma~\ref{lem:s4-full-family-geometry}. To estimate $B$,
recall that
\begin{align*}
    Bb=-2\langle\omega,db\rangle_{g_\phi}
        -(\operatorname{div}_{g_\phi}\omega)b
        +\langle h_\iota^\nu,h_\iota^n\rangle_{g_\phi}b.
\end{align*}
Lemma~\ref{lem:s4-full-family-geometry} bounds $\omega$, $\nabla^{g_\phi}\omega$,
$h_\iota^\nu$ and $h_\iota^n$, respectively, by
$C\tau_\zeta/\rho_\zeta$, $C\tau_\zeta/\rho_\zeta^2$,
$C\tau_\zeta/\rho_\zeta^2$ and $C$, with the corresponding
local H\"older estimates. Since
$|db|_{g_\phi}\leq C\rho_\zeta^{-1}\|b\|_{\mathcal X_\zeta}$,
the product estimate gives
\begin{align*}
    \|\langle\omega,db\rangle_{g_\phi}\|_{\mathcal Z_\zeta}
        +\|(\operatorname{div}_{g_\phi}\omega)b\|_{\mathcal Z_\zeta}
        +\|\langle h_\iota^\nu,h_\iota^n\rangle_{g_\phi}b\|_{\mathcal Z_\zeta}
        \leq C\tau_\zeta\|b\|_{\mathcal X_\zeta}.
\end{align*}
Hence \eqref{eq:s4-forcing-and-mixed} follows.
\end{proof}

\subsection{The normal correction and the remaining scalar equation}
Consider the second order remainders $Q_n:\mathcal U_{\zeta,c_0}\to\mathcal{Y}_{\zeta}$ and $Q_{\nu}:\mathcal U_{\zeta,c_0}\to\mathcal{Z}_{\zeta}$ given by 
\begin{align}
    \label{eq:s4-transverse-remainders}
    \begin{split}
        Q_n(b)&:=\mathscr P_{\zeta,\phi}(b)-\big(\mathscr P_{\zeta,\phi}(0)+D\mathscr P_{\zeta,\phi}(0)b\big)=\mathscr{P}_{\zeta,\phi}(b)-H_{\iota,n}-P_{\zeta,\phi}b,\\
        Q_\nu(b)&:=\mathscr T_{\zeta,\phi}(b)-\big(\mathscr T_{\zeta,\phi}(0)+D\mathscr T_{\zeta,\phi}(0)b\big)=\mathscr T_{\zeta,\phi}(b)-H_{\zeta,\phi}-Bb \qquad\forall\,b\in\mathcal U_{\zeta,c_0},
    \end{split}
\end{align}
where the second equalities in the above equations follow from Lemma~\ref{lem:s4-component-linearisation}.
\begin{lemma}[Quadratic estimates for the two components]
    \label{lem:s4-transverse-quadratic}
    After decreasing $c_0>0$ if necessary, for every
    $b_1,b_2\in\mathcal U_{\zeta,c_0}$ we have
    \begin{align}
        \label{eq:s4-transverse-quadratic}
        \begin{split}
            \|Q_n(b_2)-Q_n(b_1)\|_{\mathcal Y_\zeta}
                +\|Q_\nu(b_2)-Q_\nu(b_1)\|_{\mathcal Z_\zeta}\le&C\tau_\zeta^{-1}
                \bigl(\|b_1\|_{\mathcal X_\zeta}
                    +\|b_2\|_{\mathcal X_\zeta}\bigr)
                \|b_2-b_1\|_{\mathcal X_\zeta}.
        \end{split}
    \end{align}
    In particular,
    $\|Q_n(b)\|_{\mathcal Y_\zeta}+\|Q_{\nu}(b)\|_{\mathcal Z_\zeta}
        \leq C\tau_\zeta^{-1}\|b\|_{\mathcal X_\zeta}^2$ for every $b\in \mathcal U_{\zeta,c_0}$.
\end{lemma}
The proof of Lemma~\ref{lem:s4-transverse-quadratic} is given in Appendix~\ref{app:s4-normal} and it uses only the $C^{2,\beta}$ bound for $\phi$.
\begin{proposition}[Correction in the direction $n$]
    \label{prop:s4-transverse-elimination}
    There exist $K>0$ and $m_0\in\mathbb{N}$ such that, for every
    $m\geq m_0$ and $(v,\zeta)\in\mathcal B_m$, there
    is a unique solution $b=b(v,\zeta)$ in
    \begin{align}
        \label{eq:s4-transverse-ball}
        \mathcal{V}_{\zeta,K}:=\{b\in\mathcal X_\zeta\mbox{ : }
            \|b\|_{\mathcal X_\zeta}\leq K\tau_\zeta^2\}
    \end{align}
    of the nonlinear equation
    \begin{align}
        \label{eq:s4-transverse-equation}
        \mathscr P_{\zeta,\phi}(b)=0.
    \end{align}
    The first order remainder in the other mean-curvature component, given by
    \begin{align}
        \label{eq:s4-feedback-definition}
        \mathcal R_\zeta(\phi)
            :=\mathscr T_{\zeta,\phi}(b(v,\zeta))-H_{\zeta,\phi}
    \end{align}
    belongs to $\mathcal{Z}_{\zeta}$ and satisfies the estimates
    \begin{align}
        \label{eq:s4-feedback-estimates}
        \|\mathcal R_\zeta(\phi)\|_{\mathcal Z_\zeta}
            \leq C\tau_\zeta^3,\qquad
        \|\mathcal R_\zeta(\phi)\|_{Y_\zeta}
            \leq C\tau_\zeta^{2+\alpha/2}.
    \end{align}
    In particular, on each neck, its projection onto $\operatorname{span}\{\cos\theta,\sin\theta\}$ vanishes
    on every circle of fixed axial coordinate.

    Lastly, fix $m$ and $0<\beta'<\beta$. If
    $(v_j,\zeta_j),(v,\zeta)\in\mathcal B_m$,
    $\zeta_j\to\zeta$ and $v_j\to v$ in
    $C^{2,\beta'}(M_0)$, then
    \begin{align}
        \label{eq:s4-transverse-continuity}
        \begin{aligned}
            b(v_j,\zeta_j)\circ F_{\zeta_j}
                &\to b(v,\zeta)\circ F_\zeta
                &&\quad\text{in }C^{2,\beta'}(M_0),\\
            \mathcal R_{\zeta_j}(\phi_j)\circ F_{\zeta_j}
                &\to\mathcal R_\zeta(\phi)\circ F_\zeta
                &&\quad\text{in }C^{0,\beta'}(M_0),
        \end{aligned}
    \end{align}
    where $\phi_j:=u_{\zeta_j}+v_j\circ F_{\zeta_j}^{-1}$.
\end{proposition}
\begin{proof}
Let $C_L$ and $C_H$ be the constants in
\eqref{eq:s4-actual-inverse} and in the estimate for $H_{\iota,n}$ in \eqref{eq:s4-forcing-and-mixed} respectively. Choose $K>2C_LC_H$ and define the map $\mathcal{A}_{\zeta,\phi}:\mathcal V_{\zeta,K}\to\mathcal{X}_{\zeta}$ by
\begin{align}
    \label{eq:s4-transverse-contraction}
    \mathcal A_{\zeta,\phi}(b)
        :=-\mathscr G_{\zeta,\phi}(H_{\iota,n}+Q_n(b)) \qquad\forall\,b\in\mathcal V_{\zeta,K}
\end{align}
Since $P_{\zeta,\phi}$ is invertible, the equation
$\mathscr P_{\zeta,\phi}(b)=0$ is equivalent to
$\mathcal A_{\zeta,\phi}(b)=b$, i.e. $b$ solves $\mathscr P_{\zeta,\phi}(b)=0$ if and only if it is a fixed point for $\mathcal A_{\zeta,\phi}$. 
Since $\tau_{\zeta}\to 0$ uniformly in $\zeta\in[-1,1]$ as $m\to+\infty$, for $m$ sufficiently large we have that $\mathcal{V}_{\zeta,K}\subset\mathcal{U}_{\zeta,c_0}$. Therefore, by Proposition~\ref{prop:s4-transverse-inverse} and Lemma~\ref{lem:s4-transverse-quadratic}, we have
\begin{align*}
    \|\mathcal A_{\zeta,\phi}(b)\|_{\mathcal X_\zeta}
        &\leq C_LC_H\tau_\zeta^2+CK^2\tau_\zeta^3 \qquad\forall\, b\in\mathcal V_{\zeta,K},\\
    \|\mathcal A_{\zeta,\phi}(b_2)
        -\mathcal A_{\zeta,\phi}(b_1)\|_{\mathcal X_\zeta}
        &\leq CK\tau_\zeta\|b_2-b_1\|_{\mathcal X_\zeta} \qquad\forall\, b_1,b_2\in\mathcal V_{\zeta,K},
\end{align*}
For $m$ sufficiently large, the right-hand side of the first inequality is at most $K\tau_\zeta^2$, and $CK\tau_\zeta\leq1/2$. Thus, $\mathcal A_{\zeta,\phi}$ is a contraction of the closed ball $\mathcal{V}_{\zeta,K}$ into itself. The contraction lemma then gives the required unique solution.

By \eqref{eq:s4-transverse-remainders},
$\mathcal R_\zeta(\phi)=Bb+Q_\nu(b)$. Consequently,
\begin{align*}
    \|\mathcal R_\zeta(\phi)\|_{\mathcal Z_\zeta}
        &\leq C\tau_\zeta\|b\|_{\mathcal X_\zeta}
            +C\tau_\zeta^{-1}\|b\|_{\mathcal X_\zeta}^2
        \leq C\tau_\zeta^3,\\
    \|\mathcal R_\zeta(\phi)\|_{Y_\zeta}
        &\leq\tau_\zeta^{-1+\alpha/2}
            \|\mathcal R_\zeta(\phi)\|_{\mathcal Z_\zeta}
        \leq C\tau_\zeta^{2+\alpha/2},
\end{align*}
where the second line uses $f_\zeta\geq\tau_\zeta^{1-\alpha/2}$.
Moreover, the function $\mathcal R_\zeta(\phi)$ is $\Gamma_m$-invariant
by \eqref{eq:s4-component-parities}. In particular, its
pullback to each neck is unchanged by
$\theta\mapsto\theta+\pi$. Since both $\cos\theta$ and
$\sin\theta$ change sign under this substitution, their
integrals against that pullback vanish on every circle
of fixed axial coordinate.

It remains to prove continuity. Fix $m$ and a sequence as
in the statement. Proposition~\ref{prop:s4-scalar-input},
applied to the smoothly varying sources $D_{\zeta_j}$,
gives
\begin{align*}
    \phi_j\circ F_{\zeta_j}
        =u_{\zeta_j}\circ F_{\zeta_j}+v_j
        \rightarrow u_\zeta\circ F_\zeta+v
        =\phi\circ F_\zeta
        \quad\mbox{in }C^{2,\beta'}(M_0).
\end{align*}
The bound \eqref{eq:s4-transverse-ball} and the uniform
norm comparison under $F_{\zeta_j}$ imply that
$b(v_j,\zeta_j)\circ F_{\zeta_j}$ is bounded in
$C^{2,\beta}(M_0)$. Hence every subsequence has a further
subsequence converging in $C^{2,\beta'}(M_0)$ to a function
of the form $b_\infty\circ F_\zeta$.
The formulas for $X_b$ and $n_b$ express
$\mathscr P_{\zeta,\phi}(b)$ as a continuous function of
the first two derivatives of the parametrisation and of
$b$, with denominators bounded away from zero. Passing to
the limit in \eqref{eq:s4-transverse-equation} therefore gives
$\mathscr P_{\zeta,\phi}(b_\infty)=0$.
Equivariance of $F_{\zeta_j}$ preserves the character
condition, and lower semicontinuity of the weighted
H\"older norms gives
$\|b_\infty\|_{\mathcal X_\zeta}\leq K\tau_\zeta^2$.
Uniqueness now implies $b_\infty=b(v,\zeta)$, so the
whole sequence converges. Finally, the same continuity
of the curvature expressions in
\eqref{eq:s4-feedback-definition} proves the second
convergence in \eqref{eq:s4-transverse-continuity}.
\end{proof}

\begin{proposition}
    \label{prop:s4-remaining-scalar-equation}
    For all sufficiently large $m$ and every
    $(v,\zeta)\in\mathcal B_m$, the map
    \begin{align*}
        \widehat X_{v,\zeta}:=X_{b(v,\zeta)}:M_\zeta\to\mathbb S^4
    \end{align*}
    is a $C^{2,\beta}$ embedding of $M_{\zeta}$. Its image is minimal in $\mathbb{S}^4$ if and only if
    \begin{align}
        \label{eq:s4-remaining-scalar-equation}
        H_{\zeta,\phi}+\mathcal R_\zeta(\phi)=0
        \quad\text{on }M_\zeta.
    \end{align}
    Moreover,
    \begin{align}
        \label{eq:s4-transverse-two-jet-final}
        |(\nabla^{g_\zeta})^j b(v,\zeta)|_{g_\zeta}
            \leq C\tau_\zeta^2\rho_\zeta^{-j}
            \qquad\mbox{ on } M_\zeta,\quad \forall\,j\in\{0,1,2\},
    \end{align}
    with the local H\"older bounds of
    \eqref{eq:s4-transverse-ball}.
\end{proposition}

\begin{proof}
Since $G$ is an embedding and $\mathcal F$ is a
diffeomorphism onto $\mathcal N$, the map
\begin{align*}
    x\mapsto(\mathcal F^{-1}(G(x)),b(v,\zeta)(x))
\end{align*}
is an embedding. By Lemma~\ref{lem:s4-full-family-geometry}
and Proposition~\ref{prop:s4-transverse-elimination},
\begin{align*}
    |z\circ G|\leq C\frac{m^2\tau_\zeta}{2\pi},\qquad
    |b(v,\zeta)|\leq K\tau_\zeta^2.
\end{align*}
Both bounds tend to zero uniformly as $m\to\infty$.
Thus the image lies in the domain of the tubular
embedding $\Psi$, and
\begin{align*}
    \widehat X_{v,\zeta}(x)
        =\Psi\bigl(\mathcal F^{-1}(G(x)),b(v,\zeta)(x)\bigr)
\end{align*}
is a $C^{2,\beta}$ embedding of the closed connected
surface $M_\zeta$.

By the defining equations for $b=b(v,\zeta)$ and
$\mathcal R_\zeta(\phi)$, we have
\begin{align*}
    \langle\mathbf H_{\widehat X_{v,\zeta}},n_b\rangle=0,
    \qquad
    \langle\mathbf H_{\widehat X_{v,\zeta}},\nu_b\rangle
        =H_{\zeta,\phi}+\mathcal R_\zeta(\phi).
\end{align*}
Since $n_b$ and $\nu_b$ span the normal plane, the
mean-curvature vector vanishes if and only if the second
expression vanishes. This proves
\eqref{eq:s4-remaining-scalar-equation}. Finally,
\eqref{eq:s4-transverse-two-jet-final} and its local
H\"older estimates follow directly from
$\|b(v,\zeta)\|_{\mathcal X_\zeta}\leq K\tau_\zeta^2$
and the definition of the weighted norm.
\end{proof}
\begin{remark}
    Note that Proposition~\ref{prop:s4-scalar-input} and
    \eqref{eq:s4-feedback-estimates} give
    \begin{align*}
        \|Q_\zeta(\phi)+\mathcal R_\zeta(\phi)\|_{Y_\zeta}
            \leq C\tau_\zeta^2.
    \end{align*}
    This estimate and \eqref{eq:s4-transverse-continuity} will allow us to solve the remaining scalar equation by choosing
    $(v,\zeta)\in\mathcal B_m$ in Section~\ref{sec:s4-coupled}.
\end{remark}
\section{Final correction to minimality}
\label{sec:s4-coupled}
We finally choose the scalar correction $v$ and the parameter $\zeta$ so that the remaining mean-curvature component $H_{\zeta,\phi}+\mathcal R_\zeta(\phi)$ vanishes. The correction $b=b(v,\zeta)$ has already been selected
in Section~\ref{sec:s4-transverse}. Thus, the fixed-point problem needed to obtain the final minimal surfaces must be solved in the unknowns $(v,\zeta)$, expressed on the fixed surface $M_0$. We then prove smoothness and record the bounds needed to preserve negative curvature.

Recall that $\delta=(100m)^{-1}$. The definition of $\tau_\zeta$ and Remark~\ref{rem:s4-waist-asymptotics} therefore give
\begin{align}
    \label{eq:s4-coupled-scale-comparison}
    e^{-1}\tau_0\leq\tau_\zeta\leq e\tau_0
    \quad\forall\,\zeta\in[-1,1], \qquad \delta^{-p}\tau_0^q\rightarrow0
    \qquad\mbox{ as }m\rightarrow\infty,
\end{align}
for every fixed $p\geq0$ and $q>0$.
\subsection{The remaining equation and the fixed-point map}
Fix any $(v,\zeta)\in\mathcal B_m$, and recall that $\phi(v,\zeta):=u_\zeta+v\circ F_\zeta^{-1}$, where $u_\zeta$ is the preliminary correction defined in \eqref{eq:s4-initial-correction}. Define the combined nonlinear remainder $E_\zeta(v)$ by
\begin{align}
    \label{eq:s4-coupled-source}
    E_\zeta(v):=Q_\zeta(\phi(v,\zeta))+\mathcal R_\zeta(\phi(v,\zeta)),
\end{align}
where the function $Q_\zeta(\phi)$ is the nonlinear remainder in the scalar mean curvature of the graph within $\mathcal N$, defined in \eqref{eq:s4-Q-definition}, while the function $\mathcal R_\zeta(\phi)$ is the first order remainder defined in \eqref{eq:s4-feedback-definition}. Thus,
\begin{align*}
    H_{\zeta,\phi}+\mathcal R_\zeta(\phi)=H_\zeta+\mathcal L_\zeta\phi+E_\zeta(v) \qquad\mbox{ on } M_\zeta.
\end{align*}
As the two summands defining $E_\zeta(v)$ are $\Gamma_m$-invariant, we have $E_{\zeta}(v)\in Y_\zeta$, and we can apply to $E_{\zeta}(v)$ the linear solution operator $\mathscr{R}_{\zeta}$ that we have constructed in Proposition~\ref{prop:s4-scalar-input}.
\begin{figure}[htbp]
    \centering
    \begingroup
    \definecolor{internalteal}{RGB}{0,112,116}
    \definecolor{internalorange}{RGB}{190,102,36}
    \begin{tikzpicture}[
        x=1cm,y=1cm,>=Stealth,
        every node/.style={font=\small},
        reference/.style={black!30,line width=0.6pt},
        initial edge/.style={internalteal,line width=0.9pt},
        deformed edge/.style={internalorange,line width=0.9pt}
    ]
        \path[fill=black!2]
            (0,0)--(7.7,0)--(8.9,0.9)--(8.9,4.2)--(1.2,4.2)--(0,3.3)--cycle;
        \draw[reference,dashed]
            (0,0)--(1.2,0.9)--(8.9,0.9)
            (1.2,0.9)--(1.2,4.2);
        \draw[reference]
            (0,0)--(7.7,0)--(8.9,0.9)--(8.9,4.2)--(1.2,4.2)--(0,3.3)--cycle
            (0,3.3)--(7.7,3.3)--(8.9,4.2)
            (7.7,0)--(7.7,3.3);
        \node[anchor=south west,text=black!65] at (1.35,4.24)
            {coordinate neighbourhood in $\mathcal N$};
        \path[fill=internalteal!12]
            (0.72,0.63)--(6.53,0.63)--(7.30,1.22)--(1.49,1.22)--cycle;
        \draw[initial edge]
            (0.72,0.63)--(6.53,0.63)--(7.30,1.22)--(1.49,1.22)--cycle;
        \path[fill=internalorange!12]
            (0.84,2.00)
            .. controls (2.61,2.37) and (4.13,1.92) .. (6.73,2.14)
            --(7.50,2.73)
            .. controls (4.90,2.51) and (3.38,2.96) .. (1.61,2.59)
            --cycle;
        \draw[deformed edge]
            (0.84,2.00)
            .. controls (2.61,2.37) and (4.13,1.92) .. (6.73,2.14)
            --(7.50,2.73)
            .. controls (4.90,2.51) and (3.38,2.96) .. (1.61,2.59)
            --cycle;
        \draw[black!35,dashed]
            (1.49,1.22) .. controls (1.49,1.69) and (1.53,2.13) .. (1.61,2.59);
        \draw[black!35,dashed]
            (6.53,0.63) .. controls (6.53,1.08) and (6.61,1.68) .. (6.73,2.14);
        \coordinate (startpoint) at (3.51,0.91);
        \coordinate (endpoint) at (3.81,2.48);
        \draw[black!65,->,line width=0.65pt]
            (startpoint)--(3.51,1.58);
        \node[anchor=east,text=black!70] at (3.35,1.48)
            {$\nu_\zeta(x)$};
        \draw[internalorange,->,line width=1.05pt]
            (startpoint) .. controls (3.51,1.42) and (3.58,1.99) .. (endpoint);
        \fill[internalteal] (startpoint) circle[radius=1.6pt];
        \fill[internalorange] (endpoint) circle[radius=1.6pt];
        \node[anchor=north] at (3.51,0.38) {$\jmath_\zeta(x)$};
        \draw[black!40] (3.51,0.43)--(3.51,0.80);
        \node[anchor=south] at (3.81,2.83) {$G_{\zeta,\phi}(x)$};
        \draw[black!40] (3.81,2.79)--(3.81,2.58);
        \node[anchor=west,align=left] at (4.22,1.66)
            {geodesic in $\mathcal N$\\length $|\phi(x)|$};
        \node[anchor=west,align=left,text=internalteal] at (9.18,1.08)
            {initial surface\\$M_\zeta$};
        \draw[internalteal!70] (9.04,1.12)--(7.28,1.12);
        \node[anchor=west,align=left,text=internalorange] at (9.18,2.72)
            {deformed surface\\$M_{\zeta,\phi}$};
        \draw[internalorange!70] (9.04,2.76)--(7.47,2.66);
    \end{tikzpicture}
    \endgroup
    \caption{The first deformation moves $x\in M_\zeta$ within
    $\mathcal N$ through signed distance $\phi(x)$, initially
    in direction $\nu_\zeta(x)$.}
    \label{fig:s4-internal-displacement}
\end{figure}
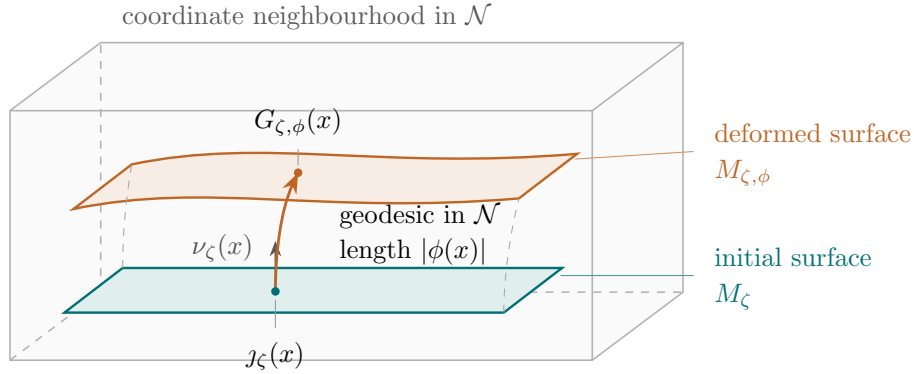
Define then
\begin{align}
    \label{eq:s4-coupled-inverse-output}
    (w_{v,\zeta},d_{v,\zeta}):=-\mathscr R_\zeta E_\zeta(v)\in X_\zeta\times\mathbb R.
\end{align}
By \eqref{eq:s4-scalar-inverse}, we have
\begin{align}
    \label{eq:s4-coupled-linear-equation}
    \mathcal L_\zeta w_{v,\zeta}
    =-E_\zeta(v)+d_{v,\zeta}\mathcal J_\zeta
    \qquad\mbox{on }M_\zeta.
\end{align}
Thus, $w_{v,\zeta}$ cancels $E_\zeta(v)$ in the linearised equation up to a multiple of the prescribed function $\mathcal J_\zeta$. As the preliminary correction $u_{\zeta}$ satisfies
\begin{align*}
    H_\zeta+\mathcal L_\zeta u_\zeta=(-\tau_\zeta\zeta+c_{H,\zeta})\mathcal J_\zeta,
\end{align*}
substituting $\phi=u_\zeta+v\circ F_\zeta^{-1}$ and $E_\zeta(v)=-\mathcal L_\zeta w_{v,\zeta}+d_{v,\zeta}\mathcal J_\zeta$ into the remaining component therefore gives the identity
\begin{align}
    \label{eq:s4-coupled-residual-identity}
    \begin{split}
        H_{\zeta,\phi}+\mathcal R_\zeta(\phi)
        =\mathcal L_\zeta
            \bigl(v\circ F_\zeta^{-1}-w_{v,\zeta}\bigr)+
            (-\tau_\zeta\zeta+c_{H,\zeta}+d_{v,\zeta})
            \mathcal J_\zeta
            \qquad\text{ on }M_\zeta.
    \end{split}
\end{align}
Define the map $\mathcal T_m:\mathcal B_m\to X_0\times\mathbb{R}$ by
\begin{align}
    \label{eq:s4-coupled-map}
    \mathcal T_m(v,\zeta):=\left(w_{v,\zeta}\circ F_\zeta,\frac{c_{H,\zeta}+d_{v,\zeta}}{\tau_\zeta}\right) \qquad\forall\,(v,\zeta)\in\mathcal B_m.
\end{align}
At a fixed point of $\mathcal T_m$, the first coordinate gives $v\circ F_\zeta^{-1}=w_{v,\zeta}$, and the second gives $\tau_\zeta\zeta=c_{H,\zeta}+d_{v,\zeta}$. Thus, \eqref{eq:s4-coupled-residual-identity} gives $H_{\zeta,\phi}+\mathcal R_\zeta(\phi)=0$. We first prove that $\mathcal T_m$ maps $\mathcal B_m$ into itself.

\begin{lemma}
    \label{lem:s4-coupled-self-map}
    There exist $C>0$ and $m_0\in\mathbb N$ such that, for every $m\geq m_0$ and $(v,\zeta)\in\mathcal B_m$, we have
    \begin{align}
        \|E_\zeta(v)\|_{Y_\zeta}
        &\leq C\tau_\zeta^2,
        \label{eq:s4-coupled-source-bound}\\
        \|w_{v,\zeta}\|_{X_\zeta}
        &\leq C\delta^{-4-2\beta}\tau_\zeta^2,
        \label{eq:s4-coupled-w-bound}\\
        |d_{v,\zeta}|
        &\leq C\delta^{-2-\beta}\tau_\zeta^2,
        \label{eq:s4-coupled-d-bound}\\
        |c_{H,\zeta}|
        &\leq C\delta^{-2-\beta}\tau_\zeta^{1+\alpha/3}.
        \label{eq:s4-coupled-cH-bound}
    \end{align}
    Hence, writing $\mathcal T_m=(\mathcal T_m^1,\mathcal T_m^2)$,
    it follows that
    \begin{align}
        \frac{\|\mathcal T_m^1(v,\zeta)\|_{X_0}}
        {\tau_0^{1+\alpha}}
        &\leq C\delta^{-4-2\beta}\tau_0^{1-\alpha}
        \rightarrow0,
        \label{eq:s4-coupled-function-smallness}\\
        |\mathcal T_m^2(v,\zeta)|
        &\leq C\delta^{-2-\beta}
        (\tau_0^{\alpha/3}+\tau_0)
        \rightarrow0
        \label{eq:s4-coupled-parameter-smallness}
    \end{align}
    as $m\to\infty$, uniformly on $\mathcal B_m$. In particular, for every sufficiently large integer $m$ the map $\mathcal T_m$ sends $\mathcal B_m$ into itself. 
\end{lemma}

\begin{proof}
By \eqref{eq:s4-scalar-quadratic} and
\eqref{eq:s4-feedback-estimates},
\begin{align*}
    \|E_\zeta(v)\|_{Y_\zeta}\leq\|Q_\zeta(\phi(v,\zeta))\|_{Y_\zeta}+\|\mathcal R_\zeta(\phi(v,\zeta))\|_{Y_\zeta}\leq C\tau_\zeta^2+C\tau_\zeta^{2+\alpha/2}\leq C\tau_\zeta^2.
\end{align*}
Applying \eqref{eq:s4-scalar-inverse} and \eqref{eq:s4-coefficient-bound} to $(w_{v,\zeta},d_{v,\zeta})=-\mathscr R_\zeta E_\zeta(v)$ therefore gives
\begin{align*}
    \|w_{v,\zeta}\|_{X_\zeta}&\leq C\delta^{-4-2\beta}\|E_\zeta(v)\|_{Y_\zeta}\leq C\delta^{-4-2\beta}\tau_\zeta^2,\\
    |d_{v,\zeta}|&\leq C\delta^{-2-\beta}\|E_\zeta(v)\|_{Y_\zeta}\leq C\delta^{-2-\beta}\tau_\zeta^2.
\end{align*}
Similarly, applying the coefficient estimate \eqref{eq:s4-coefficient-bound} to $(u_\zeta,c_{H,\zeta})=-\mathscr R_\zeta D_\zeta$ and using \eqref{eq:s4-H-error}, we obtain
\begin{align*}
    |c_{H,\zeta}|\leq C\delta^{-2-\beta}\|D_\zeta\|_{Y_\zeta}\leq C\delta^{-2-\beta}\tau_\zeta^{1+\alpha/3}.
\end{align*}
Hence, we proved the first four estimates in the statement.

By definition of $\mathcal T_m$, the norm comparison \eqref{eq:s4-norm-identification} under $F_\zeta$, and $e^{-1}\tau_0\leq\tau_\zeta\leq e\tau_0$,
\begin{align*}
    \frac{\|\mathcal T_m^1(v,\zeta)\|_{X_0}}{\tau_0^{1+\alpha}}
        &\leq\frac{C_F\|w_{v,\zeta}\|_{X_\zeta}}{\tau_0^{1+\alpha}}
        \leq C\delta^{-4-2\beta}\tau_0^{1-\alpha},\\
    |\mathcal T_m^2(v,\zeta)|
        &\leq\tau_\zeta^{-1}(|c_{H,\zeta}|+|d_{v,\zeta}|)
        \leq C\delta^{-2-\beta}(\tau_0^{\alpha/3}+\tau_0).
\end{align*}
Since $0<\alpha<1$, \eqref{eq:s4-coupled-scale-comparison} implies that both right-hand sides tend to zero as $m\to\infty$, uniformly on $\mathcal B_m$. In particular, they are bounded above by one for all sufficiently large $m$, uniformly on $\mathcal B_m$. Thus, $\mathcal T_m(\mathcal B_m)\subset\mathcal B_m$.
\end{proof}
\subsection{Compactness and continuity in a weaker H\"older topology}
For fixed $m$, the weighted $C^{2,\beta}$ bound defining $\mathcal B_m$ gives its compactness in a weaker H\"older topology. We prove continuity of the map $\mathcal T_m$ in that same topology.

Fix any $\beta'\in(0,\beta)$ and let
\begin{align*}
    \mathcal E_m&:=C^{2,\beta'}_{\mathrm{inv}}(M_0)\times\mathbb R,\\
    \|(v,\zeta)\|_{\mathcal E_m}&:=\|v\|_{C^{2,\beta'}(M_0,g_0)}+|\zeta| \qquad\forall\,(v,\zeta)\in\mathcal E_m.
\end{align*}
As invariance is preserved under uniform convergence, $C_{\mathrm{inv}}^{2,\beta'}(M_0)$ is a closed subspace of the ordinary H\"older space $C^{2,\beta'}(M_0)$. Therefore, $\mathcal E_m$ is a Banach space. Moreover, note that $\mathcal B_m\subset X_0\times\mathbb R\subset\mathcal E_m$.
\begin{lemma}
    \label{lem:s4-coupled-compactness}
    For every fixed sufficiently large $m$, the set $\mathcal B_m$ is a nonempty convex compact subset of $\mathcal E_m$. The map $\mathcal T_m:\mathcal B_m\to\mathcal B_m$ is continuous in the topology induced by $\mathcal E_m$.
\end{lemma}
\begin{proof}
The set $\mathcal B_m$ is the product of the closed ball of radius $\tau_0^{1+\alpha}$ in the stronger $X_0$-norm and the interval $[-1,1]$. As both factors are convex, $\mathcal B_m$ is convex as a subset of $C^{2,\beta'}(M_0)\times\mathbb R$ and contains $(0,0)$. Fix any $m$. As we have already noted, the norm on $X_0$ defined in \eqref{eq:s4-XY} is equivalent to the $C^{2,\beta}$ norm. The compact inclusion $C^{2,\beta}(M_0)\hookrightarrow C^{2,\beta'}(M_0)$ therefore gives relative compactness of $\mathcal B_m$ in $\mathcal E_m$.

Suppose that $(v_j,\zeta_j)\in\mathcal B_m$ converges to $(v,\zeta)$ in $\mathcal E_m$. Then, $\zeta\in[-1,1]$, and $v$ is $\Gamma_m$-invariant by uniform convergence. For every pair of distinct points in each defining ball, the $\beta$-H\"older quotient of the second derivative
converges. Taking suprema in the definition of the norm therefore gives
\begin{align*}
    \|v\|_{X_0}\leq\liminf_{j\to\infty}\|v_j\|_{X_0}
        \leq\tau_0^{1+\alpha}.
\end{align*}
Thus, $(v,\zeta)\in\mathcal B_m$, which proves compactness.

To prove continuity of $\mathcal T_m$, consider a sequence as above. The explicit formulas defining the initial surfaces and the identifications $F_\zeta$ in Lemma~\ref{lem:s4-identifications} give
$D_{\zeta_j}\circ F_{\zeta_j}\to D_\zeta\circ F_\zeta$ in $C^{0,\beta'}(M_0)$, while \eqref{eq:s4-H-error} bounds the $Y_{\zeta_j}$-norms of $D_{\zeta_j}$ uniformly in $j$. The continuity assertion of Proposition~\ref{prop:s4-scalar-input} consequently gives
\begin{align}
    \label{eq:s4-initial-correction-continuity}
    u_{\zeta_j}\circ F_{\zeta_j}\to u_\zeta\circ F_\zeta
        \quad\mbox{in }C^{2,\beta'}(M_0),\qquad
    c_{H,\zeta_j}\to c_{H,\zeta}.
\end{align}
Since $\phi(v_j,\zeta_j)\circ F_{\zeta_j}=u_{\zeta_j}\circ F_{\zeta_j}+v_j$, we have
\begin{align*}
    \phi(v_j,\zeta_j)\circ F_{\zeta_j}\rightarrow \phi(v,\zeta)\circ F_\zeta \qquad\mbox{ in }C^{2,\beta'}(M_0).
\end{align*}
Recall that $Q_\zeta(\phi)=H_{\zeta,\phi}-H_\zeta-\mathcal L_\zeta\phi$.
The convergence of $\phi(v_j,\zeta_j)\circ F_{\zeta_j}$ in $C^{2,\beta'}(M_0)$, together with the smooth dependence of the initial geometry on $\zeta$, therefore gives
\begin{align*}
    Q_{\zeta_j}(\phi(v_j,\zeta_j))\circ F_{\zeta_j}\rightarrow Q_\zeta(\phi(v,\zeta))\circ F_\zeta \qquad\mbox{ in }C^{0,\beta'}(M_0).
\end{align*}
Combining this with \eqref{eq:s4-transverse-continuity}, we obtain
\begin{align}
    \label{eq:s4-coupled-source-continuity}
    E_{\zeta_j}(v_j)\circ F_{\zeta_j}\to E_\zeta(v)\circ F_\zeta\qquad\mbox{ in }C^{0,\beta'}(M_0).
\end{align}
Since by \eqref{eq:s4-coupled-source-bound} we have
\begin{align*}
    \|E_{\zeta_j}(v_j)\|_{Y_{\zeta_j}}
        \leq C\tau_{\zeta_j}^2
        \leq Ce^2\tau_0^2,
\end{align*}
we may apply the continuity of $\mathscr R_\zeta$ to these sources to get
\begin{align*}
    w_{v_j,\zeta_j}\circ F_{\zeta_j}&\to w_{v,\zeta}\circ F_\zeta\quad\mbox{in }C^{2,\beta'}(M_0),\\
    d_{v_j,\zeta_j}&\to d_{v,\zeta}.
\end{align*}
Together with \eqref{eq:s4-initial-correction-continuity} and the positivity and continuity of $\tau_\zeta$, these yield precisely the convergences of the two coordinates of $\mathcal T_m$.
\end{proof}
\subsection{Existence, regularity, and bounds for the deformations}
We can now choose the final correction to minimality by a fixed point argument. We also record stronger estimates for this fixed point than those defining $\mathcal B_m$, since they will be needed for the Gaussian curvature comparison in the next section.
\begin{proposition}
    \label{prop:s4-coupled-existence}
    There exists $m_0\in\mathbb N$ such that, for every integer
    $m\geq m_0$, there is a pair
    $(v_m,\zeta_m)\in\mathcal B_m$ with the following properties.
    Define
    \begin{align*}
        \tau_m&:=\tau_{\zeta_m},\qquad
        \phi_m:=u_{\zeta_m}+v_m\circ F_{\zeta_m}^{-1},
        \qquad b_m:=b(v_m,\zeta_m).
    \end{align*}
    Let the displacement $G_m:M_{\zeta_m}\rightarrow\mathcal N$ within $\mathcal N$ be given by
    \begin{align*}
        G_m(x)&:=\exp_{\jmath_{\zeta_m}(x)}^{\mathcal N}\bigl(\phi_m(x)\nu_{\zeta_m}(x)\bigr) \qquad\forall\,x\in M_{\zeta_m},
    \end{align*}
    where $\jmath_{\zeta_m}:M_{\zeta_m}\hookrightarrow\mathcal N$ is the inclusion of $M_{\zeta_m}$ in $\mathcal N$. Define the unit normal to $\mathcal N$ along $G_m$ by
    \begin{align*}
        n_m(x):=N\bigl(\pi(G_m(x))\bigr)
        \qquad\forall\,x\in M_{\zeta_m}.
    \end{align*}
    The map $\widehat X_m:M_{\zeta_m}\rightarrow\mathbb S^4$ describing the combined displacement in $\mathbb S^4$, given by 
    \begin{align}
        \label{eq:s4-final-embedding}
        \begin{split}
            \widehat X_m(x):=\exp_{G_m(x)}^{\mathbb S^4}
            \bigl(b_m(x)n_m(x)\bigr)=\cos\bigl(b_m(x)\bigr)G_m(x)
            +\sin\bigl(b_m(x)\bigr)n_m(x)
            \quad\forall\,x\in M_{\zeta_m},
        \end{split}
    \end{align}
    is a smooth embedding. Its image $\Sigma_m:=\widehat X_m(M_{\zeta_m})$ is a closed, connected, $\Gamma_m$-invariant minimal surface in $\mathbb S^4$, with
    \begin{align}
        \label{eq:s4-final-genus}
        \operatorname{genus}(\Sigma_m)=m^2+1.
    \end{align}
    Moreover, there exists $C>0$ independent of $m$ such that
    \begin{align}
        |\zeta_m|&\leq C\delta^{-2-\beta}\tau_m^{\alpha/3},
        \label{eq:s4-final-balance}\\
        \|v_m\|_{X_0}&\leq C\delta^{-4-2\beta}\tau_m^2,
        \label{eq:s4-final-v}\\
        \|\phi_m\|_{X_{\zeta_m}}
        &\leq C\delta^{-4-2\beta}\tau_m^{1+\alpha/3},
        \label{eq:s4-final-phi}\\
        \|b_m\|_{\mathcal X_{\zeta_m}}&\leq C\tau_m^2.
        \label{eq:s4-final-b}
    \end{align}
    In particular, $\zeta_m\to0$ and
    \begin{align}
        \label{eq:s4-final-tau-asymptotic}
        \frac{\tau_m}{\tau_0}
        =1+O\bigl(\delta^{-2-\beta}\tau_0^{\alpha/3}\bigr)
        \rightarrow1
        \qquad\mbox{as }m\rightarrow\infty.
    \end{align}
    Finally, abbreviating $\rho_m:=\rho_{\zeta_m}$ and $f_m:=f_{\zeta_m}$, for every $x\in M_{\zeta_m}$ and $j\in\{0,1,2\}$ we have
    \begin{align}
        \label{eq:s4-final-two-jets}
        \begin{split}
            |\nabla^j\phi_m(x)|_{g_{\zeta_m}}&\leq C\delta^{-4-2\beta}\tau_m^{1+\alpha/3}f_m(x)\rho_m(x)^{-j},\\
            |\nabla^j b_m(x)|_{g_{\zeta_m}}&\leq C\tau_m^2\rho_m(x)^{-j},
        \end{split}
    \end{align}
    where $\nabla$ is the Levi--Civita connection of $g_{\zeta_m}$.
\end{proposition}
\begin{proof}
By Lemmas~\ref{lem:s4-coupled-self-map} and
\ref{lem:s4-coupled-compactness}, $\mathcal T_m$ is a
continuous self-map of the nonempty compact convex set
$\mathcal B_m$. Schauder's fixed-point theorem therefore
gives $(v_m,\zeta_m)\in\mathcal B_m$ satisfying
\begin{align}
    \label{eq:s4-fixed-point-identities}
    v_m\circ F_{\zeta_m}^{-1}=w_{v_m,\zeta_m},\qquad
    \tau_m\zeta_m=c_{H,\zeta_m}+d_{v_m,\zeta_m}.
\end{align}
Substituting these identities into
\eqref{eq:s4-coupled-residual-identity} gives
$H_{\zeta_m,\phi_m}+\mathcal R_{\zeta_m}(\phi_m)=0$.
Thus, Proposition~\ref{prop:s4-remaining-scalar-equation}
shows that $\widehat X_m$ is a $C^{2,\beta}$ minimal
embedding. The transformation rules for $\phi_m,b_m$
and the exponential maps give its $\Gamma_m$-equivariance.
Since $\widehat X_m$ is an embedding, closedness, connectedness and
the genus follow from Lemma~\ref{lem:s4-initial-geometry}.

Using \eqref{eq:s4-fixed-point-identities} and
Lemma~\ref{lem:s4-coupled-self-map}, we obtain
\begin{align*}
    |\zeta_m|
        &\leq\tau_m^{-1}(|c_{H,\zeta_m}|+|d_{v_m,\zeta_m}|)
        \leq C\delta^{-2-\beta}(\tau_m^{\alpha/3}+\tau_m),\\
    \|v_m\|_{X_0}
        &\leq C_F\|w_{v_m,\zeta_m}\|_{X_{\zeta_m}}
        \leq C\delta^{-4-2\beta}\tau_m^2.
\end{align*}
The preliminary correction satisfies
$\|u_{\zeta_m}\|_{X_{\zeta_m}}
\leq C\delta^{-4-2\beta}\tau_m^{1+\alpha/3}$ by
\eqref{eq:s4-scalar-inverse} and \eqref{eq:s4-H-error}.
Consequently,
\begin{align*}
    \|\phi_m\|_{X_{\zeta_m}}
        &\leq\|u_{\zeta_m}\|_{X_{\zeta_m}}+C_F\|v_m\|_{X_0}
        \leq C\delta^{-4-2\beta}
            (\tau_m^{1+\alpha/3}+\tau_m^2).
\end{align*}
Since $\alpha/3<1$, these inequalities prove
\eqref{eq:s4-final-balance}--\eqref{eq:s4-final-phi};
\eqref{eq:s4-final-b} is the bound from
Proposition~\ref{prop:s4-transverse-elimination}.
Moreover, \eqref{eq:s4-coupled-scale-comparison} gives
$\zeta_m\to0$. Since $\tau_m/\tau_0=e^{\zeta_m}$ and
$|e^{\zeta_m}-1|\leq e|\zeta_m|$,
\eqref{eq:s4-final-tau-asymptotic} follows. The pointwise
and local H\"older estimates in
\eqref{eq:s4-final-two-jets} are the corresponding terms
in the weighted norms.

We next prove smoothness of the image and then of the
parametrisation. For fixed $m$, write the $C^{2,\beta}$
surface locally as a graph $y\mapsto(y,V(y))$ over its
tangent plane in smooth ambient coordinates. Projecting
the minimality equation onto the normal plane and
inverting the projection of the two graph directions gives
\begin{align*}
    g_V^{ij}\partial_{ij}V^\lambda
        =F^\lambda(y,V,DV),\qquad\lambda\in\{1,2\},
\end{align*}
where $F^\lambda$ is smooth and $g_V$ is the induced
metric. The matrix $(g_V^{ij})$ is uniformly positive
definite on a smaller coordinate disk. Interior Schauder
estimates, applied successively to this system and its
derivatives, therefore give $V\in C^\infty$.
Hence $\Sigma_m$ is a smooth submanifold.

On the tubular neighbourhood, define the smooth maps $P$
and $t_\perp$ by
\begin{align*}
    P(\Psi(q,b)):=\mathcal F(q),\qquad
    t_\perp(\Psi(q,b)):=b.
\end{align*}
Since $P\circ\widehat X_m=G_m$ and $G_m$ is an embedding,
$P|_{\Sigma_m}$ is injective and has injective
differential. Compactness of $\Sigma_m$ then makes it a
smooth embedding. In particular, $G_m(M_{\zeta_m})$ is
a smooth surface in $\mathcal N$.

To recover smoothness of $\phi_m$, consider
\begin{align*}
    \mathcal E(x,t)
        :=\exp_{\jmath_{\zeta_m}(x)}^{\mathcal N}
            (t\nu_{\zeta_m}(x)).
\end{align*}
The bound $|A_{\jmath_{\zeta_m}}|\leq C/\rho_{\zeta_m}$
and bounded ambient curvature control its tangential
Jacobi fields. Indeed, after comparison by parallel
transport, integrating the Jacobi equation with
$J'(0)=-S_{\jmath_{\zeta_m}}^{\nu_{\zeta_m}}J(0)$ gives
\begin{align*}
    |J(t)-J(0)|\leq
        C\bigl(|t|/\rho_{\zeta_m}(x)+t^2\bigr)|J(0)|.
\end{align*}
These fields are orthogonal to the normal geodesic
velocity. Thus $d\mathcal E$ is invertible when
$|t|\leq c\rho_{\zeta_m}(x)$ for a fixed sufficiently
small $c>0$. By \eqref{eq:s4-dimensionless-phi},
$t=\phi_m(x)$ lies in this range for large $m$.
In the resulting smooth local coordinates, the surface
$G_m(M_{\zeta_m})$ is the graph of $\phi_m$. Its projection
to $M_{\zeta_m}$ has invertible differential, since its
composition with $G_m$ is the identity. The inverse
function theorem therefore shows that $\phi_m$ is smooth.
It follows that $G_m$ is smooth, and then
\begin{align*}
    \widehat X_m=(P|_{\Sigma_m})^{-1}\circ G_m,\qquad
    b_m=t_\perp\circ\widehat X_m
\end{align*}
are smooth as well.
\end{proof}
\section{Strictly negative Gaussian curvature and the area limit}
\label{sec:s4-negative-curvature}
For sufficiently large $m$, let $\Sigma_m=\widehat X_m(M_{\zeta_m})$ be the minimal surfaces constructed in Proposition~\ref{prop:s4-coupled-existence}. In this last section, we prove the two remaining conclusions of Theorem~\ref{thm:s4-main}, i.e. $K_{\Sigma_m}<0$ for all sufficiently large $m$ and $\operatorname{Area}(\Sigma_m)\to4\pi^2$ as $m\to+\infty$. The estimates on the Gaussian curvature of $\Sigma_m$ will be achieved by means of two comparison arguments. Away from the neck waists, we will compare the height of $G_m(M_{\zeta_m})$ with the rescaled Green height $\tau_m\Phi_m$, use the estimates on $\Phi_m$ to prove that $G_m(M_{\zeta_m})$ is negatively curved in this region, and then control the perturbation caused by the displacement $b_m$ normal to $\mathcal N$. Near the waists, we will compare the rescaled final surfaces with the unit Euclidean catenoid.

For $q\in\mathbb T$, let $r_m(q):=d_{\mathbb T}(q,L_m)$, so that $\rho_m:=\rho_{\zeta_m}=r_m\circ\pi$ on $M_{\zeta_m}$. Fix $R>0$, to be chosen sufficiently large later. We prove that, for $m$ sufficiently large depending on $R$, the Gaussian curvature of $\Sigma_m$ is negative by treating separately the regions $\rho_m\geq 2R\tau_m$, away from the waists, and $\rho_m\leq 2R\tau_m$, near the waists.
\subsection{Height coordinates and comparison with the rescaled Green height}
We first express $G_m(M_{\zeta_m})$ as a height graph over $\mathbb T$ and transfer the deformation estimates to these coordinates. We use the extension of $h_{\zeta_m}$ to $\{r_m>\tau_m\}$ from Appendix~\ref{app:s4-initial-geometry}.

Since $\rho_m\leq\pi/m$, the definition of $f_m$ gives $\sup_{M_{\zeta_m}}f_m\leq Cm^{-3/2}$ for all sufficiently large $m$. The bound \eqref{eq:s4-final-two-jets} and the asymptotics \eqref{eq:s4-coupled-scale-comparison} therefore yield
\begin{align}
    \label{eq:s4-curvature-scalar-smallness}
    \sum_{j=0}^2\rho_m(x)^j
        |\nabla^j\phi_m(x)|_{g_{\zeta_m}}
        \leq C m^{5/2+2\beta}\tau_m^{1+\alpha/3}
        =o(\tau_m)
        \qquad\forall\,x\in M_{\zeta_m}.
\end{align}
\begin{lemma}[Height coordinates away from the neck waists]
    \label{lem:s4-outer-coordinates}
    For $R$ sufficiently large and $m$ sufficiently large depending on $R$, every $x\in M_{\zeta_m}$ with $\rho_m(x)\geq2R\tau_m$ has a neighbourhood whose image under $G_m$ is a height graph over $\mathbb T$. More precisely, after applying $Z$ if $x$ lies on the lower sheet, there are a neighbourhood $U\subset\{r_m>R\tau_m\}\subset\mathbb{T}$ of $\pi(G_m(x))\in\mathbb{T}$ and functions $z_m,\widetilde b_m\in C^{\infty}(U)$ such that the corresponding portions of $G_m(M_{\zeta_m})$ and $\Sigma_m$ are parametrised respectively by
    \begin{align}
        \label{eq:s4-outer-graph-form}
        X_m(q)&:=\mathcal F(q,z_m(q)),\notag\\
        X_m^b(q)&:=\cos(\widetilde b_m(q))X_m(q)+\sin(\widetilde b_m(q))N(q) \qquad\forall\,q\in U.
    \end{align}
    Moreover, there exists $C>0$ independent of $m$, $R$, and $x$, such that
    \begin{align}
        \label{eq:s4-outer-coordinate-estimates}
        \begin{split}
            \sum_{j=0}^2r_m(q)^j|D^j(z_m-h_{\zeta_m})(q)|&\leq C m^{5/2+2\beta}\tau_m^{1+\alpha/3},\\
            \sum_{j=0}^2r_m(q)^j|D^j\widetilde b_m(q)|&\leq C\tau_m^2,
        \end{split}
    \end{align}
    and
    \begin{align}
        \label{eq:s4-outer-height-bounds}
        |z_m(q)|\leq Cm^2\tau_m,\qquad |D^jz_m(q)|\leq C\tau_m r_m(q)^{-j} \,\forall\,j\in\{1,2\}
    \end{align}
    for every $q\in U$, where $D$ is the Levi--Civita connection of $g_{\mathbb T}$ and the tensor norms are computed using this metric.
\end{lemma}
The proof of Lemma~\ref{lem:s4-outer-coordinates} is given in Appendix~\ref{app:s4-outer-coordinates}.

We next compare the initial height function $h_{\zeta_m}$ with the rescaled Green height $\tau_m\Phi_m$.
\begin{lemma}[Comparison with the rescaled Green height]
    \label{lem:s4-glued-Green-comparison}
    For $R$ sufficiently large and $m$ sufficiently large
    depending on $R$, there exists $C>0$ independent of $R$ and $m$ such that
    \begin{align}
        \label{eq:s4-glued-Green-comparison}
        \sum_{j=0}^2r_m(q)^j
            |D^j(h_{\zeta_m}-\tau_m\Phi_m)(q)|
        \leq C\tau_m
            \bigl(|\zeta_m|+m^2\tau_m^{2\alpha}+R^{-2}\bigr)
    \end{align}
    for every $q\in\mathbb T$ with $r_m(q)\geq R\tau_m$.
    Consequently, on each neighbourhood $U$ from
    Lemma~\ref{lem:s4-outer-coordinates}, we have
    \begin{align}
        \label{eq:s4-final-Green-comparison}
        \sum_{j=0}^2r_m(q)^j
            |D^j(z_m-\tau_m\Phi_m)(q)|
        \leq C\tau_m
            \bigl(m^{5/2+2\beta}\tau_m^{\alpha/3}+R^{-2}\bigr)
        \qquad\forall\,q\in U.
    \end{align}
\end{lemma}
\begin{proof}
Using the cutoffs from \eqref{eq:s4-cutoff}, define $\eta_p\in C^\infty(\mathbb T;[0,1])$ by $\eta_p:=\eta_{3\tau_m^\alpha,2\tau_m^\alpha}\circ d_{\mathbb T}(\,\cdot\,,p)$, so that
\begin{align*}
    \eta_p(q)=\begin{cases}
        1 & \mbox{ if }\, d_{\mathbb T}(q,p)\leq2\tau_m^\alpha,\\
        0 & \mbox{ if }\, d_{\mathbb T}(q,p)\geq 3\tau_m^\alpha.
    \end{cases}
\end{align*}
By \eqref{eq:s4-glued-height}, we have
\begin{align}
    \label{eq:s4-height-difference-decomposition}
    h_{\zeta_m}-\tau_m\Phi_m=-\tau_m\zeta_m\widehat v_m+\sum_{p\in L_m}\eta_p\bigl(C_{\tau_m}\big(d_{\mathbb T}(\,\cdot\,,p)\big)-\widetilde\psi_{\zeta_m}\bigr) \qquad\mbox{ on } \,\{r_m\ge R\tau_m\},
\end{align}
with
\begin{align*}
    \widehat v_m=\sum_{p\in L_m}\eta_{2\delta,\delta}\big(d_{\mathbb T}(\,\cdot\,,p)\big)J_0\big(\sqrt2\,d_{\mathbb T}(\,\cdot\,,p)\big), \qquad \widetilde\psi_{\zeta_m}=\tau_m\Phi_m-\tau_m\zeta_m\widehat v_m,
\end{align*}
where $J_0$ is the Bessel function of the first kind of order zero and the catenoidal height $C_{\tau_m}$ is given by
\begin{align*}
    C_{\tau_m}(r)=\tau_m\operatorname{arcosh}(r/\tau_m) \qquad\forall\, r\geq\tau_m.
\end{align*}
We now estimate the first summand on the right-hand side of \eqref{eq:s4-height-difference-decomposition}. By \eqref{eq:s4-substitute}, the summands defining $\widehat v_m$ have disjoint supports contained in $B_{\mathbb T}(p,2\delta)$. Thus, the power series definition \eqref{eq:j0-def} of $J_0$ and the derivative bounds \eqref{eq:der-bounds} give
\begin{align*}
    |D^j\widehat v_m|\leq C\delta^{-j} \qquad\mbox{ on } \,\mathbb{T},\quad\forall\,j\in\{0,1,2\},
\end{align*}
with $C$ independent of $m$. Since $\widehat v_m$ is supported in $\{r_m\leq2\delta\}$, it follows that
\begin{align*}
    \sum_{j=0}^2r_m^j|D^j\widehat v_m|\leq C \qquad\mbox{ on } \,\mathbb{T},
\end{align*}
for some $C>0$ independent of $m$, $R$, and $p$. Hence, 
\begin{align}
\label{eq:b1}
    \sum_{j=0}^2r_m^j|D^j(-\tau_m\zeta_m\widehat v_m)|=\tau_m|\zeta_m|\sum_{j=0}^2r_m^j|D^j\widehat v_m|\leq C\tau_m|\zeta_m| \qquad\mbox{ on }\,\mathbb{T}.
\end{align}
We then estimate the second summand on the right-hand side of \eqref{eq:s4-height-difference-decomposition}. Fix any $p\in L_m$ and note that, by \eqref{eq:s4-physical-expansion} and \eqref{eq:s4-scales}, the local Green expansion takes the form
\begin{align*}
    \Phi_m=\log\bigg(\frac{2d_{\mathbb{T}}(\,\cdot\,,p)}{\tau_m}\bigg)+\zeta_m+\mathcal E_{m,p} \qquad\mbox{ on }\, B_{\mathbb T}(p,2\delta)\smallsetminus\{p\},
\end{align*}
where $\mathcal E_{m,p}$ is the remainder defined in \eqref{eq:s4-physical-expansion}. Substituting the above expansion into the definition of $\widetilde\psi_{\zeta_m}$ gives
\begin{align*}
    \widetilde\psi_{\zeta_m}-\tau_m\log\bigg(\frac{2d_{\mathbb{T}}(\,\cdot\,,p)}{\tau_m}\bigg)=\tau_m\mathcal E_{m,p}+\tau_m\zeta_m(1-\widehat v_m) \qquad\mbox{ on }\, B_{\mathbb T}(p,2\delta)\smallsetminus\{p\}.
\end{align*}
Note that, for every $p\in L_m$ the cutoff $\eta_p$ is supported in the region $\{d_{\mathbb T}(\,\cdot\,,p)\leq3\tau_m^\alpha\}$. Therefore, to estimate the term multiplied by $\eta_p$ on $\{r_m\geq R\tau_m\}$, it suffices to consider points $q\in\mathbb T$ satisfying
\begin{align*}
    R\tau_m\leq d_{\mathbb T}(q,p)\leq3\tau_m^\alpha<\delta.
\end{align*}
At such points, $\eta_{2\delta,\delta}(d_{\mathbb T}(q,p))=1$
and the other summands defining $\widehat v_m$ vanish.
Consequently, the power series \eqref{eq:j0-def} gives
\begin{align*}
    \widehat v_m(q)
    =J_0\bigl(\sqrt2\,d_{\mathbb T}(q,p)\bigr)
    =1-\frac12d_{\mathbb T}(q,p)^2
        +O\bigl(d_{\mathbb T}(q,p)^4\bigr).
\end{align*}
Differentiating this series yields
\begin{align*}
    |D^j(1-\widehat v_m)(q)|
    \leq C d_{\mathbb T}(q,p)^{2-j}
    \qquad\forall\,j\in\{0,1,2\}.
\end{align*}
Moreover, since $\tau_m\leq d_{\mathbb T}(q,p)<1$,
Remark~\ref{rem:s4-waist-asymptotics} gives
\begin{align*}
    |\log d_{\mathbb T}(q,p)|
    \leq|\log\tau_m|\leq Cm^2.
\end{align*}
The remainder estimate in \eqref{eq:s4-physical-expansion}
therefore implies
\begin{align*}
    |D^j\mathcal E_{m,p}(q)|\leq C d_{\mathbb T}(q,p)^{2-j}
        \bigl(m^2+|\log d_{\mathbb T}(q,p)|\bigr)\leq Cm^2d_{\mathbb T}(q,p)^{2-j}\qquad\forall\,j\in\{0,1,2\}.
\end{align*}
Combining these bounds with the preceding identity for
$\widetilde\psi_{\zeta_m}$ and using $|\zeta_m|\leq1$, we obtain
\begin{align}
\label{eq:b2}
    \left|
        D^j\left[
            \widetilde\psi_{\zeta_m}
            -\tau_m\log\bigg(
                \frac{2d_{\mathbb T}(\,\cdot\,,p)}{\tau_m}
            \bigg)
        \right](q)
    \right|
    \leq C\tau_m m^2d_{\mathbb T}(q,p)^{2-j}
    \qquad\forall\,j\in\{0,1,2\},
\end{align}
for every $q\in\mathbb T$ satisfying $R\tau_m\leq d_{\mathbb T}(q,p)\leq3\tau_m^\alpha$, with $C$ independent of $m$, $R$, and $p$. The catenoidal estimate proved in Appendix~\ref{app:s4-initial-geometry}, immediately before \eqref{eq:s4-matching-estimate}, gives
\begin{align*}
    |\partial_r^j(C_{\tau_m}(r)-\tau_m\log(2r/\tau_m))|\leq C\tau_m^3r^{-2-j} \qquad\forall\,r\geq 2\tau_m,\,\forall\,j\in\{0,1,2\}.
\end{align*}
Therefore, the radial Hessian formula gives 
\begin{align}
\label{eq:b3}
    \left|
        D^j\left[
            C_{\tau_m}\bigl(d_{\mathbb T}(\,\cdot\,,p)\bigr)
            -\tau_m\log\bigg(
                \frac{2d_{\mathbb T}(\,\cdot\,,p)}{\tau_m}
            \bigg)
        \right](q)
    \right|
    \leq C\tau_m^3d_{\mathbb T}(q,p)^{-2-j}
    \qquad\forall\,j\in\{0,1,2\},
\end{align}
for every $q\in\mathbb T$ outside the cut locus of $p$ satisfying $2\tau_m\leq d_{\mathbb T}(q,p)$. Thus, by \eqref{eq:b2}, \eqref{eq:b3}, and the triangle inequality, we obtain
\begin{align}
    \label{eq:s4-log-profile-error}
    \sum_{j=0}^2d_{\mathbb{T}}(\,\cdot\,,p)^j
        |D^j(C_{\tau_m}\big(d_{\mathbb{T}}(\,\cdot\,,p)\big)-\widetilde\psi_{\zeta_m})|
    \leq C\tau_m\bigl(m^2d_{\mathbb{T}}(\,\cdot\,,p)^2+\tau_m^2d_{\mathbb{T}}(\,\cdot\,,p)^{-2}\bigr)
\end{align}
on $\{R\tau_m\leq d_{\mathbb T}(\,\cdot\,,p)\leq3\tau_m^\alpha\}$. Note that the derivative bounds \eqref{eq:der-bounds} and the radial Hessian formula give
\begin{align*}
    \sum_{j=0}^2r_m^j|D^j\eta_p|\leq C \qquad\mbox{ on } \mathbb{T},
\end{align*}
for some $C>0$ independent of $m$, $R$, and $p$. The above cutoff bounds and the product rule preserve the estimate \eqref{eq:s4-log-profile-error} after multiplication by $\eta_p$. For every $q\in\operatorname{spt}\eta_p$ with $r_m(q)\geq R\tau_m$, we have
\begin{align*}
    R\tau_m\leq r_m(q)=d_{\mathbb T}(q,p)\leq3\tau_m^\alpha,
\end{align*}
and hence $\tau_m^2d_{\mathbb T}(q,p)^{-2}\leq R^{-2}$. Each cutoff term therefore satisfies
\begin{align}
    \label{eq:b4}
    \sum_{j=0}^2r_m^j|D^j[\eta_p(C_{\tau_m}\big(d_{\mathbb{T}}(\,\cdot\,,p)\big)-\widetilde\psi_{\zeta_m})]|\leq C\tau_m\bigl(m^2\tau_m^{2\alpha}+R^{-2}\bigr) \qquad\mbox{ on }\,\{r_m\geq R\tau_m\}.
\end{align}
As we noted above, the supports of the cutoffs $\eta_p$ are contained in the regions $\{d_{\mathbb T}(\,\cdot\,,p)\leq3\tau_m^\alpha\}$, which are disjoint. Therefore, at each point in $\{r_m\geq R\tau_m\}$ at most one cutoff $\eta_p$ contributes to the sum on the right-hand side of \eqref{eq:s4-height-difference-decomposition}. Adding \eqref{eq:b1} to \eqref{eq:b4} therefore proves \eqref{eq:s4-glued-Green-comparison}.

Finally, \eqref{eq:s4-final-balance} gives $|\zeta_m|\leq Cm^{2+\beta}\tau_m^{\alpha/3}$. Since $m\geq1$ and $\tau_m<1$, we obtain
\begin{align*}
    |\zeta_m|+m^2\tau_m^{2\alpha}
    \leq C m^{5/2+2\beta}\tau_m^{\alpha/3}.
\end{align*}
Adding \eqref{eq:s4-outer-coordinate-estimates} to \eqref{eq:s4-glued-Green-comparison} proves \eqref{eq:s4-final-Green-comparison}.
\end{proof}
\subsection{Negative curvature away from the neck waists}
We first express the curvature of a general height graph on the torus in terms of its defining height function and its first two derivatives. For any open set $U\subset\mathbb{T}$ and $z\in C^2(U)$, define
\begin{align}
    \label{eq:s4-quadratic-curvature}
    \mathcal K_2(z):=\det D^2z+z\Delta_{\mathbb T}z+|Dz|^2\in C^0(U),
\end{align}
where the determinant and trace are taken with respect to
$g_{\mathbb T}$.
\begin{lemma}[Curvature of a height graph over $\mathbb{T}$]
    \label{lem:s4-exact-height-curvature}
    Let $U\subset\mathbb T$ be open, $z\in C^2(U,(-z_0,z_0))$, and $X:=\mathcal F(\,\cdot\,,z)$ on $U$. The Gaussian curvature $K_X$ of $X^*g_{\mathcal N}$ satisfies
    \begin{align}
        \label{eq:s4-exact-height-curvature}
        \begin{split}
            (\cos^2z+|Dz|^2)^2K_X
                ={}&\det D^2z+\sin z\cos z\,\Delta_{\mathbb T}z
                    +|Dz|^2+|Dz|^4\\
                &+2\tan z\bigl((\Delta_{\mathbb T}z)|Dz|^2
                    -D^2z(Dz,Dz)\bigr).
        \end{split}
    \end{align}
    For $R$ sufficiently large and $m$ sufficiently large depending on $R$, if $U\subset\{r_m\geq R\tau_m\}$ and $z$ satisfies the bounds in \eqref{eq:s4-outer-height-bounds}, then
    \begin{align}
        \label{eq:s4-height-quadratic-error}
        |K_X-\mathcal K_2(z)| \leq C(m^4\tau_m^2+R^{-2})\tau_m^2r_m^{-4}
            \qquad\mbox{ on }U,
    \end{align}
    where $C>0$ is independent of $m$ and $R$.
\end{lemma}
The proof of Lemma~\ref{lem:s4-exact-height-curvature} is given in Appendix~\ref{app:s4-exact-curvature}.
\begin{lemma}[Negative curvature after the displacement within $\mathcal N$]
    \label{lem:s4-scalar-negative-curvature}
    There exists $c_*>0$ such that, after choosing $R$ sufficiently large and then $m$ sufficiently large depending on $R$, every graph $X_m:U\to\mathcal N$ from Lemma~\ref{lem:s4-outer-coordinates} satisfies
    \begin{align}
        \label{eq:s4-scalar-negative-curvature}
        K_{X_m}\leq-c_*\tau_m^2r_m^{-4} \qquad\mbox{ on }\,U.
    \end{align}
\end{lemma}
\begin{proof}
Since $2\det D^2\Phi_m=(\Delta_{\mathbb T}\Phi_m)^2-|D^2\Phi_m|^2$
and $\Delta_{\mathbb T}\Phi_m=-2\Phi_m$ away from $L_m$, by \eqref{eq:s4-height-quadratic-margin} we have
\begin{align}
    \label{eq:s4-Green-curvature-margin}
    \begin{split}
        \mathcal K_2(\tau_m\Phi_m)=\tau_m^2\bigl(|D\Phi_m|^2-\tfrac12|D^2\Phi_m|^2\bigr)\leq-c\tau_m^2(m^4+r_m^{-4})\leq-2c_*\tau_m^2r_m^{-4},
    \end{split}
\end{align}
after fixing $c_*>0$ sufficiently small. 
%
Define $e_m:=z_m-\tau_m\Phi_m$ on $U$. Expanding \eqref{eq:s4-quadratic-curvature}, we obtain
\begin{align}
\begin{split}
\label{eq:b5}
    \mathcal K_2(z_m)-\mathcal K_2(\tau_m\Phi_m)
        &=\tau_m(\Delta_{\mathbb T}\Phi_m)\Delta_{\mathbb T}e_m
            -\tau_m\langle D^2\Phi_m,D^2e_m\rangle
            +\det D^2e_m\\
        &\quad+\tau_m\Phi_m\Delta_{\mathbb T}e_m
            +\tau_me_m\Delta_{\mathbb T}\Phi_m
            +e_m\Delta_{\mathbb T}e_m\\
        &\quad+2\tau_m\langle D\Phi_m,De_m\rangle+|De_m|^2.
\end{split}
\end{align}
Since $R\tau_m\leq r_m\leq\pi/m$ on $U$, the Green bounds
\eqref{eq:s4-physical-values} and \eqref{eq:s4-physical-derivatives}, together with Remark~\ref{rem:s4-waist-asymptotics}, give
\begin{align*}
    |\tau_m\Phi_m|\leq Cm^2\tau_m\leq C\tau_mr_m^{-2}, \qquad |D^j(\tau_m\Phi_m)|\leq C\tau_mr_m^{-j}\quad\forall\,j\in\{1,2\}.
\end{align*}
Combining these bounds with \eqref{eq:s4-final-Green-comparison} and $r_m<1$, we find that every term on the right-hand side of \eqref{eq:b5} is bounded by
\begin{align*}
    C\bigl(m^{5/2+2\beta}\tau_m^{\alpha/3}+R^{-2}\bigr)
        \tau_m^2r_m^{-4}.
\end{align*}
Here, we have used $m^{5/2+2\beta}\tau_m^{\alpha/3}+R^{-2}\leq1$, which holds for $R$ sufficiently large and then $m$ sufficiently large, to bound the terms $\det D^2e_m$, $e_m\Delta_{\mathbb T}e_m$, and $|De_m|^2$. Combining this with \eqref{eq:s4-height-quadratic-error} gives
\begin{align*}
    K_{X_m}\leq\bigl[-2c_*+C\bigl(m^{5/2+2\beta}\tau_m^{\alpha/3}+m^4\tau_m^2+R^{-2}\bigr)\bigr]\tau_m^2r_m^{-4}.
\end{align*}
Choosing $R$ sufficiently large so that $CR^{-2}<c_*/2$, and then $m$ sufficiently large depending on $R$ so that the other two error terms have sum less than $c_*/2$, the statement follows.
\end{proof}
%
%
\begin{lemma}[Curvature change after the displacement normal to $\mathcal{N}$]
    \label{lem:s4-transverse-curvature-comparison}
    For $R$ sufficiently large and $m$ sufficiently large
    depending on $R$, the parametrisations $X_m$ and $X_m^b$
    in Lemma~\ref{lem:s4-outer-coordinates} satisfy
    \begin{align}
        \label{eq:s4-transverse-curvature-comparison}
        |K_{X_m^b}-K_{X_m}|\leq C(m^{-2}+\tau_m)\tau_m^2r_m^{-4}\qquad\mbox{ on }\,U,
    \end{align}
    where $C>0$ is independent of $m$ and $R$.
\end{lemma}
\begin{proof}
All the computations are in the flat coordinates $(s,t)$ on $U\subset\mathbb{T}$, with the immersions regarded as maps into $\mathbb R^5$. To compare the curvatures using the Gauss equation in $\mathbb{R}^5$, we first estimate the change in the first two derivatives of the immersions and show that both induced metrics are uniformly equivalent to $g_{\mathbb T}$.

By \eqref{eq:s4-outer-height-bounds} and the definition $X_m=\mathcal F(\,\cdot\,,z_m)$, we have
\begin{align*}
    |DX_m|\leq C,\qquad
    |D^2X_m|\leq C(1+\tau_mr_m^{-2}).
\end{align*}
Moreover, $g_{X_m}=\cos^2z_m\,g_{\mathbb T}+dz_m\otimes dz_m$. Since $|z_m|\to0$ uniformly and $|Dz_m|\leq C/R$ on $U$, we have
\begin{align}\label{eq:metric-bound-1}
    \tfrac12g_{\mathbb T}\leq g_{X_m}\leq2g_{\mathbb T}
\end{align}
for $R$ sufficiently large and then $m$ sufficiently large. As the derivatives of $N$ are bounded, differentiating
$X_m^b=\cos(\widetilde b_m)X_m+\sin(\widetilde b_m)N$
and using \eqref{eq:s4-outer-coordinate-estimates} gives
\begin{align}
    \label{eq:s4-transverse-E12}
    |DX_m^b-DX_m|\leq C\tau_m^2r_m^{-1},\qquad
    |D^2X_m^b-D^2X_m|\leq C\tau_m^2r_m^{-2}.
\end{align}
For the second inequality, the terms are bounded by
\begin{align*}
    C\bigl(\widetilde b_m^2(1+\tau_mr_m^{-2})
        +|\widetilde b_m|+|D\widetilde b_m|
        +|D^2\widetilde b_m|+|D\widetilde b_m|^2\bigr)
        \leq C\tau_m^2r_m^{-2},
\end{align*}
using $\tau_m\leq r_m<1$. By \eqref{eq:s4-transverse-E12} and the bound for $|DX_m|$,
\begin{align*}
    |g_{X_m^b}-g_{X_m}|_{g_{\mathbb T}}\leq C\tau_m^2r_m^{-1}\leq C\tau_m\rightarrow0.
\end{align*}
Combining this with \eqref{eq:metric-bound-1} gives
\begin{align}\label{eq:metric-bound-2}
    \tfrac14g_{\mathbb T}\leq g_{X_m^b}\leq4g_{\mathbb T}
\end{align}
for all sufficiently large $m$.

For $Y\in\{X_m,X_m^b\}$, let $P_Y:=I-DYg_Y^{-1}(DY)^T$ be the orthogonal projection onto the normal space of $Y$ with respect to the Euclidean inner product on $\mathbb R^5$. The Gauss equation in the ambient $\mathbb{R}^5$ gives
\begin{align}
    \label{eq:s4-euclidean-Gauss-comparison}
    K_Y=\frac{\langle P_Y\partial_{ss}Y,P_Y\partial_{tt}Y\rangle-|P_Y\partial_{st}Y|^2}{\det g_Y} \qquad\mbox{ on }\,U.
\end{align}
The bounds \eqref{eq:metric-bound-1}, \eqref{eq:metric-bound-2}, and \eqref{eq:s4-transverse-E12} imply that
\begin{align*}
    |P_{X_m^b}-P_{X_m}|
        +|(\det g_{X_m^b})^{-1}-(\det g_{X_m})^{-1}|
        &\leq C\tau_m^2r_m^{-1},\\
    |P_{X_m^b}\partial_{ij}X_m^b-P_{X_m}\partial_{ij}X_m|
        &\leq C\bigl(\tau_m^2r_m^{-2}
            +\tau_m^2r_m^{-1}(1+\tau_mr_m^{-2})\bigr) \qquad\forall\,i,j\in\{s,t\}.
\end{align*}
Subtracting the products and denominators in the Gauss
equation, and using
$|P_{X_m}\partial_{ij}X_m|\leq C(1+\tau_mr_m^{-2})$, gives
\begin{align*}
    |K_{X_m^b}-K_{X_m}|
        \leq C\bigl((1+\tau_mr_m^{-2})\tau_m^2r_m^{-2}
            +\tau_m^2r_m^{-1}(1+\tau_mr_m^{-2})^2
            +\tau_m^4r_m^{-4}\bigr).
\end{align*}
Dividing by the required curvature scale yields
\begin{align*}
    \frac{|K_{X_m^b}-K_{X_m}|}{\tau_m^2r_m^{-4}}\leq C\bigl(r_m^2+\tau_m+r_m^3+\tau_mr_m+\tau_m^2/r_m+\tau_m^2\bigr)\leq C(m^{-2}+\tau_m),
\end{align*}
because $\tau_m\leq r_m\leq\pi/m<1$.
\end{proof}
\subsection{Curvature near the neck waists and the area limit}
It remains to prove negative Gaussian curvature on the regions of $\Sigma_m$ corresponding to $\{\rho_m\leq2R\tau_m\}$. We use ambient coordinates centred at each neck and rescale them by $\tau_m^{-1}$, so that the initial catenoid has waist radius one. For fixed $R$, the estimates for $\phi_m$ and $b_m$ show that the corresponding portions of $\Sigma_m$ converge in $C^2$ to the compact portion of this catenoid at distance at most $2R$ from its axis. Since this portion has strictly negative Gaussian curvature, the rescaled surfaces also have negative curvature for all sufficiently large $m$. Since rescaling preserves the sign of Gaussian curvature, the corresponding regions of $\Sigma_m$ also have negative curvature.
\begin{lemma}[Negative curvature near the neck waists]
    \label{lem:s4-core-negative-curvature}
    For $R$ sufficiently large and $m$ sufficiently large depending on $R$, we have
    \begin{align}
        \label{eq:s4-core-negative-curvature}
        K_{\Sigma_m}(\widehat X_m(x))\leq-\frac{1}{2(2R)^4\tau_m^2} \qquad\forall\,x\in M_{\zeta_m}\mbox{ with }\rho_m(x)\leq2R\tau_m.
    \end{align}
\end{lemma}
\begin{proof}
Fix $R$ sufficiently large. Since $3R\tau_m<2\tau_m^\alpha$ for all sufficiently large $m$, the region $\rho_m\leq3R\tau_m$ consists of exact catenoidal pieces. For $p\in L_m$, define rescaled tubular coordinates by
\begin{align*}
    \mathcal P_{m,p}(y_1,y_2,v,w)
        :=\Psi(p+\tau_m(y_1,y_2),\tau_mv,\tau_mw)
        \qquad\forall\,(y_1,y_2,v,w)\in(-5R,5R)^4.
\end{align*}
The coordinate vectors of $\Psi$ at $(p,0,0)$ are orthonormal. Since rescaling evaluates the smooth coefficients of $\Psi^*g_{\mathbb S^4}$ at points tending to $(p,0,0)$, where they equal the Euclidean metric coefficients, and each derivative in the rescaled coordinates introduces a factor $\tau_m$, we obtain
\begin{align*}
    \tau_m^{-2}\mathcal P_{m,p}^*g_{\mathbb S^4}\rightarrow g_{\mathrm{Eucl}} \qquad\mbox{in }C^k((-5R,5R)^4)
\end{align*}
for every fixed $k$, uniformly in $p\in L_m$.
For $A\in\{2R,3R\}$ set
$D_A:=[-\operatorname{arcosh}A,\operatorname{arcosh}A]
\times(\mathbb R/2\pi\mathbb Z)$. The initial neck in these
coordinates is the unit catenoid
\begin{align*}
    c(\xi,\theta)&:=(\cosh\xi\cos\theta,
        \cosh\xi\sin\theta,\xi,0),\\
    \mathcal C_{\zeta_m,p}&=\mathcal P_{m,p}\circ c,
    \qquad \rho_m\circ\mathcal C_{\zeta_m,p}=\tau_m\cosh\xi.
\end{align*}
Its rescaled metric converges smoothly to
$\cosh^2\xi(d\xi^2+d\theta^2)$ on $D_{3R}$.
Consequently, the covariant derivative bounds in
\eqref{eq:s4-final-two-jets} control derivatives in
$(\xi,\theta)$. Since
\begin{align*}
    f_m\circ\mathcal C_{\zeta_m,p}
        =\max\{\tau_m^{3/2}\cosh^{3/2}\xi,
            \tau_m^{1-\alpha/2}\cosh^{1/2}\xi\}
        \leq C_R\tau_m^{1-\alpha/2},
\end{align*}
we obtain
\begin{align*}
    \left\|\frac{\phi_m\circ\mathcal C_{\zeta_m,p}}{\tau_m}
        \right\|_{C^2(D_{3R})}
        &\leq C_Rm^{4+2\beta}\tau_m^{1-\alpha/6}\rightarrow0,\\
    \left\|\frac{b_m\circ\mathcal C_{\zeta_m,p}}{\tau_m}
        \right\|_{C^2(D_{3R})}
        &\leq C_R\tau_m\rightarrow0.
\end{align*}
The normal exponential map of the initial neck in the rescaled
metric has uniformly bounded derivatives on this fixed region.
The first bound therefore shows that the first deformation,
expressed through $\mathcal P_{m,p}^{-1}$, converges to $c$
in $C^2$. In the tubular coordinates $\Psi$, the second
deformation changes only the fourth coordinate, adding
$b_m\circ\mathcal C_{\zeta_m,p}/\tau_m$. Hence
\begin{align*}
    \sup_{p\in L_m}
    \|\mathcal P_{m,p}^{-1}\circ\widehat X_m
        \circ\mathcal C_{\zeta_m,p}-c\|_{C^2(D_{3R})}
        \rightarrow0.
\end{align*}
The Euclidean catenoid has Gaussian curvature
$K_c=-\cosh^{-4}\xi$. The Gauss equation, the $C^2$
convergence of these immersions, and the smooth convergence
of the ambient metrics imply
\begin{align*}
    \sup_{p\in L_m}\sup_{D_{2R}}
        \left|\tau_m^2K_{\Sigma_m}
            (\widehat X_m\circ\mathcal C_{\zeta_m,p})
            +\cosh^{-4}\xi\right|\rightarrow0.
\end{align*}
The factor $\tau_m^2$ occurs because the ambient metric was
multiplied by $\tau_m^{-2}$. On $D_{2R}$,
$\cosh^{-4}\xi\geq(2R)^{-4}$, so taking the error smaller
than $1/(2(2R)^4)$ proves the claim.
\end{proof}
\begin{proof}[Proof of Theorem~\ref{thm:s4-main}]
Choose $R$ as in Lemma~\ref{lem:s4-scalar-negative-curvature}
and then take $m$ sufficiently large. If
$\rho_m(x)\geq2R\tau_m$, the height coordinates near $G_m(x)$
give, with $q=\pi(G_m(x))$,
\begin{align*}
    K_{\Sigma_m}(\widehat X_m(x))\leq[-c_*+C(m^{-2}+\tau_m)]\tau_m^2r_m(q)^{-4}\leq-\frac{c_*}{2}\tau_m^2r_m(q)^{-4}<0.
\end{align*}
Here we use reflection by $Z$ on the lower sheet.
If $\rho_m(x)\leq2R\tau_m$, negativity follows from
Lemma~\ref{lem:s4-core-negative-curvature}. Thus, $K_{\Sigma_m}<0$ everywhere. The embeddedness, minimality, and genus of $\Sigma_m$ have already been established in Proposition~\ref{prop:s4-coupled-existence}.

To prove the area limit, we first prove that the area of the initial surfaces $M_{\zeta_m}$ satisfies the claimed asymptotics. Let
\begin{align*}
    \mathcal U_m:=\mathbb T\smallsetminus
        \bigcup_{p\in L_m}B_{\mathbb T}(p,3\tau_m^\alpha).
\end{align*}
The metric on either graph over $\mathcal U_m$ is $\cos^2h_{\zeta_m}\,g_{\mathbb T}+dh_{\zeta_m}\otimes dh_{\zeta_m}$. Its area density relative to $g_{\mathbb T}$ is
\begin{align*}
    J_m:=\cos h_{\zeta_m}
        \sqrt{\cos^2h_{\zeta_m}+|Dh_{\zeta_m}|^2},
    \qquad
    |J_m-1|\leq C(m^4\tau_m^2+\tau_m^{2-2\alpha})\rightarrow0.
\end{align*}
The removed disks have total area $9\pi m^2\tau_m^{2\alpha}\to0$. The metric of $\mathcal N$ is bounded above by the Euclidean product metric in Fermi coordinates. Thus, each catenoidal piece has area at most
\begin{align*}
    2\pi\tau_m^2\int_{-a_m}^{a_m}\cosh^2\xi\,d\xi
        &=2\pi\tau_m^2(a_m+\sinh a_m\cosh a_m)
        \leq C\tau_m^{2\alpha},
    \qquad a_m:=\operatorname{arcosh}(2\tau_m^{\alpha-1}).
\end{align*}
The interpolation annuli have projected area at most
$C\tau_m^{2\alpha}$ and satisfy
$|Dh_{\zeta_m}|\leq C\tau_m^{1-\alpha}$, so their areas
obey the same bound. Consequently,
\begin{align}
    \label{eq:s4-initial-area-limit}
    \operatorname{Area}(M_{\zeta_m})
        =2\int_{\mathcal U_m}J_m\,d\vol_{g_{\mathbb T}}
            +O(m^2\tau_m^{2\alpha})
        \rightarrow2\operatorname{Area}(\mathbb T)=4\pi^2.
\end{align}
Finally, \eqref{eq:s4-curvature-scalar-smallness} and
$\rho_m\geq\tau_m$ give
\begin{align*}
    \frac{|\phi_m|}{\rho_m}+|d\phi_m|_{g_{\zeta_m}}
        \leq Cm^{5/2+2\beta}\tau_m^{\alpha/3}\rightarrow0.
\end{align*}
The uniformly controlled rescaled charts of
Lemma~\ref{lem:s4-initial-geometry} therefore give
\begin{align*}
    |G_m^*g_{\mathcal N}-g_{\zeta_m}|_{g_{\zeta_m}}
        \leq Cm^{5/2+2\beta}\tau_m^{\alpha/3}.
\end{align*}
For the transverse deformation,
\eqref{eq:s4-full-family-geometry} bounds the shape operator
and normal connection form independently of $m$.
Since $|b_m|\leq C\tau_m^2$ and
$|db_m|_{G_m^*g_{\mathcal N}}\leq C\tau_m$,
the exact metric formula \eqref{eq:s4-graph-metric-exact} gives
\begin{align*}
    |\widehat X_m^*g_{\mathbb S^4}-G_m^*g_{\mathcal N}|_{G_m^*g_{\mathcal N}}
        \leq C\tau_m^2.
\end{align*}
The ratio of the final and initial area elements is therefore
$1+O(m^{5/2+2\beta}\tau_m^{\alpha/3}+\tau_m^2)$, uniformly
on $M_{\zeta_m}$. Integrating and using
\eqref{eq:s4-initial-area-limit} proves
$\operatorname{Area}(\Sigma_m)\to4\pi^2$.
\end{proof}
\appendix
\section{Geometry of the suspension band and the normal Jacobi operator}
\label{app:s4-geometry}
\label{app:s4-geometric-identities}
We collect in this appendix the explicit computations concerning the suspension band and the normal linearisation used in Section~\ref{sec:s4-geometry}.

\smallskip
\noindent\textit{Proof of Lemma~\ref{lem:s4-suspension-geometry}.} Fix $(s,t,z)\in\mathbb T\times(-z_0,z_0)$ and set
$q:=\mathcal F(s,t,z)$. Differentiation gives
$E_1(q)=f_s(s,t)$, $E_2(q)=f_t(s,t)$, and
$E_3(q)=-\sin z\,f(s,t)+\cos z\,e_5$, an orthonormal
triple. Inner products in this computation are Euclidean.
Since $n$ is tangent to the sphere and perpendicular
to $T\mathcal N$,
$\nabla^{\mathbb S^4}_v n=dn_q(v)+\langle v,n(q)\rangle q
=dn_q(v)$ for $v\in T_q\mathcal N$. Thus, for example,
\begin{align*}
    S_q(E_1(q))=-(\nabla^{\mathbb S^4}_{E_1}n)(q)
       =-\frac1{\cos z}\partial_s(n\circ\mathcal F)(s,t,z)
        =-\frac1{\cos z}\partial_sN(s,t)
        =-\sec z\,E_1(q).
\end{align*}
Similarly, we prove the remaining identities in \eqref{eq:s4-S}. Since $\operatorname{tr}S_q=0$, the hypersurface
$\mathcal N$ is minimal in $\mathbb S^4$.

For tangent fields $X,Y$ on $\mathcal N$, projection
of the Euclidean derivative gives
\begin{align*}
    (\nabla^{\mathcal N}_XY)(q)
       =dY_q(X(q))+\langle X(q),Y(q)\rangle q
                   -\langle S_q(X(q)),Y(q)\rangle n(q).
\end{align*}
Thus, for example, as $\partial_{ss}f=-f-N$ we have
\begin{align*}
    (\nabla^{\mathcal N}_{E_1}E_1)(q)=\sec z\,\partial_{ss}f(s,t)+q+\sec z\,N(s,t) =q-\sec z\,f(s,t)=\tan z\,E_3(q).
\end{align*}
Analogously, for $i,j\in\{1,2\}$ we get
\begin{align*}
    (\nabla^{\mathcal N}_{E_i}E_j)(q)=\tan z\,\delta_{ij}E_3(q),\qquad(\nabla^{\mathcal N}_{E_i}E_3)(q)=-\tan z\,E_i(q),\qquad(\nabla^{\mathcal N}_{E_3}E_i)(q)=(\nabla^{\mathcal N}_{E_3}E_3)(q)=0.
\end{align*}
At $z=0$ the normal components in the first identity
vanish, so the second fundamental form of $f$ in
$\mathcal N$ is zero, i.e. $f(\mathbb{T})$ is totally geodesic within $\mathcal{N}$. Since $S_q(E_3(q))=0$, the Gauss equation gives sectional curvature one on each plane
$\operatorname{span}\{E_i(q),E_3(q)\}$ for $i\in\{1,2\}$, and hence $\operatorname{Ric}_{g_{\mathcal N}}(E_3(q),E_3(q))=2$.
The scalar Jacobi formula therefore gives $L_{\mathbb T}=\Delta_{\mathbb T}+2$.

Finally set $\lambda_1=-\sec z$, $\lambda_2=\sec z$.
Using $(\nabla_XS)(Y)=\nabla_X(S(Y))-S(\nabla_XY)$,
the only possibly nonzero terms are
\begin{align*}
    (\nabla_{E_3}^{\mathcal N}S)_q(E_i(q))&=\lambda_i\tan z\,E_i(q),\\
    (\nabla_{E_i}^{\mathcal N}S)_q(E_j(q))&=\lambda_j\tan z\,\delta_{ij}E_3(q),\\
    (\nabla_{E_i}^{\mathcal N}S)_q(E_3(q))&=\lambda_i\tan z\,E_i(q)
       \qquad\forall\,i,j\in\{1,2\}.
\end{align*}
Thus $|S_q|^2=2\sec^2z$ and
$|(\nabla^{\mathcal N}S)_q|^2=6\sec^2z\tan^2z$.
The inequalities $\sec z\leq\sec z_0$ and
$|\tan z|\leq\sec^2z_0|z|$ prove the asserted bounds.

\smallskip
\noindent\textit{Proof of Lemma~\ref{lem:s4-normal-identities}.} For every $X,Y,\in\Gamma(TM)$, differentiating $\langle Y,n\rangle=0$ gives
$h_\iota^n(X,Y)=\langle S(X),Y\rangle$.
Taking the trace over $TM\oplus\mathbb R\nu$ and
using $\operatorname{tr}S=0$ gives
$H_{\iota,n}=-\langle S(\nu),\nu\rangle$.
Differentiating $\langle\nu,n\rangle=0$ similarly gives
$\omega(X)=\langle S(X),\nu\rangle$.

The difference between the spherical and intrinsic
connections on $\mathcal N$ is a multiple of $n$,
so $h_\iota^\nu=h_\jmath^\nu$ and
$S_\iota^\nu=S_\jmath^\nu$. Consequently
\begin{align*}
    \nabla^{\mathcal N}_XY&=\nabla^M_XY+h_\iota^\nu(X,Y)\nu,
    \qquad \nabla^{\mathcal N}_X\nu=-S_\iota^\nu(X)
    \quad\forall\,X,Y\in\Gamma(TM).
\end{align*}
Substitution into
$(\nabla_X^M\omega)(Y)=X\langle S(Y),\nu\rangle
-\langle S(\nabla_X^MY),\nu\rangle$ gives
\eqref{eq:s4-omega-derivative}.
Since $S$ is self-adjoint and $S(E_3)=0$, we have
$S(\nu)=S(\nu_h)$ and
$\langle S(\nu),\nu\rangle=\langle S(\nu_h),\nu_h\rangle$.
Thus the first three bounds follow from $|S|\leq C$,
and the differentiated identity gives
$|\nabla^M\omega|\leq|\nabla^{\mathcal N}S|
+2|S||S_\iota^\nu|$. Lemma~\ref{lem:s4-suspension-geometry}
proves the last bound.

\smallskip
\noindent\textit{Proof of Lemma~\ref{lem:s4-normal-linearisation}.} For $V=u\nu+bn$, by \eqref{eq:s4-normal-connection-frame} we have
\begin{align*}
    \nabla_X^\perp V
       =(Xu-b\omega(X))\nu+(Xb+u\omega(X))n.
\end{align*}
About any point $p\in M$, choose a local tangent orthonormal frame $(e_1,e_2)$ whose covariant derivatives vanish at $p$. Applying the above identity twice and using the definition \eqref{eq:s4-delta-perp} of $\Delta^{\perp}$ yields
\begin{align*}
    \Delta^\perp V={}&
       (\Delta_Mu-2\langle\omega,db\rangle
          -b\operatorname{div}_M\omega-u|\omega|^2)\nu\\
       &+(\Delta_Mb+2\langle\omega,du\rangle
          +u\operatorname{div}_M\omega-b|\omega|^2)n.
\end{align*}
Since $A_\iota=h_\iota^\nu\nu+h_\iota^n n$, the
second-fundamental-form term in \eqref{eq:s4-normal-variation-formula} is represented in the frame $(\nu,n)$ by a matrix with diagonal entries $|h_\iota^\nu|^2=|S_\iota^\nu|^2$ and
$|h_\iota^n|^2=|S_\iota^n|^2$, and off-diagonal entries $\langle h_\iota^\nu,h_\iota^n\rangle=\langle S_\iota^\nu,S_\iota^n\rangle$. This gives $B,B^*,q$ as stated. For the first diagonal entry, note that the Gauss equation and the preceding identities give
\begin{align*}
    \operatorname{Ric}_{g_{\mathcal N}}(\nu,\nu)
       &=2-|S(\nu)|^2,\qquad
    |S(\nu)|^2=|\omega|^2+H_{\iota,n}^2.
\end{align*}
Together with $S_\iota^\nu=S_\jmath^\nu$, these prove
\eqref{eq:s4-J-matrix}. On a closed $M$,
integration by parts gives
$\int_Mu\langle\omega,db\rangle
=-\int_Mb(\langle\omega,du\rangle+u\operatorname{div}_M\omega)$.
Substituting into the definition of $B$ proves
$\int_MuBb=\int_MbB^*u$, with all integrals taken
against $d\operatorname{vol}_M$.

\section{Estimates for the initial surfaces}
\label{app:s4-initial}
In this appendix, we derive the required estimates on the Green function and we discuss the geometric and topological properties of the initial surfaces. Moreover, we verify the hypotheses of the scalar linear construction. All constants have the dependence specified in Section~\ref{sec:s4-scalar-family}.

\subsection{The Green function and its derivatives}
\label{app:s4-green}
In this section, we prove Lemma~\ref{lem:s4-Green}. Since $\widehat{G_a-G_0}(0)=0$ and
\begin{align*}
    \lvert\widehat{G_a-G_0}(n)\rvert=\bigg\lvert\frac{a}{|n|^2(|n|^2-a)}\bigg\rvert\le\frac{4a}{3\lvert n\rvert^4} \qquad\forall\,n\in\mathbb Z^2\smallsetminus\{0\},
\end{align*}
we get $\|G_a-G_0\|_{C^0}\leq Ca$. Since $G_0\in L^2(\mathbb T_0)$, this also gives a uniform $L^2$ bound for $G_a$ as $a$ ranges over $[0,1/4]$. Interior elliptic estimates for the equation $(\Delta_0+a)G_a=-1$ on $\mathbb T_0\smallsetminus\{0\}$ give uniform bounds for every derivative on fixed compact subsets away from $0\in\mathbb T_0$.

To separate the singularity of $G_a$ from a smooth
remainder, we construct a radial fundamental solution $P_a$
for the elliptic operator $\Delta_0+a$. Recall that $J_0$ denotes the Bessel function of the first kind of order zero. Set
$h_j:=\sum_{i=1}^j i^{-1}$ for every integer $j\geq1$
and define, for $a\geq0$ and $r>0$,
\begin{align}
    \label{eq:s4-radial-fundamental}
    P_a(r):=2\pi\left[J_0(\sqrt a\,r)\log r+
       \sum_{j\geq1}\frac{(-1)^{j+1}h_j}{(j!)^2}
                       \left(\frac{ar^2}{4}\right)^j\right].
\end{align}
The series converge with all derivatives locally uniformly
for $a\geq0$ and $r>0$. Termwise differentiation gives
$P_a''+r^{-1}P_a'+aP_a=0$ on $(0,\infty)$.
Moreover, as $r\downarrow0$, uniformly for $a\in[0,1/4]$,
\begin{align*}
    P_a(r)&=2\pi\log r+O(ar^2(1+|\log r|)),\\
    P_a'(r)&=2\pi/r+O(ar(1+|\log r|)).
\end{align*}
Writing $r(q):=d_{\mathbb T_0}(q,0)$ on
$B_{\mathbb T_0}(0,1)$, integration by parts on
punctured disks therefore gives
\begin{align*}
    (\Delta_0+a)(P_a\circ r)=4\pi^2\delta_0
    \qquad\text{distributionally on }B_{\mathbb T_0}(0,1).
\end{align*}
Indeed, $2\pi rP_a'(r)\to4\pi^2$ and $rP_a(r)\to0$,
which determine the point mass and eliminate the
boundary term involving the derivative of the test function.

Thus $Q_a:=G_a-P_a\circ r$ satisfies
$(\Delta_0+a)Q_a=-1$ distributionally on
$B_{\mathbb T_0}(0,1)$. The uniform $L^2$ bounds for
$G_a$ and $P_a\circ r$, together with interior elliptic
estimates, give a smooth extension of $Q_a$ across the
origin, with uniform bounds for all derivatives on
$B_{\mathbb T_0}(0,1/2)$.
Set $c_a:=Q_a(0)$. Since $P_a(r)-P_0(r)\to0$ as
$r\downarrow0$, evaluating the absolutely convergent
Fourier series of $G_a-G_0$ at the origin gives
\begin{align*}
    c_a-c_0=-a\sum_{n\in\mathbb Z^2\smallsetminus\{0\}}
        \frac1{|n|^2(|n|^2-a)}=O(a).
\end{align*}
The invariance of $Q_a$ under both coordinate reflections
gives $D_0Q_a(0)=0$. Taylor's formula and the series
defining $P_a$ therefore yield, for $0<R<r_*/2$,
\begin{align*}
    \|Q_a-c_a\|_{C^{k,\beta}_R(\mathcal A_R)}
    &\leq C_{k,\beta}R^2,\\
    \|P_a\circ r-2\pi\log r\|_{C^{k,\beta}_R(\mathcal A_R)}
    &\leq C_{k,\beta}aR^2(1+|\log R|).
\end{align*}
Setting $E_a:=G_a-2\pi\log r-c_a$ on $0<r<r_*$
proves the local expansion and remainder estimate.
Combining these with the previously established
estimates away from the origin gives the global bounds.

Finally, on $B_{\mathbb T_0}(0,r_*/2)\smallsetminus\{0\}$,
the Hessian of $2\pi\log r$ has eigenvalues
$-2\pi/r^2$ and $2\pi/r^2$. For $q\in B_{\mathbb T_0}(0,r_*/2)\smallsetminus\{0\}$,
we have $q\in\mathcal A_{r(q)}$. The local expansion
and \eqref{eq:s4-Green-remainder}, applied with
$k=2$, $\beta=1/2$ and $R=r(q)$, therefore give
\begin{align*}
    |D_0^2G_a(q)|
    &=\bigl|D_0^2(2\pi\log r)(q)+D_0^2E_a(q)\bigr|\\
    &\geq |D_0^2(2\pi\log r)(q)|-|D_0^2E_a(q)|\\
    &\geq \frac{2\sqrt2\pi}{r(q)^2}
        -\frac{\|E_a\|_{C^{2,1/2}_{r(q)}(\mathcal A_{r(q)})}}{r(q)^2}\\
    &\geq \frac{2\sqrt2\pi}{r(q)^2}
        -C(1+a|\log r(q)|),
\end{align*}
uniformly for $a\in[0,1/4]$.
Consequently, there exist $r_0\in(0,r_*/2)$ and $c>0$, independent of $a$, such that $|D_0^2G_a|^2\geq cr^{-4}$ on $B_{\mathbb T_0}(0,r_0)\smallsetminus\{0\}$ for every $a\in[0,1/4]$. Outside $B_{\mathbb T_0}(0,r_0)$, the uniform bound
for $G_a$ allows us to choose $a_*\in(0,1/4]$ such
that $|aG_a|\leq1/2$ for every $a\in[0,a_*]$.
Consequently,
\begin{align*}
    |D_0^2G_a|^2
    \geq\frac12|\Delta_0G_a|^2
    =\frac12|1+aG_a|^2
    \geq\frac18.
\end{align*}
Combining the two regions proves the first Hessian
bound with some constant $c>0$. Finally, the global
gradient estimate gives
$|D_0G_a|^2\leq C(1+r^{-4})$.
Decreasing $a_*$ so that $2a_*C\leq c$, we obtain
\begin{align*}
    2a|D_0G_a|^2-2|D_0^2G_a|^2
    \leq-c(1+r^{-4}),
\end{align*}
which proves the second bound.
\subsection{Geometry of the initial surfaces}
\label{app:s4-initial-geometry}
Here, we prove Lemma~\ref{lem:s4-initial-geometry}. Fix $\zeta\in[-1,1]$ and recall that $r_c=\tau_\zeta^\alpha$. By \eqref{eq:s4-physical-values} and $\log\tau_\zeta=\log2-A_m+\zeta$, for every $q\in\mathbb{T}$ such that $\rho(q)\geq r_c$ we have $\Phi_m(q)\geq(1-\alpha)m^2/(2\pi)+(1-\alpha)\log(\sqrt2m)-C$. Together with the trivial upper bound $\rho\leq\pi/m$, for all sufficiently large $m$ this gives
\begin{align}
    \label{eq:s4-truncated-positivity}
    \frac{1-\alpha}{4\pi}m^2\leq\Phi_m(q)\leq Cm^2 \qquad\forall\,q\in\mathbb{T}\, \mbox{ s.t. } \rho(q)\geq r_c.
\end{align}
The summands defining $\widehat v_m$ are uniformly bounded and have disjoint supports, so $\|\widehat v_m\|_{C^0}\leq C$. Hence,
\begin{align*}
    0<\widetilde\psi_\zeta(q)\leq Cm^2\tau_\zeta \qquad\forall\,q\in\mathbb{T}\, \mbox{ s.t. } \rho(q)\geq r_c.
\end{align*}
As
\begin{align*}
    0<C_{\tau_\zeta}(r)\leq\tau_\zeta\log(2r/\tau_\zeta)\leq Cm^2\tau_\zeta
    \qquad\forall\,\tau_\zeta<r\leq3r_c,
\end{align*}
the convex interpolation gives $0<h_\zeta\leq Cm^2\tau_\zeta$ on $\Omega_\zeta$, and all pieces lie in the region $|z|<z_0$ for large $m$ by Remark \ref{rem:s4-waist-asymptotics}.

The catenoidal parametrisations are embeddings by Section~\ref{sec:s4-geometry}. Positivity, disjoint neck disks and exact agreement on overlaps show that their union with the graphs is smooth and embedded. Extend $h_\zeta$ to $\{q\in\mathbb T \mbox{ : } \rho(q)>\tau_\zeta\}$ by setting $h_\zeta(q):=C_{\tau_\zeta}(r_p(q))$ for every $p\in L_m$ and every $q\in\mathbb T$ satisfying $\tau_\zeta<r_p(q)<9\tau_\zeta$, and define the open set
\begin{align*}
    \mathcal O_\zeta
    :=\{\mathcal F(q,z) \mbox{ : }\rho(q)>\tau_\zeta,\
        |z|<h_\zeta(q)\}\subset\mathcal{N}.
\end{align*}
Note that $\mathcal O_\zeta$ has compact closure in $\mathcal{N}$. Near the waist corresponding to $p\in L_m$, the inequality $|z|<h_\zeta(q)$ is equivalent to $r_p(q)>\tau_\zeta\cosh(z/\tau_\zeta)$. Hence, the boundary of $\mathcal O_{\zeta}$ there is given by $\{\mathcal F(q,z) \mbox{ : } r_p(q)=\tau_\zeta\cosh(z/\tau_\zeta)\}$, which is
precisely the image of $\mathcal C_{\zeta,p}$ near
$\xi=0$. Away from the waists, the boundary is given
by $\{\mathcal F(q,z) \mbox{ : } z=\pm h_\zeta(q)\}$. Thus, $\partial\mathcal O_\zeta=M_\zeta$, and we denote its outward unit normal within $\mathcal N$ by $\nu_\zeta$. The symmetries of $h_\zeta$ and the condition $|z|<h_\zeta(q)$ make $\mathcal O_\zeta$ and its boundary $\Gamma_m$-invariant. Joining two tori with $m^2$ disks removed from each by $m^2$ annuli gives a closed connected surface with Euler characteristic $-2m^2$ and genus $m^2+1$. Since $\rho(q)\geq9\tau_\zeta$ for every $q\in\Omega_\zeta$ and $\rho_\zeta(\mathcal C_{\zeta,p}(\xi,\theta))=\tau_\zeta\cosh\xi$, the lattice covering radius gives $\tau_\zeta\leq\rho_\zeta\leq\pi/m$.

For $p\in L_m$ and for every $q\in\mathbb{T}$ such that $0<r_p(q)\leq\delta$, \eqref{eq:s4-physical-expansion} and \eqref{eq:s4-exact-mismatch} give
\begin{align*}
    \widetilde\psi_\zeta(q)-\tau_\zeta\log(2r_p(q)/\tau_\zeta)=\tau_\zeta\mathcal E_{m,p}(q)+\tau_\zeta\zeta\bigl(1-J_0(\sqrt2r_p(q))\bigr).
\end{align*}
By differentiating the explicit expression $C_{\tau_\zeta}$ for the catenoidal height, for every $r\geq2\tau_\zeta$ and every integer $j\geq0$ we get $|\partial_r^j(C_{\tau_\zeta}(r)-\tau_\zeta\log(2r/\tau_\zeta))|\leq C_j\tau_\zeta^3r^{-2-j}$. Using \eqref{eq:s4-physical-expansion}, the Taylor expansion $J_0(\sqrt2r)=1-\frac12r^2+O(r^4)$ and its differentiated series, together with $|\log\tau_\zeta|\asymp m^2$, for every integer $k\geq0$ we obtain
\begin{align}
    \label{eq:s4-matching-estimate}
    \|\widetilde\psi_\zeta-C_{\tau_\zeta}\circ r_p\|_{C^{k,\beta}_{r_c}(\{2r_c\leq r_p\leq3r_c\})}\leq C_{k,\beta}\bigl(\tau_\zeta r_c^2|\log\tau_\zeta|+\tau_\zeta^3r_c^{-2}\bigr).
\end{align}
Since $r_c=\tau_\zeta^\alpha$ and $0<\alpha<1$, Remark~\ref{rem:s4-waist-asymptotics} gives $r_c^2|\log\tau_\zeta|+\tau_\zeta^2r_c^{-2}\to0$ as $m\to\infty$, uniformly in $\zeta$. The scaled cutoff bounds and \eqref{eq:s4-physical-derivatives} then yield
\begin{align*}
    |D^jh_\zeta(q)|\leq C_j\tau_\zeta\rho(q)^{-j}
    \qquad\forall\,q\in\mathbb{T}\, \mbox{ s.t. } \rho(q)\geq 2\tau_{\zeta},\quad\forall\,j\geq1.
\end{align*}
We now prove the bounds for the horizontal component of $\nu_\zeta$ and the shape operator
$S_{\jmath_\zeta}^{\nu_\zeta}$ in \eqref{eq:s4-initial-geometry-bounds}. In order to do so, we will use the height graph coordinates on the region where $\rho_\zeta\geq2\tau_\zeta$ and the catenoidal parametrisations near the waists. On a graph $F_w(q)=\mathcal F(q,w(q))$, with $w=\pm h_\zeta$, let $\nu^+$ be the upward unit normal. The induced metric on the graph is $F_w^*g_{\mathcal N}=\cos^2w\,g_{\mathbb T}+dw\otimes dw$, and its scalar second fundamental form in the coordinates $(s,t)$ has components
\begin{align*}
    h^{\nu^+}_{ij}=\frac{\cos w}{\sqrt{\cos^2w+|Dw|^2}}\bigl(w_{ij}+\sin w\cos w\,\delta_{ij}+2\tan w\,w_iw_j\bigr),\qquad i,j\in\{1,2\}.
\end{align*}
Where $\rho\geq2\tau_\zeta$, the metric $F_w^*g_{\mathcal N}$ is uniformly comparable to $g_{\mathbb T}$. Since $\nu_\zeta=\pm\nu^+$, these formulas and \eqref{eq:s4-height-graph-identities} give $|(\nu_\zeta)_h|\leq C\tau_\zeta/\rho_\zeta$ and $|S_{\jmath_\zeta}^{\nu_\zeta}|_{g_\zeta} \leq C\tau_\zeta(\rho_\zeta^{-2}+m^2)\leq C\tau_\zeta\rho_\zeta^{-2}$, using $\rho_\zeta\leq\pi/m$. For $x=\mathcal C_{\zeta,p}(\xi,\theta)$ with $\rho_\zeta(x)\leq2\tau_\zeta$, set $c:=\cos(\tau_\zeta\xi)$ and
$E_r:=\cos\theta E_1(x)+\sin\theta E_2(x)$. Then,
\begin{align*}
    \nu_\zeta(x)
    =\frac{-E_r+c\sinh\xi E_3(x)}
        {\sqrt{1+c^2\sinh^2\xi}},
    \qquad
    |(\nu_\zeta)_h(x)|
        \leq\sec z_0\,\frac{\tau_\zeta}{\rho_\zeta(x)}.
\end{align*}
Since $\rho_\zeta(x)=\tau_\zeta\cosh\xi\leq2\tau_\zeta$, we have $|\xi|\leq\operatorname{arcosh}2$. The preceding formula gives the required horizontal-normal bound. Moreover,
\begin{align*}
    |z(x)|=\tau_\zeta|\xi|\leq\tau_\zeta\operatorname{arcosh}2,
\end{align*}
which also gives the required height bound in this region. It remains to estimate the shape operator. Under the coordinate map
\begin{align*}
    (y_1,y_2,y_3)\mapsto\mathcal F\bigl(p+\tau_\zeta(y_1,y_2),\tau_\zeta y_3\bigr),
\end{align*}
the pullback of $\tau_\zeta^{-2}g_{\mathcal N}$ is
\begin{align*}
    dy_3^2+\cos^2(\tau_\zeta y_3)(dy_1^2+dy_2^2).
\end{align*}
Since $\tau_\zeta\to0$, these metrics converge smoothly to the Euclidean metric on every fixed bounded subset of $\mathbb R^3$ as $m\to\infty$. In these coordinates the neck is the unit catenoid. The condition $|\xi|\leq\operatorname{arcosh}2$ confines the points under consideration to a fixed compact portion of this catenoid. Smooth convergence of the ambient metrics therefore gives a uniform bound for its shape operator in the rescaled metric. Scaling back and using $\rho_\zeta(x)\leq2\tau_\zeta$ gives
\begin{align*}
    |S_{\jmath_\zeta}^{\nu_\zeta}(x)|_{g_\zeta}
        \leq\frac{C}{\tau_\zeta}
        \leq\frac{4C\tau_\zeta}{\rho_\zeta(x)^2}.
\end{align*}

To prove \eqref{eq:s4-scaled-injectivity}, first consider
a complete Euclidean catenoid at a point of horizontal
radius $r$. Every loop based there of length less than
$2r$ projects horizontally into the open disk of radius
$r$ centred at the projected base point: each point
can be reached along the loop by an arc of length less
than $r$. This disk avoids the origin, so the loop has
zero winding number and is contractible on the catenoid.
Since the catenoid has nonpositive Gaussian curvature,
the Cartan--Hadamard theorem on its universal cover
therefore gives injectivity radius at least $r$.

Fix a large $B>4$ and set $r:=\rho_\zeta(x)$.
If $r\leq B\tau_\zeta$, the metrics
$\tau_\zeta^{-2}g_\zeta$ converge smoothly to the unit
catenoid metric on the neck region
$\rho_\zeta\leq4B\tau_\zeta$. Every path from $x$ to
the boundary of this region has length at least
\begin{align*}
    \cos z_0(4B\tau_\zeta-r)>r.
\end{align*}
Thus geodesics of length at most $r/2$ remain in this
region. The limiting exponential maps are embeddings
on the closed tangent balls of radius $r/(2\tau_\zeta)$.
Their $C^1$ convergence therefore gives
$\operatorname{inj}_{g_\zeta}(x)\geq r/2$ for large $m$,
uniformly for $r/\tau_\zeta\in[1,B]$.

If $r\geq B\tau_\zeta$, the sheet through $x$ is a graph
over $B_{\mathbb T}(\pi(x),r/2)$. In graph coordinates
rescaled by $r$, the metric differs from the Euclidean
metric in $C^2$ by $O(B^{-2})+o_m(1)$, and the boundary
is at distance at least $\cos z_0/2>1/3$ from $x$.
Let $E_x$ denote the exponential map of
$r^{-2}g_\zeta$, expressed in these coordinates.
The geodesic equation gives $\|DE_x-I\|<1/2$ on its
tangent ball of radius $1/5$, for $B$ and then $m$
sufficiently large. Integrating along line segments yields
\begin{align*}
    |E_x(v)-E_x(w)|\geq\tfrac12|v-w|
\end{align*}
for all $v,w$ in this ball. Thus $E_x$ is injective
and its differential is nonsingular, proving
\eqref{eq:s4-scaled-injectivity}, uniformly in $\zeta$.

Finally, local catenoidal charts in the first region
and graph charts in the second give the asserted
coordinate estimates after rescaling by $r$, using
the smooth catenoidal convergence and the bounds
for all derivatives of $h_\zeta$.

\subsection{Application of the scalar linear construction}
\label{app:s4-scalar-input}
We prove Proposition~\ref{prop:s4-scalar-input} by applying
Proposition~4.18, Corollary~4.21 and Lemmas~4.6 and~5.1 of
\cite{KapouleasMcGrathLD} to $\Gamma_m$-invariant functions.
We first verify the hypotheses of these results.
The ambient three-manifold is $\mathcal N$, the base surface
is the central torus, and its Fermi map is $\mathcal F$.
Convention~2.1 of the cited paper allows an open ambient
manifold and does not require completeness.

Since $L_{\mathbb T}=\Delta_{\mathbb T}+2$, its kernel consists
of the four first Fourier modes. Invariance under the lattice
translations requires both Fourier frequencies to belong to
$m\mathbb Z$. Thus, for $m\geq2$, every invariant source is
orthogonal to the kernel. The inverse on its orthogonal
complement preserves invariance, because $L_{\mathbb T}$
commutes with the symmetries. Its $C^{0,\beta}\to C^{2,\beta}$
bound is independent of $m$, since the operator and torus
are fixed. This verifies Assumption~4.1 on the invariant
subspace used in the construction.

To verify Assumption~3.11, use polar coordinates $(r,\theta)$
centred at $p\in L_m$ and multiply the local solutions
\begin{align*}
    J_0(\sqrt2r),\qquad
    \sqrt2J_1(\sqrt2r)\cos\theta,\qquad
    \sqrt2J_1(\sqrt2r)\sin\theta
\end{align*}
by the cutoff in \eqref{eq:s4-substitute}, where
$J_1(x):=\sum_{j\geq0}(-1)^j(x/2)^{2j+1}/(j!(j+1)!)$.
Their expansions are $1-r^2/2+O(r^4)$, $x_1+O(r^3)$
and $x_2+O(r^3)$, respectively. Consequently, their values
and differentials at $p$ form the standard basis of
$\mathbb R\oplus T_p^*\mathbb T$. The cutoff functions are
supported in $r\leq2\delta$, and their images under
$L_{\mathbb T}$ are supported in $\delta\leq r\leq2\delta$.
Differentiating the cutoff bounds the $C^{2,\beta}$ norm
by $C\delta^{-2-\beta}$ times the largest coefficient.
The supports are separated by a distance comparable to
$\delta$, so the same bound holds for their sum over $L_m$.
The expansions also bound the difference from each prescribed
affine function by $Cr^2$, with the corresponding scaled
derivative bounds. These are the requirements of
Assumption~3.11. On invariant functions, the coordinate
reflections eliminate the linear terms and the lattice
translations identify the constant terms. Hence the two
obstruction spaces reduce to $\mathbb R\widehat v_m$ and
$\mathbb Rk_m$.

For Convention~3.8, the disks of radius $9\delta$ are disjoint
because $18\delta<\ell/m$, and they lie in the fixed Fermi
neighbourhood for large $m$. Moreover, on a disk of radius
$\delta'\leq3\delta$, the first Dirichlet eigenvalue of
$-\Delta$ is $j_{0,1}^2/(\delta')^2>2$, where $j_{0,1}$
is the first positive zero of $J_0$. Thus $\Delta+2$ is
Dirichlet invertible there. The function $P_2/(2\pi)$ in
\eqref{eq:s4-radial-fundamental} provides the required
logarithmic solution.

The scale and Green estimates give, uniformly for $|\zeta|\leq1$,
\begin{align*}
    9\tau_\zeta^\alpha&<\tau_\zeta^{\alpha/100}<\delta,
        \qquad \tau_\zeta\leq\tau_\zeta^{1-\alpha/100},\\
    \delta^{-2}\|\psi_\zeta\|_{C^{2,\beta}(\partial B_{\mathbb T}(p,\delta))}
        &\leq C\tau_\zeta m^{4+\beta}
            =o(\tau_\zeta^{1-\alpha/9}),\\
    \|\psi_\zeta\|_{C^{3,\beta}(\mathbb T\smallsetminus
        \bigcup_{p\in L_m}B_{\mathbb T}(p,\tau_\zeta^\alpha))}
        &\leq C\tau_\zeta^{1-\alpha(3+\beta)}
            =o(\tau_\zeta^{8/9}),\\
    \psi_\zeta(q)&\geq c\tau_\zeta m^2
        \geq\tau_\zeta^{1+\alpha/5}
        \qquad\mbox{if }d_{\mathbb T}(q,L_m)\geq\tau_\zeta^\alpha.
\end{align*}
For the global H\"older estimate, pairs at distance less than
$\tau_\zeta^\alpha/4$ lie in a common chart where
$d_{\mathbb T}(\cdot,L_m)\geq\tau_\zeta^\alpha/2$.
The local estimate applies there; for the remaining pairs,
divide twice the supremum bound by $(\tau_\zeta^\alpha/4)^\beta$.
The displayed estimates verify Convention~3.15.
Taking exponents $(3/2,1/2)$ in Convention~4.5 requires
$1/4>2\alpha$ and $(1-\alpha)/2>2\alpha$, which follow
from \eqref{eq:s4-alpha}. Its fixed neck cutoff is chosen
before $m$.

By \eqref{eq:s4-exact-mismatch}, the constant mismatch is
$\tau_\zeta\zeta$. Thus Definition~3.17, with zero elevation
and tilt and with sheet corrections
$-\tau_\zeta\zeta\widehat v_m$ on both sides, gives precisely
$M_\zeta$. Since $k_m$ is supported in
$\delta\leq r_p\leq2\delta$ and $3\tau_\zeta^\alpha<\delta$,
its lift in the cited construction is $\mathcal J_\zeta$.

We apply the linear construction on invariant functions,
as in the proof of Theorem~6.17 of the cited paper.
In the torus step, solve $L_{\mathbb T}u^{\mathrm{base}}=E'$
using the invariant inverse above. Invariance gives a common
value $c$ and zero differential at every lattice point.
Therefore,
\begin{align*}
    u'&:=u^{\mathrm{base}}-c\widehat v_m,\qquad
    L_{\mathbb T}u'=E'-ck_m,\\
    u'(p)&=0,\qquad du'(p)=0
        \qquad\forall\,p\in L_m,
\end{align*}
which is the normalised solution required in equation~(4.13).
The neck inverses can be made equivariant by averaging their
conjugates over the neck stabiliser and then translating
over the lattice. Averaging does not increase their norms,
since the symmetries preserve the metrics and weights.
Reflection by $Z$ identifies the two substitute terms with
the same sign, so they have a single coefficient $c_E$.
Moreover, invariant functions have zero first angular
harmonics on every neck. The norm comparison in
Section~\ref{sec:s4-scalar-family} and Corollary~4.21
therefore give \eqref{eq:s4-scalar-inverse}.

To obtain the separate bound for $c_E$, we prove
\begin{align}
    \label{eq:s4-k-size}
    c\delta^{-2-\beta}\leq\|k_m\|_{C^{0,\beta}(\mathbb T)}
        \leq C\delta^{-2-\beta}.
\end{align}
Indeed, writing $\eta:=\eta_{2,1}$, the product rule and
$L_{\mathbb T}J_0(\sqrt2r)=0$ give on the cutoff annulus
\begin{align*}
    k_m(p+\delta x(\cos\theta,\sin\theta))
        =\delta^{-2}\bigl(\eta''(x)+x^{-1}\eta'(x)\bigr)+O(1),
        \qquad 1\leq x\leq2,
\end{align*}
with bounded derivatives of the remainder in $x$.
The leading coefficient is not identically zero: otherwise
$(x\eta')'=0$, which is incompatible with a smooth cutoff
equal to one near $x=1$ and zero near $x=2$.
Comparing a point where this coefficient is nonzero with a
point where $k_m=0$, at distance comparable to $\delta$,
gives a H\"older quotient bounded below by
$c\delta^{-2-\beta}$. Scaling gives the upper bound.
Dividing \eqref{eq:s4-scalar-inverse} by this lower bound
proves \eqref{eq:s4-coefficient-bound}.

We choose the operators continuously by retaining the
normalisations in equations~(4.13) and~(4.16) of the cited
construction. The torus solutions have zero value and
differential at $L_m$. On each neck, the constant and first
angular modes have zero value and axial derivative at the
waist; the higher modes are obtained from the inverse of
$\Delta_{\mathbb S^2}+2$ on angular orders at least two,
where its kernel is absent. For fixed $m$, pullback by
$F_\zeta$ makes the coefficients and cutoffs depend smoothly
on $\zeta$ on fixed domains. The normalised ODE solutions
and elliptic inverses are consequently continuous from
$C^{0,\beta'}$ to $C^{2,\beta'}$ for $0<\beta'<\beta$.
This proves continuity of each finite stage of the correction
series. The contraction estimate in Proposition~4.18 bounds
the tail after $N$ terms by $C_m2^{-N}\|E\|_{Y_\zeta}$,
uniformly in $\zeta$, in $X_\zeta\times\mathbb R$.
Given the sequence in the statement, first let its index
tend to infinity in a fixed finite sum, and then let
$N\to\infty$. Uniform boundedness of the source norms
controls the tails and proves the asserted continuity.

For the initial error, Lemma~4.6 applies with both substitute
terms equal to $-\tau_\zeta\zeta k_m$. It gives
\eqref{eq:s4-H-error}. Since
$(u_\zeta,c_{H,\zeta})=-\mathscr R_\zeta D_\zeta$, linearity
and \eqref{eq:s4-scalar-inverse} yield
\begin{align*}
    \mathcal L_\zeta u_\zeta
        &=-D_\zeta+c_{H,\zeta}\mathcal J_\zeta\\
        &=-H_\zeta+(-\tau_\zeta\zeta+c_{H,\zeta})\mathcal J_\zeta,
\end{align*}
which is \eqref{eq:s4-linear-cancellation}.
For $(v,\zeta)\in\mathcal B_m$, the norm comparison under
$F_\zeta$ and $e^{-1}\leq\tau_\zeta/\tau_0\leq e$ give
\begin{align*}
    \|\phi\|_{X_\zeta}
        +|c_{H,\zeta}|\|k_m\|_{C^{0,\beta}(\mathbb T)}
        &\leq C\delta^{-4-2\beta}\tau_\zeta^{1+\alpha/3}
            +C\tau_\zeta^{1+\alpha}
         =o(\tau_\zeta^{1+\alpha/4}).
\end{align*}
This proves \eqref{eq:s4-full-ball-bound} and, after the
fixed norm comparison, verifies the smallness hypothesis
of Lemma~5.1. Hence $G_{\zeta,t\phi}$ is well defined and
embedded for every $t\in[0,1]$.

For completeness, the factor in the quadratic estimate
can be read directly from the weights. Fix $x\in M_\zeta$
and put $r:=\rho_\zeta(x)$. After dividing lengths by $r$,
the displacement has local $C^{2,\beta}$ norm at most
$Cf_\zeta(x)\|\phi\|_{X_\zeta}/r$.
Mean curvature is a smooth function of the first two
derivatives of the graph. Taylor's formula, with its
constant and linear terms subtracted, therefore gives
\begin{align*}
    \frac{r^2}{f_\zeta(x)}
    \|Q_\zeta(\phi)\|_{C^{0,\beta}(B_{g_\zeta}(x,r/100),r^{-2}g_\zeta)}
        \leq C\frac{f_\zeta(x)}r\|\phi\|_{X_\zeta}^2.
\end{align*}
Here the unweighted remainder has size
$Cr^{-3}f_\zeta(x)^2\|\phi\|_{X_\zeta}^2$, because mean
curvature scales by $r^{-1}$. Since
$f_\zeta/\rho_\zeta\leq C\tau_\zeta^{-\alpha/2}$, taking
the supremum proves \eqref{eq:s4-scalar-quadratic}.
Finally, \eqref{eq:s4-scalar-two-jet} and its local
H\"older estimates follow from the definition of $X_\zeta$.

\section{Estimates for normal deformations}
\label{app:s4-normal}
In this appendix, we prove the geometric comparisons and nonlinear estimates deferred from Section~\ref{sec:s4-transverse}, using its notation and conventions.

\smallskip
\noindent\textit{Proof of Lemma~\ref{lem:s4-full-family-geometry}.}
By Proposition~\ref{prop:s4-scalar-input}-(3),
$\|\phi\|_{X_\zeta}\leq\tau_\zeta^{1+\alpha/4}$.
The definition of $f_\zeta$ and
$\tau_\zeta\leq\rho_\zeta\leq\pi/m$ imply
\begin{align}
    \label{eq:s4-dimensionless-phi}
    \sup_{M_\zeta}\frac{\tau_\zeta^{1+\alpha/4}f_\zeta}{\rho_\zeta}
        &\leq C\bigl(\tau_\zeta^{1-\alpha/4}
            +\tau_\zeta^{1+\alpha/4}m^{-1/2}\bigr)=o_m(1),\\
    \label{eq:s4-sharp-phi}
    \sup_{M_\zeta}\frac{\tau_\zeta^{1+\alpha/4}f_\zeta}{\tau_\zeta}
        &\leq C\tau_\zeta^{\alpha/4}m^{-3/2}=o_m(1).
\end{align}
For the first inequality, divide the two terms defining
$f_\zeta$ by $\rho_\zeta$ and maximise each separately.
For the second, both terms increase with $\rho_\zeta$,
and their ratio at $\rho_\zeta=\pi/m$ is
$(m/\pi)\tau_\zeta^{(1-\alpha)/2}=o_m(1)$.

Fix $x\in M_\zeta$ and set $r:=\rho_\zeta(x)$.
Use the charts of Lemma~\ref{lem:s4-initial-geometry},
translate the ambient origin to $\iota_\zeta(x)$, and
divide lengths by $r$. The initial immersion and its
normal exponential map have uniform $C^{3,\beta}$ bounds.
The local $C^{2,\beta}$ norm of $\phi/r$ is at most
$C\tau_\zeta^{1+\alpha/4}f_\zeta(x)/r=o_m(1)$.
Applying the chain rule to the exponential map gives
the same bound for the change in the rescaled immersion.

To pass to the geometric quantities, use the coordinate
identities
\begin{align*}
    (g_\phi)_{ij}&=\langle\partial_i\iota,\partial_j\iota\rangle,
    & (h_\iota^\nu)_{ij}&=\langle\partial_i\partial_j\iota,\nu\rangle,
    & (S_\iota^\nu)^i{}_j
        &=\sum_{k=1}^2(g_\phi)^{ik}(h_\iota^\nu)_{kj}.
\end{align*}
The unit normal depends smoothly on the position and first
derivatives, and the rescaled metric stays positive definite.
Thus the metric changes in $C^{1,\beta}$, and the normal
and rescaled shape operator in $C^{0,\beta}$, by at most
$C\tau_\zeta^{1+\alpha/4}f_\zeta(x)/r$.
Restoring the scale and taking suprema gives
\begin{align*}
    \|g_\phi-g_\zeta\|_{C^{1,\beta}_\zeta(\rho_\zeta^2)}
        &\leq C\sup_{M_\zeta}
            \frac{\tau_\zeta^{1+\alpha/4}f_\zeta}{\rho_\zeta}=o_m(1),\\
    \|\nu-\nu_\zeta\|_{C^{0,\beta}_\zeta(f_\zeta/\rho_\zeta)}
        &\leq C\tau_\zeta^{1+\alpha/4},\\
    \|S_\iota^\nu-S_{\jmath_\zeta}^{\nu_\zeta}
        \|_{C^{0,\beta}_\zeta(f_\zeta/\rho_\zeta^2)}
        &\leq C\tau_\zeta^{1+\alpha/4}.
\end{align*}
The first estimate gives $C^{-1}g_\zeta\leq g_\phi\leq Cg_\zeta$.
By \eqref{eq:s4-sharp-phi}, the other two changes are
$o_m(1)\tau_\zeta/\rho_\zeta$ and
$o_m(1)\tau_\zeta/\rho_\zeta^2$, respectively, with the
corresponding local H\"older bounds. The formulas in
Appendix~\ref{app:s4-initial-geometry} give the initial
estimates also with these H\"older bounds. Combining them
with the smoothness of $(E_1,E_2,E_3)$ proves the asserted
bounds for $\nu_h$ and $S_\iota^\nu$.
The same composition estimate applied to the height gives
\begin{align*}
    \|z\circ G\|_{C^{0,\beta}_\zeta(1)}
        \leq C\left(\frac{m^2\tau_\zeta}{2\pi}
            +\tau_\zeta^{1+\alpha/4}\sup_{M_\zeta}f_\zeta\right)
        \leq C\frac{m^2\tau_\zeta}{2\pi}.
\end{align*}
Since $S(E_3)=0$, \eqref{eq:s4-normal-identities} gives
$H_{\iota,n}=-\langle S(\nu_h),\nu_h\rangle$ and
$\omega(Y)=\langle S(dG(Y)),\nu_h\rangle$.
The H\"older product estimate therefore gives the bounds
for $H_{\iota,n}$ and $\omega$, while boundedness of $S$
gives the bound for $h_\iota^n$.
The components of $\nabla^{\mathcal N}S$ are smooth multiples
of $\tan z$, as computed in Appendix~\ref{app:s4-geometric-identities}.
Consequently, \eqref{eq:s4-omega-derivative} yields
\begin{align*}
    \|\nabla^{g_\phi}\omega\|_{C^{0,\beta}_\zeta(1)}
        \leq C\left(\frac{m^2\tau_\zeta}{2\pi}
            \sup_{M_\zeta}\rho_\zeta^2+\tau_\zeta\right)
        \leq C\tau_\zeta.
\end{align*}
The formula $q=2+|S_\iota^n|_{g_\phi}^2-|\omega|_{g_\phi}^2$
and $\tau_\zeta/\rho_\zeta\leq1$ give the remaining bound
for $q$. All these expressions use at most two derivatives
of $G$, and hence of $\phi$.

Finally, the normal geodesic from $x$ to $G(x)$ has length
$|\phi(x)|$. Since $|d\pi|\leq\sec z_0$ and distance to
$L_m$ is $1$-Lipschitz on $\mathbb T$,
\begin{align*}
    |d_{\mathbb T}(\pi(G(x)),L_m)-\rho_\zeta(x)|
        \leq\sec z_0|\phi(x)|=o_m(1)\rho_\zeta(x).
\end{align*}
The last equality follows from \eqref{eq:s4-dimensionless-phi}
and proves \eqref{eq:s4-deformed-scale}.

\smallskip
\noindent\textit{Proof of Lemma~\ref{lem:s4-metric-comparison}.}
We first compare $g_\zeta$ and $g_*$ on the upper half of
a cell. Where $\tau_\zeta e^\xi\leq r_c/2$, the height is
$\tau_\zeta\xi$ and $r_{m,\zeta}=\tau_\zeta\cosh\xi$.
Substitution in \eqref{eq:s4-suspension-metric} gives
\begin{align*}
    r_{m,\zeta}^{-2}\Theta_{m,\zeta,p}^*(g_\zeta-g_*)
        =(\cos^2(\tau_\zeta\xi)-1)
            (\tanh^2\xi\,d\xi^2+d\theta^2).
\end{align*}
The height is bounded by $Cm^2\tau_\zeta/(2\pi)$.
Since its first $\xi$ derivative is $\tau_\zeta$ and its
second derivative vanishes, the coefficients in this
identity and their first two derivatives are bounded by
$C(m^2\tau_\zeta/(2\pi))^2$.

On $r_c/2\leq\tau_\zeta e^\xi\leq r_c$, the radius
$r_{m,\zeta}$ is comparable to $r_c$. Differentiating
\eqref{eq:s4-model-radius} gives
\begin{align*}
    \left|\partial_\xi^j
        \left(\frac{r_{m,\zeta}'}{r_{m,\zeta}}-1\right)\right|
        \leq C\frac{\tau_\zeta^2}{r_c^2}
        \qquad\forall\,j\in\{0,1,2\}.
\end{align*}
Writing $z=C_{\tau_\zeta}(r_{m,\zeta}(\xi))$ in these
coordinates, we have $|\partial_\xi^jz|\leq C\tau_\zeta$
for $j\in\{1,2,3\}$. Moreover,
\begin{align*}
    r_{m,\zeta}^{-2}\Theta_{m,\zeta,p}^*(g_\zeta-g_*)
        ={}&\left[\cos^2z
            \left(\frac{r_{m,\zeta}'}{r_{m,\zeta}}\right)^2
            +\left(\frac{\partial_\xi z}{r_{m,\zeta}}\right)^2-1\right]d\xi^2\\
        &+(\cos^2z-1)d\theta^2.
\end{align*}
Thus its coefficients and their first two derivatives are
bounded by $C((m^2\tau_\zeta/(2\pi))^2+\tau_\zeta^2/r_c^2)$.
Where $\tau_\zeta e^\xi\geq r_c$, the height graph
coordinates instead give
\begin{align*}
    g_*=g_{\mathbb T},\qquad
    g_\zeta-g_*=(\cos^2h_\zeta-1)g_{\mathbb T}
        +dh_\zeta\otimes dh_\zeta.
\end{align*}
Here $d_{\mathbb T}(q,L_m)\geq r_c/2$, and the initial
height estimates give $|h_\zeta|\leq Cm^2\tau_\zeta/(2\pi)$ and
\begin{align*}
    d_{\mathbb T}(q,L_m)^j|D^jh_\zeta(q)|
        \leq C_j\tau_\zeta\qquad\forall\,j\in\{1,2,3\}.
\end{align*}
Differentiating the last metric formula
therefore gives the same bound in rescaled coordinates.
The second derivative bounds control the first derivative
H\"older seminorms. Since $r_c=\tau_\zeta^\alpha$, these
three comparisons and $\Gamma_m$-invariance prove the first
inequality in \eqref{eq:s4-metric-comparison}.

The rescaled metrics are uniformly positive definite and
close in $C^{1,\beta}$. Their distances are therefore
uniformly comparable, and their Christoffel symbols differ
by a small quantity in $C^{0,\beta}$. For a scalar function,
the precise change of Hessian is
\begin{align*}
    \bigl((\nabla^{g_\zeta})^2b-(\nabla^{g_*})^2b\bigr)_{ij}
        =-\sum_{k=1}^2
        \bigl(\Gamma^k_{ij}(g_\zeta)-\Gamma^k_{ij}(g_*)\bigr)\partial_kb.
\end{align*}
This identity compares the local $C^{2,\beta}$ norms.
Covering the metric balls by uniformly many smaller
coordinate disks gives the asserted weighted norm equivalence.
The same connection comparison applies to the first
derivatives of tensor fields. Thus the preceding proof's
estimate for $g_\phi-g_\zeta$ also holds in the norm based
on $g_*$. The triangle inequality and
\eqref{eq:s4-dimensionless-phi} prove the second metric estimate.

Finally, fix a chart at scale $r=\rho_\zeta(x)$.
The coordinate formula
\begin{align*}
    \Delta_g b=\sum_{i,j=1}^2g^{ij}\partial_i\partial_jb
        -\sum_{i,j,k=1}^2g^{ij}\Gamma^k_{ij}(g)\partial_kb
\end{align*}
shows that every coefficient of
$r^2(\Delta_{g_\phi}-\Delta_{g_*})$ is bounded in
$C^{0,\beta}$ by $C\eta_m(\zeta)$ in the rescaled chart.
The H\"older product estimate, followed by the supremum
over $x$, proves \eqref{eq:s4-Laplacian-comparison}.

\smallskip
\noindent\textit{Proof of Lemma~\ref{lem:s4-transverse-quadratic}.}
For $\eta_b\in\{n_b,\nu_b\}$, the coordinate formula for
the tested mean curvature is
\begin{align}
    \label{eq:s4-tested-curvature-formula}
    \langle\mathbf H_{X_b},\eta_b\rangle
        =\sum_{i,j=1}^2(g_b)^{ij}
            \langle\partial_i\partial_jX_b,\eta_b\rangle.
\end{align}
Indeed, the derivatives on the right are taken in
$\mathbb R^5$, and both the tangential Christoffel terms
and the radial spherical term are orthogonal to $\eta_b$.
The formulas for $X_b$, the normal projections and their
normalisations show that this expression depends smoothly
on the first two derivatives of $\iota,n,b$ and on the
values of $\nu$. Since $n=N\circ\pi\circ G$, and $\nu$
depends only on $G$ and its first derivatives, it uses
at most two derivatives of $\phi$.

Fix $x\in M_\zeta$ and set $r:=\rho_\zeta(x)$.
In the rescaled charts, $r^{-1}(\iota-\iota(x))$ has a
uniform $C^{2,\beta}$ bound and uniformly positive
definite induced metric. The fields $\iota,n$ also have
uniform $C^{2,\beta}$ bounds there, while $\nu$ has a
uniform $C^{1,\beta}$ bound. Put $a:=b/r$. Then
\begin{align*}
    \frac{X_b-\iota(x)}r
        =\frac{\iota-\iota(x)}r
            +\frac{\cos(ra)-1}r\iota+\frac{\sin(ra)}rn.
\end{align*}
The last two coefficients extend smoothly to $r=0$, with
uniform derivatives for bounded $a$. Moreover,
$\|b\|_{\mathcal X_\zeta}\leq c_0\tau_\zeta$ and
$r\geq\tau_\zeta$ bound the local $C^{2,\beta}$ norm of
$a$ by $Cc_0$. For small $c_0$, the rescaled graph metrics
remain uniformly positive definite and
$|U_b|,|W_b|\geq1/2$, by \eqref{eq:s4-testing-lengths}.
Thus inversion of the metric and normalisation of the
projections have uniformly bounded derivatives.
Consequently, the map
$a\mapsto r\langle\mathbf H_{X_{ra}},\eta_{ra}\rangle$
has uniformly bounded second derivative from
$C^{2,\beta}$ to $C^{0,\beta}$.

Each differentiation with respect to $b$ contributes
$r^{-1}$ because $a=b/r$, and restoring mean curvature
from its rescaled value contributes another $r^{-1}$.
The weight in $\mathcal Z_\zeta$ contributes $r^2$.
Since $r^{-1}\leq\tau_\zeta^{-1}$, we obtain
\begin{align*}
    \|D^2\mathscr P_{\zeta,\phi}(b)[h,k]\|_{\mathcal Y_\zeta}+\|D^2\mathscr T_{\zeta,\phi}(b)[h,k]\|_{\mathcal Z_\zeta}\leq C\tau_\zeta^{-1}
        \|h\|_{\mathcal X_\zeta}\|k\|_{\mathcal X_\zeta}
        \qquad\forall\,h,k\in\mathcal X_\zeta.
\end{align*}
Since $DQ_\eta(0)=0$, integration along $t\mapsto tb$ gives
$\|DQ_\eta(b)h\|_{C^{0,\beta}_\zeta(\rho_\zeta^{-2})}
\leq C\tau_\zeta^{-1}\|b\|_{\mathcal X_\zeta}\|h\|_{\mathcal X_\zeta}$.
Applying this bound to
\begin{align*}
    Q_\eta(b)-Q_\eta(\widetilde b)
        =\int_0^1DQ_\eta(\widetilde b+t(b-\widetilde b))
            [b-\widetilde b]\,dt
\end{align*}
proves \eqref{eq:s4-transverse-quadratic}. Taking
$\widetilde b=0$ and using $Q_\eta(0)=0$ gives its final
assertion.

\section{Coordinate and curvature calculations}
\label{app:s4-curvature}
\label{app:s4-outer-coordinates}
\label{app:s4-exact-curvature}
In this appendix, we give the details of the coordinate and curvature computations used in Section~\ref{sec:s4-negative-curvature}. Recall that $D$ is the Levi--Civita connection of $g_{\mathbb T}$, and all derivative norms on the torus are computed using this metric.

\smallskip
\noindent\textit{Proof of Lemma~\ref{lem:s4-outer-coordinates}.}
Fix $R$ sufficiently large. The initial height satisfies
\begin{align}
    \label{eq:s4-initial-flat-derivatives}
    |D^kh_{\zeta_m}(q)|\leq C_k\tau_m r_m(q)^{-k}
        \qquad\forall\,k\geq1,\quad r_m(q)\geq R\tau_m/2.
\end{align}
On a catenoidal piece this follows by differentiating
$\tau_m\operatorname{arcosh}(r/\tau_m)$ for
$r\geq R\tau_m/2$. On the remaining graphs it follows
from the Green estimates and \eqref{eq:s4-matching-estimate},
as in Appendix~\ref{app:s4-initial-geometry}.

Fix an upper-sheet point as in the statement, let $q_0$
be its initial projection onto $\mathbb T$, and abbreviate
$r:=r_m(q_0)$. For a sufficiently small fixed $c_1>0$,
the disk $B_{\mathbb T}(q_0,4c_1r)$ lies in
$\{r_m>R\tau_m/2\}$, and $r_m$ is comparable to $r$
there. Write
\begin{align*}
    X_m^0(q):=\mathcal F(q,h_{\zeta_m}(q)),\qquad
    A:=m^{5/2+2\beta}\tau_m^{1+\alpha/3}.
\end{align*}
In the coordinates $(q-q_0)/r$, the metric
$r^{-2}(X_m^0)^*g_{\mathcal N}$ is uniformly positive
definite and has bounded derivatives of each fixed order.
The same bounds hold for the initial immersion and its
normal exponential map after ambient lengths are divided
by $r$. Since $f_m\leq Cm^{-3/2}$,
\eqref{eq:s4-final-two-jets} gives
\begin{align*}
    \sum_{j=0}^2r^j|D^j(\phi_m\circ X_m^0)|\leq CA,
    \qquad
    \sum_{j=0}^2r^j|D^j(b_m\circ X_m^0)|\leq C\tau_m^2.
\end{align*}
Here the conversion from covariant to coordinate derivatives
uses the preceding metric bounds. Moreover,
$A/\tau_m=m^{5/2+2\beta}\tau_m^{\alpha/3}\to0$.

Express the first deformation in product coordinates as
\begin{align*}
    G_m\circ X_m^0(q)
        =\mathcal F(Q_m(q),H_m(q)),\qquad
    Q_m(q)=q+e_m(q),\qquad H_m(q)=h_{\zeta_m}(q)+a_m(q).
\end{align*}
At zero displacement the normal exponential map equals
$X_m^0$. Taylor's formula in the displacement, followed
by two differentiations in $q$, therefore gives
\begin{align}
    \label{eq:s4-outer-projection-estimates}
    |D^je_m|+|D^ja_m|\leq CA r^{-j}
        \qquad\forall\,j\in\{0,1,2\}.
\end{align}
Indeed, in the rescaled coordinates these differences have
$C^2$ norm at most $C(A/r+(A/r)^2)\leq CA/r$;
multiplication by $r$ restores the stated estimates.

Since $DQ_m=I+De_m$ and $|De_m|\leq CA/r\to0$,
$Q_m$ is injective on a smaller disk and has a smooth inverse.
Differentiating $Q_m\circ Q_m^{-1}=\operatorname{id}$ gives
\begin{align*}
    DQ_m^{-1}&=(DQ_m)^{-1}\circ Q_m^{-1},\\
    D^2Q_m^{-1}[u,v]
        &=-DQ_m^{-1}\bigl((D^2Q_m\circ Q_m^{-1})
            [DQ_m^{-1}u,DQ_m^{-1}v]\bigr).
\end{align*}
Thus $|DQ_m^{-1}|\leq2$ and $|D^2Q_m^{-1}|\leq CAr^{-2}$.
Also $|Q_m(q)-q|\leq CA=o(r)$. As $r\geq2R\tau_m$,
the smaller disk can be chosen so that its image $U$
lies in $\{r_m>R\tau_m\}$ and $r_m$ is comparable to
$r$ on both disks.

Define
\begin{align*}
    z_m:=H_m\circ Q_m^{-1},\qquad
    \widetilde b_m:=b_m\circ X_m^0\circ Q_m^{-1}.
\end{align*}
The fundamental theorem of calculus gives
\begin{align*}
    ((z_m-h_{\zeta_m})\circ Q_m)(q)
        =a_m(q)-\int_0^1
            Dh_{\zeta_m}(q+t e_m(q))\cdot e_m(q)\,dt.
\end{align*}
Differentiating twice uses at most three derivatives of
$h_{\zeta_m}$ and two of $e_m$. By
\eqref{eq:s4-initial-flat-derivatives} and
\eqref{eq:s4-outer-projection-estimates}, the sum of the
absolute values of these derivatives, weighted by $1,r,r^2$,
is at most $C\tau_m A/r\leq CA$ for the integral term.
Consequently,
\begin{align*}
    \sum_{j=0}^2r^j
        |D^j((z_m-h_{\zeta_m})\circ Q_m)|\leq CA.
\end{align*}
Composition with $Q_m^{-1}$ proves the first estimate in
\eqref{eq:s4-outer-coordinate-estimates}; the same chain
rule applied to $b_m\circ X_m^0$ proves the second.
The identity $\pi(G_m\circ X_m^0)=Q_m$ also gives the
parametrisation of the second deformation in
\eqref{eq:s4-outer-graph-form}. Finally,
$|h_{\zeta_m}|\leq Cm^2\tau_m$,
\eqref{eq:s4-initial-flat-derivatives}, and $A=o(\tau_m)$
give \eqref{eq:s4-outer-height-bounds}. Reflection by $Z$
gives the lower-sheet statement.

\smallskip
\noindent\textit{Proof of Lemma~\ref{lem:s4-exact-height-curvature}.}
For this calculation abbreviate
\begin{align*}
    c:=\cos z,\qquad s_z:=\sin z,\qquad p:=Dz,
    \qquad H:=D^2z,\qquad Q:=c^2+|p|^2.
\end{align*}
Differentiating $X(q)=\mathcal F(q,z(q))$ gives tangent
vectors $T_i=cE_i+p_iE_3$, for $i\in\{1,2\}$. The induced
metric and upward unit normal are therefore
\begin{align*}
    (g_X)_{ij}=c^2\delta_{ij}+p_ip_j,\qquad
    \det g_X=c^2Q,\qquad
    \nu^+=\frac{-p_1E_1-p_2E_2+cE_3}{\sqrt Q}.
\end{align*}
The connection formulas for $(E_1,E_2,E_3)$ give
\begin{align*}
    \nabla^{\mathcal N}_{T_i}T_j
        &=-s_z(p_iE_j+p_jE_i)+(H_{ij}+cs_z\delta_{ij})E_3,\\
    \langle\nabla^{\mathcal N}_{T_i}T_j,\nu^+\rangle
        &=\frac c{\sqrt Q}
            (H_{ij}+cs_z\delta_{ij}+2\tan z\,p_ip_j).
\end{align*}
Set $B:=H+cs_zI+2\tan z\,p\otimes p$. The determinant
of this scalar second fundamental form divided by
$\det g_X$ equals $\det B/Q^2$. Moreover,
$S(T_1)=-E_1$ and $S(T_2)=E_2$, so the matrix of
$\langle S(T_i),T_j\rangle$ is $\operatorname{diag}(-c,c)$.
The Gauss equation for $\mathcal N\subset\mathbb S^4$
gives sectional curvature $1-Q^{-1}$ on the graph tangent
plane. Applying the Gauss equation to the graph gives
\begin{align*}
    Q^2K_X=Q^2-Q+\det B=Q(|p|^2-s_z^2)+\det B.
\end{align*}
For a symmetric $2\times2$ matrix $A$, its cofactor matrix
is $\operatorname{cof}(A)=(\operatorname{tr}A)I-A$, and
\begin{align*}
    \det(A+t\,p\otimes p)
        =\det A+t\,p^T\operatorname{cof}(A)p.
\end{align*}
Taking $A=H+cs_zI$ gives
\begin{align*}
    \det B={}&\det H+cs_z\operatorname{tr}H+c^2s_z^2
        +2\tan z\,p^T\operatorname{cof}(H)p+2s_z^2|p|^2.
\end{align*}
Since $Q=c^2+|p|^2$ and $c^2+s_z^2=1$,
$Q(|p|^2-s_z^2)+c^2s_z^2+2s_z^2|p|^2=|p|^2+|p|^4$.
Substitution, followed by
$p^T\operatorname{cof}(H)p=(\operatorname{tr}H)|p|^2-H(p,p)$,
proves \eqref{eq:s4-exact-height-curvature}.

For the error estimate, the assumed bounds and
$r_m\leq\pi/m$ give
\begin{align*}
    |z|\leq Cm^2\tau_m\leq C\tau_m r_m^{-2},\qquad
    |p|\leq C\tau_m r_m^{-1},\qquad
    |H|\leq C\tau_m r_m^{-2}.
\end{align*}
Consequently, $|\mathcal K_2(z)|\leq C\tau_m^2r_m^{-4}$.
Also $Q\geq\cos^2z_0$ and
$|Q^{-2}-1|\leq C(z^2+|p|^2)$. The remaining terms satisfy
\begin{align*}
    |(cs_z-z)\operatorname{tr}H|
        &\leq C|z|^3|H|\leq Cz^2\tau_m^2r_m^{-4},\\
    |p|^4&\leq C|p|^2\tau_m^2r_m^{-4},\\
    |2\tan z\,p^T\operatorname{cof}(H)p|
        &\leq C|z||H||p|^2
            \leq C|p|^2\tau_m^2r_m^{-4}.
\end{align*}
Subtracting $\mathcal K_2(z)$ from the exact formula divided
by $Q^2$ therefore gives
\begin{align*}
    |K_X-\mathcal K_2(z)|
        \leq C(z^2+|p|^2)\tau_m^2r_m^{-4}
        \leq C(m^4\tau_m^2+R^{-2})\tau_m^2r_m^{-4},
\end{align*}
where the last inequality uses $r_m\geq R\tau_m$.
\let\c\textcedilla
\bibliographystyle{amsalpha}
\bibliography{main}
\end{document}